%% file: analogs_dual_canonical.tex
\documentclass[oneside,english]{amsart}
\usepackage[latin9]{inputenc}
\usepackage{amstext}
\usepackage{amsthm}
\usepackage{amssymb}
\usepackage{stackrel}

\makeatletter
\numberwithin{equation}{section}
\numberwithin{figure}{section}

\input{LyxMacro.tex}

\usepackage[a4paper, margin=1in]{geometry}
\ifdefined\showcaptionsetup
 \PassOptionsToPackage{caption=false}{subfig}
\fi
\usepackage{subfig}
\makeatother

\usepackage{babel}
\begin{document}
%macro for this file

\newcommand{\W}{\mathbb{W}}
\newcommand{\V}{\mathbb{V}}
\newcommand{\txi}{\tilde{\xi}}

\newcommand{\tcZ}{\tilde{\cZ}}
\newcommand{\tz}{\tilde{z}}

\newcommand{\tTheta}{\widetilde{\Theta}}

\newcommand{\Kq}{\mathfrak{K}_{q}}
\newcommand{\hAq}{\hat{\mathcal{A}}_{q}(\mathfrak{n})}

\newcommand{\ee}{\mathbf e}
\newcommand{\ff}{\mathbf f}

\newcommand{\unitObj}{\mathbf 1}
\newcommand{\wtSp}{\opname Q}
\renewcommand{\op}{^{\opname{op}}}

\newcommand{\simpObj}{\opname{Simp}}

\newcommand{\Span}{\opname{Span}}

\renewcommand{\res}{\opname{res}}
\newcommand{\frR}{\mathfrak{R}}
\newcommand{\frI}{\mathfrak{I}}
\newcommand{\frz}{\mathfrak{f}}
\renewcommand{\diag}{\delta}
\newcommand{\symm}{\mathsf{D}}
\newcommand{\ubi}{\underline{\bi}}
\newcommand{\ubj}{\underline{\bj}}

\newcommand{\Br}{\mathsf{Br}}
\newcommand{\cB}{\mathcal{B}}
\newcommand{\xar}[1]{\xymatrix{\ar[r]^{#1}&}}
\newcommand{\xline}[1]{\xymatrix{\ar@{-}[r]^{#1}&}}
\newcommand{\ow}{\overrightarrow{w}}
\newcommand{\Conf}{\mathrm{Conf}}
\newcommand{\leftweave}{\overleftarrow{\mathfrak{m}}}

\newcommand{\ueta}{\underline{\eta}}
\newcommand{\uzeta}{\underline{\zeta}}
\newcommand{\uxi}{\underline{\xi}}

\newcommand{\qO}{\mathcal{O}_q}

\newcommand{\up}{\opname{up}}
\newcommand{\dCan}{\opname{B}^{*}}
\newcommand{\cdCan}{\opname{\mathring{B}}^{*}}
\newcommand{\hdCan}{\hat{\opname{B}^{*}}}

\renewcommand{\qClAlg}{{\clAlg_q}}

\newcommand{\cone}{{M^\circ}}
\newcommand{\uCone}{\underline{\cone}}
\newcommand{\yCone}{N_{\ufv}}
\newcommand{\tropSet}{{\mathcal{M}^\circ}}
\newcommand{\domTropSet}{{\overline{\mathcal{M}}^\circ}}

\newcommand{\sol}{\mathrm{TI}}
\newcommand{\intv}{{\mathrm{BI}}}

\newcommand{\Perm}{\mathrm{P}}

\newcommand{\tui}{\widetilde{\ui}}
\newcommand{\tuk}{\widetilde{\uk}}

\newcommand{\bideg}{\opname{bideg}}

\newcommand{\ubeta}{\underline{\beta}}
\newcommand{\udelta}{\underline{\delta}}
\newcommand{\ugamma}{\underline{\gamma}}
\newcommand{\ualpha}{\underline{\alpha}}

\newcommand{\LP}{{\mathcal{LP}}}
\newcommand{\hLP}{{\widehat{\mathcal{LP}}}}

\newcommand{\bLP}{{\overline{\mathcal{LP}}}}
\newcommand{\bClAlg}{{\overline{\clAlg}}}
\newcommand{\bUpClAlg}{{\overline{\upClAlg}}}
\newcommand{\base}{\mathbb{B}}
\newcommand{\alg}{\mathbf{A}}
\newcommand{\algfr}{\alg_{\mathrm{f}}}

\newcommand{\midClAlg}{{\clAlg^{\mathrm{mid}}}}
\newcommand{\upClAlg}{\mathcal{U}}

\newcommand{\bQClAlg}{{\bClAlg_q}}

\newcommand{\qUpClAlg}{{\upClAlg_q}}

\newcommand{\bQUpClAlg}{{\bUpClAlg_q}}

\newcommand{\canClAlg}{{\clAlg^{\mathrm{can}}}}

\newcommand{\AVar}{\mathscr{V}}
\newcommand{\XVar}{\mathbb{X}}

\newcommand{\Jac}{\hat{\mathop{J}}}

\newcommand{\wt}{\mathrm{wt}}
\newcommand{\cl}{\mathrm{cl}}

\newcommand{\domCone}{{\overline{M}^\circ}}

\newcommand{\tree}{{\mathbb{T}}}

\newcommand{\img}{{\mathrm{Im}}}

\newcommand{\Id}{{\mathrm{Id}}}
\newcommand{\prin}{{\mathrm{prin}}}

\newcommand{\mm}{{\mathbf{m}}}

\newcommand{\cPtSet}{{\mathcal{CPT}}}
\newcommand{\bPtSet}{{\mathcal{BPT}}}
\newcommand{\tCPtSet}{{\widetilde{\mathcal{CPT}}}}

\newcommand{\frn}{{\mathfrak{n}}}
\newcommand{\frsl}{{\mathfrak{sl}}}

\newcommand{\frg}{\mathfrak{g}}
\newcommand{\hfrg}{\hat{\mathfrak{g}}}
\newcommand{\hfrh}{\hat{\mathfrak{h}}}

\newcommand{\frh}{\mathfrak{h}}
\newcommand{\frp}{\mathfrak{p}}
\newcommand{\frd}{\mathfrak{d}}
\newcommand{\frj}{\mathfrak{j}}
\newcommand{\frD}{\mathfrak{D}}
\newcommand{\frS}{\mathfrak{S}}
\newcommand{\frC}{\mathfrak{C}}
\newcommand{\frM}{\mathfrak{M}}
%%%%%%%%%

\newcommand{\Int}{\mathrm{Int}}
\newcommand{\ess}{\mathrm{ess}}
\newcommand{\Mono}{\mathrm{Mono}}

\newcommand{\bfm}{{\mathbf{m}}}
\newcommand{\bfI}{{\mathbf{I}}}

\newcommand{\Pot}{\mathrm{Pot}}
\newcommand{\kGp}{\opname{K}}

\newcommand{\s}{\mathrm{s}}
\newcommand{\fd}{\mathrm{fd}}

\renewcommand{\sc}{\mathrm{sc}}%semi-classical
\newcommand{\Hall}{\mathrm{Hall}}
\newcommand{\income}{\mathrm{in}}

\newcommand{\tn}{\tilde{n}}

\newcommand{\sing}{\opname{sing}}

\renewcommand{\inj}{\opname{inj}}

\renewcommand{\ext}{\mathrm{ext}}

\newcommand{\cJac}{\Jac}%all jacobian algebras are completed

\newcommand{\stilt}{\mathrm{s}\tau\mathrm{-tilt}}
\newcommand{\Fac}{\mathrm{Fac}}
\newcommand{\Sub}{\mathrm{Sub}}

\newcommand{\rigid}{\mathrm{rigid}}
\newcommand{\tauRigid}{\tau\mathrm{-rigid}}
\newcommand{\spTilt}{\mathrm{s}\tau\mathrm{-tilt}}
\newcommand{\clTilt}{\mathrm{c-tilt}}
\newcommand{\maxRigid}{\mathrm{m-rigid}}

\newcommand{\Li}{\mathrm{Li}}

\newcommand{\Trop}{\opname{Trop}}

\newcommand{\seq}{\boldsymbol{\mu}}%{\overleftarrow{\mu}}%{\mu} \newcommand{\seqnu}{\boldsymbol{\nu}}%{\overleftarrow{\nu}}%{\nu}

\newcommand{\bseq}{\mu_{\bullet}}

\newcommand{\val}{\mathbf{v}}
\newcommand{\hookuparrow}{\mathrel{\rotatebox[origin=c]{90}{$\hookrightarrow$}}} 
\newcommand{\hookdownarrow}{\mathrel{\rotatebox[origin=c]{-90}{$\hookrightarrow$}}}
\newcommand{\twoheaddownarrow}{\mathrel{\rotatebox[origin=c]{-90}{$\twoheadrightarrow$}}}

\newcommand{\rd}{{\opname{red}}} 
\newcommand{\bCan}{\overline{\can}}

\newcommand{\im}{\opname{Im}} 

\newcommand{\lex}{\opname{lex}} 

\newcommand{\vu}{\mathbf{u}}

\newcommand{\sd}{{\bf t}}
\newcommand{\dsd}{\ddot{\sd}} 
\newcommand{\rsd}{\dot{\sd}}
\newcommand{\ssd}{{\bf s}}
\newcommand{\ddB}{\ddot{B}}
\newcommand{\ddI}{\ddot{I}}
\newcommand{\dB}{\dot{B}}
\newcommand{\ddH}{\ddot{H}}
\newcommand{\ddLambda}{\ddot{\Lambda}}

\newcommand{\nsd}{{\bf r}}
\newcommand{\usd}{{\bf u}}

\newcommand{\bti}{{\mathbf{\tilde{i}}}}
\newcommand{\ubti}{{\underline{\bti}}}

\newcommand{\col}{\opname{col}}

\newcommand{\frRing}{\mathcal{R}}
\newcommand{\bFrRing}{\overline{\frRing}}
\newcommand{\frGroup}{\mathcal{P}}
\newcommand{\frMonoid}{\overline{\frGroup}}

\newcommand{\uBase}{\underline{\base}}
\newcommand{\bBase}{\overline{\base}}

\newcommand{\bUBase}{\overline{\uBase}}

\newcommand{\ddBS}{\ddot{X}}
\newcommand{\dBS}{\dot{X}}

\newcommand{\udim}{\mathrm{dim}}

\newcommand{\ev}{\mathrm{ev}}
\newcommand{\CC}{\mathrm{CC}}

\newcommand{\envAlg}{\mathsf{U}_q}
\newcommand{\qAff}{\mathsf{U}_{\varepsilon}(\hat{\mathfrak{g}})}

\newcommand{\BZ}{\mathrm{BZ}}

\newcommand{\HL}{\mathrm{HL}}
\DeclarePairedDelimiter\floor{\lfloor}{\rfloor}
\newcommand{\simeqd}{\mathrel{\rotatebox[origin=c]{-90}{$\xrightarrow{\sim}$}}}
\newcommand{\simequ}{\mathrel{\rotatebox[origin=c]{90}{$\xrightarrow{\sim}$}}}

\newcommand{\ucN}{\underline{\mathcal{N}}}
\newcommand{\sqbinom}[2]{\genfrac{[}{]}{0pt}{}{#1}{#2}}
\title[]{Analogs of the dual canonical bases for cluster algebras from Lie
theory}
\author{Fan Qin}
\email{qin.fan.math@gmail.com}
\begin{abstract}
We construct common triangular bases for almost all the known (quantum)
cluster algebras from Lie theory. These bases provide analogs of the
dual canonical bases, long anticipated in cluster theory. In cases
where the generalized Cartan matrices are symmetric, we show that
these cluster algebras and their bases are quasi-categorified.

We base our approach on the combinatorial similarities among cluster
algebras from Lie theory. For this purpose, we introduce new cluster
operations to propagate structures across different cases, which allow
us to extend established results on quantum unipotent subgroups to
other such algebras.

We also obtain fruitful byproducts. First, we prove $A=U$ for these
quantum cluster algebras. Additionally, we discover rich structures
of the locally compactified quantum cluster algebras arising from
double Bott-Samelson cells, including $T$-systems, standard bases,
and Kazhdan-Lusztig type algorithms. Notably, in type $ADE$, we
obtain their monoidal categorifications via monoidal categories associated
with positive braids. As a special case, these categories provide
monoidal categorifications of the quantum function algebras in type
$ADE$.
\end{abstract}

\maketitle
\tableofcontents{}

\section{Introduction\label{sec:intro} }

\subsection{Background}

Cluster algebras were invented by Fomin and Zelevinsky \cite{fomin2002cluster}.
They are algebras with certain combinatorial structures. Particularly,
they have distinguished elements called cluster monomials, which are
recursively defined by combinatorial algorithms called mutations.
They also have natural quantization \cite{BerensteinZelevinsky05}.
Originally, Fomin and Zelevinsky introduced them with the following
ambitious expectations:
\begin{enumerate}
\item The (quantized) coordinate rings of many interesting varieties from
Lie theory are isomorphic to appropriate cluster algebras.
\item These (quantized) coordinate rings possess distinguished bases, which
are analogs of the dual canonical bases \cite{Lusztig90,Lusztig91}\cite{Kas:crystal}
and contain the cluster monomials.
\end{enumerate}
An ever-expanding list of varieties have been found to meet the first
expectation. But the progress on the second expectation has been slow.
Previously, it was shown that, for the quantized coordinate rings
$\qO[N(w)]$ of unipotent subgroups $N(w)$ viewed as cluster algebras,
the dual canonical bases $\dCan$ contain the cluster monomials up
to scalar multiples, see \cite{qin2017triangular,qin2020dual} or
\cite{Kang2018}\cite{mcnamara2021cluster}.

We aim to show that the second expectation is met by almost all the
known cluster algebras arising from Lie theory. Particularly, we will
greatly expand known results from unipotent subgroups to other interesting
varieties.

\subsection{Main results}

Let $\alg$ denote a classical or quantum cluster algebra of a chosen
kind: $\alg$ might denote a cluster algebra $\bClAlg$, an upper
cluster algebra $\bUpClAlg$, or their localization $\clAlg$, $\upClAlg$
at the frozen variables. Moreover, $\clAlg$ and $\upClAlg$ are $\frRing$-algebras
over the frozen torus algebra $\frRing$, see Section \ref{subsec:Basics-of-cluster}.
Recall that $\alg$ is constructed from a combinatorial data called
a seed $\sd$, denoted $\alg=\alg(\sd)$.

Assume that $\alg$ has a basis $\base$ subject to certain natural
conditions, such as containing all cluster monomials (Definition \ref{def:based-cluster-algebra}).
The pair $(\alg,\base)$ is called a based cluster algebra. 

Let $\cT$ denote a monoidal category and $K$ its (deformed) Grothendieck
ring. We say $(\alg,\base)$ is categorified by $\cT$ if there is
an isomorphism $\kappa:\alg\simeq K$, such that $\kappa(\base)$
consists of the isoclasses of the simple objects of $\cT$ up to scalar
multiples. We say $(\alg,\base)$ is quasi-categorified by $\cT$
if it is categorified up to quantization changes, base changes (changing
the base $\frRing$), and localization, see Definition \ref{def:quasi-categorification}.
Note that, assuming mild conditions, we can keep track of the structures
constants of $\base$ under these changes using the correction technique
(Theorem \ref{thm:correction}). Particularly, quasi-categorification
implies that the structure constants of $\base$ are positive.

\begin{Thm}\label{thm:intro-quasi-categorification}

For almost all the quantum upper cluster algebras $\upClAlg$ arising
from Lie theory, $\upClAlg$ possess the common triangular bases $\can$
in the sense of \cite{qin2017triangular}. Moreover, $(\upClAlg,\can)$
are quasi-categorified by non-semisimple categories when the generalized
Cartan matrices $C$ are symmetric.

\end{Thm}

Note that $\can$ contains all cluster monomials by its definition.
By \cite{qin2017triangular,qin2020dual}, any $\qO[N(w)]$ has $\can$
and it coincides with the dual canonical basis $\dCan$. Therefore,
$\can$ in Theorem \ref{thm:intro-quasi-categorification} provides
a natural analog of $\dCan$, fulfilling the expectation (2). 

Theorem \ref{thm:intro-quasi-categorification} applies to $\upClAlg(\sd)$
where $\sd$ belongs to the following families.
\begin{enumerate}
\item \textbf{Algebraic group }$G$: Assume $C$ is of finite type. Let
$G$ denote the associated connected, simply connected, complex semisimple
algebraic group $G$. We claim that the (quantized) coordinate rings
$\C[G]$ and $\qO[G]$ are cluster algebras $\bUpClAlg(\sd)$ for
some $\sd$, see Claim \ref{claim:G_case}. The proofs will appear
elsewhere (\cite{fomin2020introduction} treated $\C[SL_{n}]$).
\item \textbf{Subvarieties of }$G$: Let $\qO[N(w)]$, $\qO[N^{w}]$, and
$\qO[G^{u,v}]$ denote quantum unipotent subgroups, quantum unipotent
cells, and quantum double Bruhat cells. They are quantum cluster algebras
associated with some seeds $\sd$, see \cite{GeissLeclercSchroeer10,GeissLeclercSchroeer11}\cite{GY13,goodearl2016berenstein}.
\item \textbf{Configurations of flags}: The coordinate rings of Grassmannians,
open Positroid varieties, open Richardson varieties, double Bott-Samelson
cells, and braid varieties are classical cluster algebras for some
$\sd$, see \cite{scott2006grassmannians} \cite{serhiyenko2019cluster}
\cite{galashin2019positroid} \cite{elek2021bott}\cite{shen2021cluster}
\cite{galashin2022braid} \cite{casals2022cluster}.
\end{enumerate}
Exotic cluster structures on $\C[G]$ by \cite{gekhtman2023unified}
are not examined in Theorem \ref{thm:intro-quasi-categorification}.

Before this work, statements in Theorem \ref{thm:intro-quasi-categorification}
had been established only in the very special cases $\qO[N(w)]$,
$\qO[N^{w}]$ \cite{qin2017triangular,qin2020dual} \cite{Kang2018}
\cite{mcnamara2021cluster}. Theorem \ref{thm:intro-quasi-categorification}
provides a vast generalization of those results, confirming that the
same statements hold in a far more general setting.

\begin{Rem}

Family (3) includes cluster algebras arising from various monoidal
categories: In type $ADE$, it includes those arising from representations
of quantum affine algebras, see Section \ref{subsec:Cluster-structures-quantum-affine}
and \cite{HernandezLeclerc09}\cite{qin2017triangular}; In type $A_{1}^{(1)}$,
it includes those arising from Satake categories, see the last quiver
in \cite[Figure 6.1]{cautis2019cluster}.

\end{Rem}

\begin{Rem}

In Theorem \ref{thm:intro-quasi-categorification}, $\can$ is positive
when $C$ is symmetric. In this case, $(\upClAlg,\can)$ is always
categorified by a semisimple category, see Remark \ref{rem:ss-categorification}.
We use modules of quiver Hecke algebras \cite{KhovanovLauda08,KhovanovLauda08:III}\cite{Rouquier08}
to provide a non-semisimple category for Theorem \ref{thm:intro-quasi-categorification}.
When $C$ is of type $ADE$, we can construct an alternative category
using modules of quantum affine algebras by Theorem \ref{thm:sub_category_upClAlg}.

\end{Rem}

\begin{Rem}[Quantized coordinate rings and integral forms]

For varieties from Lie theory, the quantization of their coordinate
rings are often not established in literature. But when the coordinate
ring is a cluster algebra, we have natural quantization from cluster
theory. And the quantum cluster algebra defined over $\Z[q^{\pm\Hf}]$
could be viewed as the integral form of the quantum cluster algebra
defined over $\Q(q^{\Hf})$. From this perspective, cluster theory
provides quantized coordinate rings and their integral forms.

\end{Rem}

\begin{Rem}[Deformation quantization]

One often understands the non-commutative multiplications in quantum
upper cluster algebras $\upClAlg_{q}$ as being inspired by the underlying
Poisson structures on the corresponding classical upper cluster algebras
$\upClAlg_{1}$; see \cite{gekhtman2003cluster,GekhtmanShapiroVainshtein05}\cite{BerensteinFominZelevinsky05}\cite{gekhtman2017hamiltonian}.
Nevertheless, while $\upClAlg_{q}$ are commonly viewed as quantizations
of $\upClAlg_{1}$, it is unclear whether their classical limits---obtained
by specializing $q^{\Hf}$ to $1$---exactly coincide with $\upClAlg_{1}$,
see the discussion in \cite{geiss2020quantum}. 

However, for Lie-theoretic quantum upper cluster algebras $\upClAlg_{q}$
in Theorem \ref{thm:intro-quasi-categorification}, and for $\upClAlg_{q}$
possessing quantum theta bases \cite{davison2019strong} (need conditions
including $B(\sd)$ being skew-symmetric, see Section \ref{sec:Preliminaries}),
the existence of bases guarantees that the classical limits indeed
coincide with $\upClAlg_{1}$. Moreover, the natural bijections between
the bases of $\upClAlg_{1}$ and $\upClAlg_{q}$ induce canonical
linear maps from $\upClAlg_{1}$ to $\upClAlg_{q}$. In these cases,
it is natural and reasonable to view $\upClAlg_{q}$ as deformation
quantizations of $\upClAlg_{1}$, see \cite[Remark 1.7]{davison2019strong}.

\end{Rem}

We have the following conjecture, see Example \ref{eg:global-crystal-sl2}.

\begin{Conj}\label{conj:triangular_global_basis}

Up to $q^{\frac{\Z}{2}}$ multiples, the common triangular basis $\can$
for $\qO[G]$ coincides with the global crystal basis of $\qO[G]$
in the sense of \cite{Kashiwara93}.

\end{Conj}

Our approach to Theorem \ref{thm:intro-quasi-categorification} has
led to fruitful byproducts. In the following, we let $\rsd$ denote
seeds associated with double Bott-Samelson cells and $\dsd$ denote
the seeds associated with decorated double Bott-Samelson cells as
in \cite{shen2021cluster}; see Section \ref{sec:Cluster-algebras-signed-words}.

First, it is a fundamental yet largely open problem to determine if
we have $\clAlg=\upClAlg$ and $\bClAlg=\bUpClAlg$. Most known cases
concern classical $\clAlg=\upClAlg$, see \cite{muller2013locally}
\cite{canakci2015cluster} \cite{shen2021cluster} \cite{moon2022compatibility}
\cite{ishibashi2023u}. At the quantum level, \cite{goodearl2016berenstein}
proved $\clAlg=\upClAlg$ and $\bClAlg=\bUpClAlg$ for quantum double
Bruhat cells. In the following theorem, we make a significant progress
on this problem at the quantum level.

\begin{Thm}\label{thm:intro-A-U}

For the quantum upper cluster algebras in Theorem \ref{thm:intro-quasi-categorification},
the associated classical and quantum cluster algebras satisfy $\clAlg=\upClAlg$.
Moreover, we have $\bClAlg(\rsd)=\bUpClAlg(\rsd)$, where $\rsd$
denote seeds associated with double Bott-Samelson cells.

\end{Thm}

Moreover, we discover rich structures on $\bUpClAlg(\rsd)$ arising
from double Bott-Samelson cells with the help of the basis $\can$.
We introduce the interval variables $W_{[j,k]}$ and the fundamental
variables $W_{j}$ for $\bUpClAlg(\rsd)$ (Section \ref{sec:Applications-dBS}).
They are distinguished cluster variables in analogous to the Kirillov-Reshetikhin
modules and the fundamental modules of quantum affine algebras respectively.

\begin{Thm}[{Proposition \ref{prop:T-systems} Theorem \ref{thm:dBS_PBW}}]\label{thm:intro-dBS}

The quantum cluster algebra $\bUpClAlg(\rsd)$ associated with double
Bott-Samelson cells has the following structures.

(1) The standard basis $\stdMod$: it consists of normalized ordered
products of fundamental variables.

(2) The $T$-systems: these are recursive equations relating the interval
variables.

Moreover, its common triangular basis $\can$ equals the Kazhdan-Lusztig
basis associated with $\stdMod$.

\end{Thm}

At the classical level $\kk=\C$, the standard basis were known by
\cite{elek2021bott}\cite{shen2021cluster}\cite{gao2020augmentations}
(Remark \ref{rem:previous_dBS_is_CGL}).

We view the standard basis as an analog of the dual PBW basis for
a quantum group. Recall that the dual PBW basis consists of ordered
products of root vectors, which can be calculated via braid group
actions. In the subsequent paper \cite{qin2024infinite}, we will
extend $\bUpClAlg(\rsd)$ to infinite ranks and show that the fundamental
variables $W_{j}$ could also be computed via the braid group action
from \cite{jang2023braid}.

\begin{Conj}\label{conj:dBS-categorification}

If $C$ is symmetric, $(\bUpClAlg(\rsd),\can)$ admits a monoidal
categorification.

\end{Conj}

In Section \ref{subsec:Categories-associated-to-braid}, we prove
Conjecture \ref{conj:dBS-categorification} for $C$ of type $ADE$.
For this purpose, we will introduce new monoidal categories $\cC_{\beta}$,
$\cC_{\widetilde{\beta}}'$ consisting of finite dimensional modules
of quantum affine algebras, where $\beta$ and $\widetilde{\beta}$
are positive braids.

\begin{Thm}[{Theorem \ref{thm:ADE-braid-categorification}, Theorem \ref{thm:ADE-braid-categorification-ddBS}}]\label{thm:intro-dBS-ddBS-categorification}

Assume $C$ is of type $ADE$. Then the based cluster algebra associated
with double Bott-Samelson cells are categorified by $\cC_{\beta}$,
and those associated with decorated double Bott-Samelson cells are
categorified by $\cC_{\widetilde{\beta}}'$. 

\end{Thm}

Finally, assume the claim that $\qO[G]$ is a cluster algebra (Claim
\ref{claim:G_case}). Let $\can$ denotes its common triangular basis.
In our convention, $\qO[G]$ is the opposite algebra of some quantum
cluster algebra associated with a decorated double Bott-Samelson cell.
Then Theorem \ref{thm:intro-dBS-ddBS-categorification} implies the
following result.

\begin{Thm}\label{thm:intro-qO-categorification}

In type $ADE$, $(\qO[G]\otimes\Q(q^{\Hf}),\can)$ admits a monoidal
categorification.

\end{Thm}

\subsection{Strategy}

We employ a unified approach to various algebras in Theorem \ref{thm:intro-quasi-categorification}
based on the combinatorial similarities of different cluster algebras.
The following observation is crucial: 

\emph{While varieties from Lie theory could be very different, their
coordinate rings might possess closely-related seeds as cluster algebras.} 

Let $\alg$ denote $\clAlg$ or $\upClAlg$. Correspondingly, we introduce
and investigate operations that relate cluster algebras $\alg$ associated
with closely-related seeds, called cluster operations. Notably, they
include the following:
\begin{enumerate}
\item \textbf{Freezing} $\frz_{F}$ (Section \ref{sec:Freezing-operators}):
Recall that a seed $\sd$ is associated with a set of unfrozen vertices
$I_{\ufv}$ and a set of frozen vertices $I_{\fv}$. Choosing any
$F\subset I_{\ufv}$ and setting it to be frozen, we obtain a new
seed $\frz_{F}\sd$. We introduce freezing operators $\frz_{F}$ sending
basis elements of $\alg(\sd)$ to those of $\alg(\frz_{F}\sd)$.
\item \textbf{Base change} $\varphi$ (Section \ref{subsec:Base-changes}):
Assume two seeds $\sd$, $\sd'$ are similar, i.e., $I_{\ufv}=I_{\ufv}'$
and their $B$-matrices satisfy $\tB{}_{I_{\ufv}\times I_{\ufv}}=\tB'{}_{I_{\ufv}\times I_{\ufv}}$.
Under certain conditions, there is a homomorphism $\var:\frRing\rightarrow\frRing'$,
called a variation map \cite{kimura2022twist}\cite{fraser2016quasi},
such that we obtain an $\frRing'$-module isomorphism $\frRing'\otimes_{\frRing}\alg\overset{\varphi}{\simeq}\alg'$,
and $\base'$ can be constructed from $\base$ via $\varphi$.
\end{enumerate}
Despite their elementary definitions, these operations possess strong
properties and have broad applications. They propagate various structures
from one cluster algebra to another, including: the localized cluster
monomials (Theorem \ref{thm:projection_quantum_cluster_monomial}),
bases (Theorems \ref{thm:freeze_good_bases} and \ref{thm:sub_cluster_triangular_basis},
Lemmas \ref{lem:similar-cluster-basis} and \ref{lem:similar-basis},
Proposition \ref{prop:similar-common-tri-basis}), quasi-categorifications
or categorifications (Theorems \ref{thm:sub_category_upClAlg} and
\ref{thm:subcategory-b-upcluster}), and the $\clAlg=\upClAlg$ property
(Corollaries \ref{cor:freezing-A-U} and \ref{cor:base-change-A-U}).
In addition, using the octahedral axiom in triangulated categories,
we could interpret the freezing operators in additive categorifications,
which is of independent interests (Theorem \ref{thm:additive-freezing}).

Note that Theorem \ref{thm:intro-quasi-categorification} was known
for quantum unipotent cells $\qO[N^{w}]$, which are localization
of $\qO[N(w)]$. We hope to verify Theorem \ref{thm:intro-quasi-categorification}
for the (quantized) coordinate rings $\kk[\AVar]$ of other varieties
$\AVar$ using this known result, but $\AVar$ might be more complicated
than $N^{w}$. To circumvent this difficulty, we introduce and use
the following \textbf{extension and reduction technique}:

\emph{Extend the generalized Cartan matrices $C$ to $\tC$. Then
study the cluster algebras $\kk[\AVar]$ by applying cluster operations
to quantum unipotent cells $\qO[\tN^{\tw}]$ associated with $\tC$.}

In practice, our approach to \emph{$\kk[\AVar]$} is the following
(Sections \ref{sec:Cluster-algebras-signed-words} \ref{sec:Applications:-cluster-algebras-Lie}):
\begin{enumerate}
\item Take $\qO[\tN^{\tw}]$ associated with an extended generalized Cartan
matrix $\tC$.
\item By freezing $\qO[\tN^{\tw}]$ and applying base changes, obtain $\upClAlg(\rsd)$
associated with double Bott-Samelson cells for the original generalized
Cartan matrix $C$.
\item Obtain $\kk[\AVar]$ from $\upClAlg(\rsd)$, possibly after applying
cluster operations.
\end{enumerate}
\begin{Rem}

The extension and reduction technique is cluster theoretic. It would
be interesting to understand it in Lie theory. Note that, when the
author was preparing this paper, a paper \cite{bao2024product} on
totally positivity appeared, which also extended the Cartan matrix.
However, to the best of the author's knowledge, there is currently
no mathematical relationship or translation between the methods of
these two papers.

\end{Rem}

\subsection{Contents}

Section \ref{sec:Preliminaries} contains basic notions and results
for cluster algebras. New notions including dominant tropical points
are introduced in Section \ref{subsec:Bases-for-partial-compactified}.

In Section \ref{sec:Freezing-operators}, we introduce and investigate
the freezing operators. Base changes of cluster algebras are explored
in Section \ref{sec:coeff-change}.

In Section \ref{sec:monoidal_categorification}, we introduced based
cluster algebras. Then we discuss their categorifications and the
monoidal subcategories associated with the freezing process.

In Section \ref{sec:Cluster-algebras-signed-words}, we recall the
seeds associated with (decorated) double Bott-Samelson cells. Then
we prove that the associated cluster algebras have common triangular
bases and admit quasi-categorification, see Theorems \ref{thm:bases_dBS}
and \ref{thm:bases-compactified-dBS}.

In Section \ref{sec:Applications:-cluster-algebras-Lie}, applying
the results on double Bott-Samelson cells in Section \ref{sec:Cluster-algebras-signed-words},
we show Theorems \ref{thm:intro-quasi-categorification} and \ref{thm:intro-A-U}
for almost all the known cluster algebras from Lie theory.

In Section \ref{sec:Applications-dBS}, we explore rich structures
of the cluster algebras associated with double Bott-Samelson cells.
New monoidal categories are introduced to categorify them. 

In Section \ref{subsec:From-principal-coefficients}, we recall principal
coefficients seeds. In Section \ref{sec:Freezing-operators-in-additive},
we interpret freezing operators in additive categorifications.

\subsection{Convention\label{subsec:Convention}}

Unless otherwise specified, we work with a base ring $\kk=\Z$ for
the classical case, and $\kk=\Z[q^{\pm\Hf}]$ for the quantum case,
where $q^{\Hf}$ is a formal parameter. The elements in $\kk_{\geq0}:=\N[q^{\pm\Hf}]$
are called nonnegative. We say two elements $z_{1},z_{2}$ $q$-commute,
denoted $z_{1}\sim z_{2}$, if $z_{1}*z_{2}=q^{\alpha}z_{2}*z_{1}$
for some $\alpha\in\Q$.

We will use $\cdot$ to denote the commutative multiplication, which
is often omitted, and $*$ the $q$-twisted product, which is the
default multiplication. 

For any $I\times J$ matrix $B$ and $I'\subset I$, $J'\subset J$,
$B_{I'\times J'}$ denotes the $I'\times J'$ submatrix of $B$. We
denotes its $(i,j)$-th entry by $B_{ij}$, $B_{i,j}$, $b_{ij}$,
or $b_{i,j}$. $(\ )^{T}$ denotes the matrix transpose. All vectors
are column vectors unless otherwise specified.

Given any skew-symmetric $\Q$-matrix $(b_{ij})_{i,j\in I}$. We associate
with it a quiver: the set of vertices is $I$, and there is an arrow
from $i\rightarrow j$ with weight $b_{ij}$ if $b_{ij}>0$. A collection
of $w_{s}$-weighted arrows is equivalent to a $\sum_{s}w_{s}$-weighted
arrow. 

Given any $I'\subset I$, we take the natural decomposition $\Z^{I}\simeq\Z^{I'}\oplus\Z^{I\backslash I'}$
and let $\pr_{I'}$ denote the natural projection from $\Z^{I}$ to
$\Z^{I'}$. Let $<_{\lex}$ denote the lexicographical order on $\Z^{[1,l]}$,
and $<_{\rev}$ the reverse lexicographical order, i.e., $(a_{1},\ldots,a_{l})<_{\rev}(b_{1},\ldots,b_{l})$
if $(a_{l},\ldots,a_{1})<_{\lex}(b_{l},\ldots,b_{1})$.

A word $\ueta=(\eta_{1},\eta_{2},\ldots)$ has letters in $\Z_{>0}$,
while a signed word $\ubi=(\bi_{1},\bi_{2},\ldots)$ has letters in
$\Z\backslash\{0\}$. $\forall j,k\in\Z$, $[j,k]$ denotes $\{i\in\Z|j\leq i\leq k\}$
unless otherwise specified.

$\Span_{\kk}\cZ$ denotes the set of finite $\kk$-linear combinations
of elements of $\cZ$.

\section*{Acknowledgments}

The author would like to thank Bernhard Keller, Melissa Sherman-Bennett,
Peigen Cao, David Hernandez, Jianrong Li, Jianghua Lu, Linhui Shen,
and Milen Yakimov for discussion. Particularly, he is grateful to
Bernhard Keller for showing him the inspiring Figure \ref{fig:quiver_extension}
in MSRI, 2012.

\section{Preliminaries\label{sec:Preliminaries}}

\subsection{Basics of cluster algebra\label{subsec:Basics-of-cluster}}

\subsubsection*{Seeds}

Let $I$ denote a finite set of vertices and we choose its partition
$I=I_{\ufv}\sqcup I_{\fv}$ into the frozen vertices and unfrozen
vertices. $|I|$ is called the rank. Let $d_{i}$, $i\in I$, denote
a set of strictly positive integers, called symmetrizers. Let $d$
be their least common multiplier. Denote the Langlands dual $d_{i}^{\vee}=\frac{d}{d_{i}}$.

Given $b_{ij}\in\Q$, $i,j\in I$, such that $d_{i}^{\vee}b_{ij}=-d_{j}^{\vee}b_{ji}$.
Define the $B$-matrix $\tB=(b_{ik})_{i\in I,k\in I_{\ufv}}$ and
assume it is a $\Z$-matrix. Define its unfrozen part $B:=\tB_{I_{\ufv}\times I_{\ufv}}$.

\begin{Assumption}\label{assumption:injectivity}

We assume that $\tB$ is of full rank unless otherwise specified.

\end{Assumption}

Introduce lattices $\cone=\Z^{I}$ and $\yCone=\Z^{I_{\ufv}}$ with
unit vectors $f_{i}$ and $e_{k}$ respectively. Define the linear
map $p^{*}:\yCone\rightarrow\cone$ such that $p^{*}(n)=\tB n$, $\forall n\in\yCone$.

By \cite{gekhtman2003cluster,GekhtmanShapiroVainshtein05}, we can
find a $\Z$-valued skew-symmetric bilinear form $\lambda$ on $\cone$
such that $\lambda(f_{i},p^{*}e_{k})=\delta_{ik}\diag_{k}$ for strictly
positive integers $\diag_{k}$. $\lambda$ is called a compatible
Poisson structure. Define the quantization matrix $\Lambda=(\Lambda_{ij})_{i,j\in I}:=(\lambda(f_{i},f_{j}))_{i,j}$.
The pair $(\tB,\Lambda)$ is called compatible by \cite{BerensteinZelevinsky05}.

\begin{Def}

A seed $\sd$ is a collection $(I,I_{\ufv},(d_{i})_{i\in I},\tB,(x_{i})_{i\in I})$,
where $x_{i}$ are indeterminates. It is called a quantum seed if
it is endowed with a compatible Poisson structure $\lambda$.

\end{Def}

We shall often omit the symbols $I,I_{\ufv},(d_{i})_{i\in I}$ when
the context is clear. 

Let $\kk$ denote $\Z$ or $\Z[q^{\pm\Hf}]$. Define the Laurent polynomial
$\LP=\kk[x_{i}^{\pm}]_{i\in I}$ using the usual commutative product
$\cdot$. Denote the Laurent monomial $x^{g}=\prod_{i\in I}x_{i}^{g_{i}}$,
$y_{k}=x^{p^{*}e_{k}}$, and $y^{n}=x^{p^{*}n}$ for $g\in\Z^{I},k\in I_{\ufv},n\in\Z^{I_{\ufv}}$.
We further endow $\LP$ with the following $q$-twisted product $*$:
\begin{align*}
x^{g}*x^{h} & =q^{\Hf\lambda(g,h)}x^{g+h},\ \forall g,h\in\Z^{I}.
\end{align*}
By the multiplication of $\LP$, we mean the $q$-twisted product.
Its subalgebra $\bLP=\kk[x_{k}^{\pm}]_{k\in I_{\ufv}}[x_{j}]_{j\in I_{\fv}}$
is called its partial compactification.\footnote{The partial compactification $\bLP$ should not be confused with the
bar-involution $\overline{(\ )}$ on $\LP$.} Let $\cF$ denote the skew-field of fractions of $\LP$. 

Define the frozen monoid $\frMonoid:=\{x^{m}|m\in\N^{I_{\fv}}\}$,
the frozen group $\frGroup:=\{x^{m}|m\in\Z^{I_{\fv}}\}$, and the
frozen torus algebra $\frRing:=\kk[x_{j}^{\pm}]_{j\in I_{\fv}}$,
$\bFrRing:=\kk[x_{j}]_{j\in I_{\fv}}$. We also define the bar involution
$\overline{(\ )}$ on $\LP$ as the $\Z$-anti-automorphism such that
$\overline{q^{\alpha}x^{g}}=q^{-\alpha}x^{g}$. 

The seed $\sd\op$ opposite to $\sd$ is the collection $(I,I_{\ufv},(d_{i})_{i\in I},-\tB,(x_{i})_{i\in I})$
endowed with the compatible Poisson structure $-\lambda$. We have
the canonical $\kk$-algebra anti-isomorphism $\iota:\LP(\sd)\simeq\LP(\sd\op)$
such that $\iota(x^{g})=x^{g}$, see \cite[(3.1)]{qin2020analog}
.

\subsubsection*{Mutations}

Given a seed $\sd$. Choose any vertex $k\in I$ and any sign $\varepsilon\in\{+,-\}$.
Let $[\ ]_{+}$ denote $\max(\ ,0)$. Following \cite{BerensteinFominZelevinsky05},
we have the following $I\times I$ matrix $\tE$ and the $I_{\ufv}\times I_{\ufv}$-matrix
$F$:

\begin{align*}
(\tE_{\varepsilon})_{ij} & =\begin{cases}
\delta_{ij} & k\notin\{i,j\}\\
-1 & i=j=k\\{}
[-\varepsilon b_{ik}]_{+} & j=k,i\neq k
\end{cases}, & (F_{\varepsilon})_{ij}=\begin{cases}
\delta_{ij} & k\notin\{i,j\}\subset I_{\ufv}\\
-1 & i=j=k\\{}
[\varepsilon b_{kj}]_{+} & i=k,j\neq k\in I_{\ufv}
\end{cases}.
\end{align*}

The mutation $\mu_{k}$ produces a new seed $\sd'=\mu_{k}\sd=(I,I_{\ufv},(d_{i})_{i\in I},\tB',(x_{i}')_{i\in I})$
such that $\tB'=\tE_{\varepsilon}\tB F_{\varepsilon}$. If $\sd$
is a quantum seed endowed with $\lambda$, $\mu_{k}$ also produces
a new quantum seed $\sd'$ with a quantization matrix $\Lambda'$,
such that $\Lambda'=(\tE_{\varepsilon})^{T}\Lambda\tE_{\varepsilon}$.
Note that $\mu_{k}$ is an involution and $\tB'$ and $\lambda'$
do not depend on $\varepsilon$, see \cite{BerensteinZelevinsky05}.

We further introduce a $\kk$-algebra isomorphism $\mu_{k}^{*}:\cF(t')\simeq\cF(t)$:
\begin{align}
\mu^{*}x_{i}' & =\begin{cases}
x_{k}^{-1}\cdot(x^{\sum_{j}[-b_{jk}]_{+}f_{j}}+x^{\sum_{i}[b_{ik}]_{+}f_{i}}) & i=k\\
x_{i} & i\neq k
\end{cases}.\label{eq:exchange_relation}
\end{align}
We will often identify $\cF(\sd')$ and $\cF(\sd)$ via $\mu_{k}^{*}$
and omit the symbol $\mu_{k}^{*}$. 

Let there be given an initial seed $\sd_{0}$. For any sequence $\uk=(k_{1},\ldots,k_{l})$
in $I_{\ufv}$, we use $\seq_{\uk}$ to denote the sequence of mutations
$\mu_{k_{l}}\cdots\mu_{k_{2}}\mu_{k_{1}}$ (read from right to left)
and $\seq_{\uk}^{*}$ the composition $\mu_{k_{1}}^{*}\mu_{k_{2}}^{*}\cdots\mu_{k_{l}}^{*}$.
Denote $\Delta^{+}:=\Delta_{\sd_{0}}^{+}:=\{\seq_{\uk}\sd_{0}|\forall\uk\}$.
Take any seed $\sd\in\Delta^{+}$. The associated $x$-variables $x_{i}(\sd)$
are called the (quantum) cluster variables in $\sd$. We call $x_{j}:=x_{j}(\sd)$,
$j\in I_{\fv}$, the frozen variables. In $\LP(\sd)$, the monomials
$x^{m}$, $m\in\N^{I}$, are called the cluster monomials in $\sd$,
and $x^{m}$, $m\in\N^{I_{\ufv}}\oplus\Z^{I_{\fv}}$, the localized
cluster monomials in $\sd$.

\begin{Def}\label{def:cluster-alg}

The (partially compactified) cluster algebra $\bClAlg(\sd_{0})$ is
the $\kk$-subalgebra of $\cF(\sd_{0})$ generated by the $x_{i}(\sd)$,
$\sd\in\Delta^{+}$. The (partially compactified) upper cluster algebra
is $\bUpClAlg(\sd_{0})=\cap_{\sd\in\Delta^{+}}\bLP(\sd)$. Define
the (localized) cluster algebra to be the localization $\clAlg(\sd_{0}):=\bClAlg(\sd_{0})[x_{j}^{-1}]_{j\in I_{\fv}}$
and, similarly, the (localized) upper cluster algebra $\upClAlg(\sd_{0})=\bUpClAlg(\sd_{0})[x_{j}^{-1}]_{j\in I_{\fv}}=\cap_{\sd\in\Delta^{+}}\LP(\sd)$. 

\end{Def}

By \cite{fomin2002cluster}\cite{BerensteinZelevinsky05}, we have
$\bClAlg(\sd_{0})\subset\bUpClAlg(\sd_{0})$ and $\clAlg(\sd_{0})\subset\upClAlg(\sd_{0})$.

\subsection{Dominance order, degrees, and support}

Take any seed $\sd$. We define $\yCone^{\oplus}(\sd)=\oplus_{k\in I_{\ufv}}\N e_{k}(\sd)$
and $\yCone^{+}(\sd)=\yCone^{\oplus}(\sd)\backslash\{0\}$. 

\begin{Def}[{Dominance order \cite{qin2017triangular}\cite[Proof of Proposition 4.3]{cerulli2015caldero}}]

$\forall g,h\in\cone(\sd)$. We say $g$ dominates $h$, denoted by
$h\preceq_{\sd}g$, if and only if $h=g+p^{*}n$ for some $n\in\yCone^{\oplus}(\sd)$.

\end{Def}

Omit the symbol $\sd$ below for simplicity. Recall that we have Laurent
monomials $y^{n}:=x^{p^{*}n}\in\LP$. The monoid algebra $\kk[\yCone^{\oplus}]:=\kk[y^{n}]_{n\in\yCone^{\oplus}(t)}$
has a maximal ideal $\kk[\yCone^{+}(t)]$. Denote the corresponding
completion by $\widehat{\kk[\yCone^{\oplus}(t)]}$. We define the
rings of formal Laurent series to be
\begin{align*}
\hLP & :=\LP\otimes_{\kk[\yCone^{\oplus}]}\widehat{\kk[\yCone^{\oplus}]},\ \widehat{\kk[\yCone]}:=\kk[\yCone]\otimes_{\kk[\yCone^{\oplus}]}\widehat{\kk[\yCone^{\oplus}]}.
\end{align*}
For $m\in M$, denote $\hLP_{\preceq m}:=x^{m}\cdot\widehat{\kk[\yCone^{\oplus}]}\subset\hLP$
and $\LP_{\preceq m}:=x^{m}\cdot\kk[\yCone^{\oplus}]\subset\LP$.

\begin{Def}

Consider any formal sum $z=\sum c_{m}x^{m}$, $c_{m}\in\kk$. 

(1) If the set $\{m|c_{m}\neq0\}$ has a unique $\prec_{\sd}$-maximal
element $g$, we say $z$ has degree $g$ and denote $\deg^{\sd}z=g$.
We further say that $z$ is $g$-pointed (or pointed at $g$) if $c_{g}=1$.
In this case, $\forall\alpha\in\Hf\Z$, define the normalization of
$q^{\alpha}z$ to be $[q^{\alpha}z]^{\sd}:=z$.

(2) If the set $\{m|c_{m}\neq0\}$ has a unique $\prec_{\sd}$-minimal
element $h$, we say $z$ has codegree $h$ and denote $\codeg^{\sd}z=h$.
We further say that $z$ is $h$-copointed (or copointed at $h$)
if $c_{h}=1$.

\end{Def}

For example, the localized cluster monomials are pointed at distinct
degrees and copointed at distinct codegrees in $\LP(\sd)$, see \cite{FominZelevinsky07}\cite{DerksenWeymanZelevinsky09}\cite{Tran09}\cite{gross2018canonical}.

\begin{Def}[Support]\label{def:support}

$\forall n=\sum n_{k}e_{k}\in\yCone$, we define its support $\supp n:=\{k\in I_{\ufv}|n_{k}\neq0\}$.
For any $z=x^{g}\cdot\sum_{n\in\yCone^{\oplus}}b_{n}y^{n}$ in $\hLP$
with $b_{n}\in\kk$, $b_{0}\neq0$. We define its support $\supp z:=\cup_{n:b_{n}\neq0}\supp n$.
If further $z\in\LP$, we define its support dimension $\suppDim z\in\yCone^{\oplus}$
such that $(\suppDim z)_{k}=\max\{n_{k}|b_{n}\neq0\}$, $\forall k\in I_{\ufv}$.

\end{Def}

Following \cite{qin2017triangular}, we say $\sd_{0}$ injective-reachable
if there exists a seed $\sd_{0}[1]=\Sigma\sd_{0}$ for some mutation
sequence $\Sigma$ and some permutation $\sigma$ of $I_{\ufv}$,
such that $\deg^{\sd_{0}}x_{\sigma k}(\sd_{0}[1])\in-f_{k}+\Z^{I_{\fv}}$,
for any $k\in I_{\ufv}$. In this case, $\Sigma$ is called a green
to red sequence following \cite{keller2011cluster}. Note that, if
$\sd_{0}$ is injective-reachable, then so are all seeds $\sd\in\Delta^{+}$
\cite{qin2017triangular}\cite{muller2015existence}. 

\begin{Def}[\cite{qin2017triangular}]

Let $z_{s}$, $s\in\N$, denote a collection of pointed elements in
$\hLP$. A (possibly infinite) decomposition in $\hLP$ is called
$\prec_{\sd}$-unitriangular if it takes the form $z=z_{0}+\sum_{\deg z_{s}\prec_{\sd}\deg z_{0}}b_{s}z_{s},\ b_{s}\in\kk$.
It is further called $(\prec_{\sd},\mm)$-unitriangular, where $\mm:=q^{-\Hf}\Z[q^{-\Hf}]$,
if $b_{s}\in\mm$.

\end{Def}

\begin{Lem}\label{lem:bounded_basis}

Take any $\Theta\subset\cone(\sd)$. Let $\base=\{b_{m}|m\in\Theta\}$
denote a subset of $\LP(\sd)$, where $b_{m}$ are $m$-pointed. Denote
$V=\Span_{\kk}\base$. Let $Z=\{z_{m}|m\in\Theta\}$ denote any subset
of $V$ where $z_{m}$ are $m$-pointed. Then $Z$ is also a $\kk$-basis
of $V$ if $\prec_{\sd}$ is bounded from below on $\Theta$.

\end{Lem}

\begin{proof}

We have a decomposition $z_{m}=b_{m}+\sum_{m'\prec_{\sd}m\in\Theta}c_{m'}b_{m'}$,
$c_{m'}\in\kk$, $\forall m\in\Theta$. Take the inverse transition
$b_{m}=z_{m}+\sum_{m'\prec_{\sd}m\in\Theta}c'_{m'}z_{m'},$ $c_{m'}'\in\kk$,
see \cite{qin2017triangular}. This is a finite sum by the boundedness
of $\prec_{\sd}$ on $\Theta$. The desired claim follows.

\end{proof}

\subsection{Tropical transformations}

Take seeds $\sd'=\mu_{k}\sd$ for some $k\in I_{\ufv}$. The tropical
transformation $\phi_{\sd',\sd}:\cone(\sd)\ra\cone(\sd')$ is piecewise
linear such that, for any $m=\sum m_{i}f_{i}(\sd)\in\oplus\Z f_{i}(\sd)=\cone(\sd)$,
its image $m'=\sum m_{i}'f_{i}(\sd')\in\cone(\sd')$ is given by $m_{i}'=\begin{cases}
-m_{k} & i=k\\
m_{i}+b_{ik}(t)[m_{k}]_{+} & b_{ik}(t)\geq0\\
m_{i}+b_{ik}(t)[-m_{k}]_{+} & b_{ik}(t)\leq0
\end{cases}$.

For any seeds $\sd'=\seq_{\uk}\sd$, we define $\phi_{\sd',\sd}$
to be the composition of the tropical transformations along $\seq_{\uk}$.
It is independent of the choice of $\seq_{\uk}$, see \cite{gross2013birational}. 

\begin{Def}[\cite{qin2019bases}]

$\forall m\in\cone(\sd)$ , the tropical point represented by $m$
is defined to be the equivalence class $[m]$ in $\sqcup_{\sd'\in\Delta^{+}}\cone(\sd')$
under the identifications $\phi_{\sd',\sd}$, $\forall\sd'$. Define
the set of tropical points $\tropSet$ to be $\{[m]|\forall m\in\cone(\sd)\}$.

\end{Def}

\begin{Def}[\cite{qin2019bases}]

Given an $m$-pointed element $z\in\LP(\sd)$. It is said to be compatibly
pointed at $\sd,\sd'$ if $z$ is $\phi_{t',t}m$-pointed in $\LP(t')$.
We say $z$ is $[m]$-pointed if it is compatibly pointed at $\sd,\sd'$,
for all $\sd'\in\Delta^{+}$.

\end{Def}

For example, by \cite{DerksenWeymanZelevinsky09}\cite{Tran09}\cite{gross2018canonical},
the localized cluster monomials are pointed at distinct tropical points.

\begin{Def}

Let $\Theta$ denote a subset of $\cone(\sd)$ or $\tropSet$ and
$\cS=\{s_{g}|g\in\Theta\}$ be a subset of $\hLP(\sd)$. If $s_{g}$
are $g$-pointed, $\cS$ is said to be $\Theta$-pointed.

\end{Def}

\begin{Lem}[{\cite[Lemma 3.4.12]{qin2019bases}}]\label{lem:tropicalized_same_cluster_monomial}

Assume that $\sd$ is injective-reachable. Take any $u\in\upClAlg(\sd)$.
If there exists a localized cluster monomial $z$ such that $u$ is
$[\deg z]$-pointed, then $u=z$.

\end{Lem}

\subsection{Partial compactification and optimized seeds\label{subsec:Bases-for-partial-compactified}}

Take any seed $\sd$ without assuming $\tB(\sd)$ to be of full rank.
For any vertex $j$, let $\nu_{j}$ denote the order of vanishing
at $x_{j}=0$ on $\cF(\sd)$, i..e, $\forall Z\in\cF(t)$, we have
the reduced form $Z=x_{j}^{v_{j}(Z)}*P*Q^{-1}$, where $P,Q\in\kk[x_{i}]_{i\in I}$
are not divisible by $x_{j}$. Note that $\nu_{j}$ is a valuation.
We say $Z\in\cF(t)$ is regular at $x_{j}=0$ if $\nu_{j}(Z)\geq0$.

\begin{Lem}\label{lem:no_pole_change_chart}

Take any $\sd'=\mu_{k}\sd$, $k\in I_{\ufv}$, $j\neq k$. Let $\nu_{j}$
and $\nu_{j}'$ denote the orders of vanishing on $\cF(\sd)$ and
$\cF(\sd')$ respectively. Then for any $Z\in\cF(\sd')$, we have
$\nu_{j}'(Z)=\nu_{j}(\mu_{k}^{*}Z)$.

\end{Lem}

\begin{proof}

Take any polynomial $P'\in\kk[x_{i}(\sd')]_{i\in I}$ not divisible
by $x_{j}$. View it as a polynomial in $x_{j}$ and $x_{k}$, we
can write $P'=\sum_{r,s\geq0}c_{r,s}*(x'_{k})^{r}*x_{j}^{s}$ with
$0\neq\sum_{r}c_{r,0}*(x'_{k})^{r}$. Note that $\mu_{k}^{*}c_{r,s}=c_{r,s}$
and $(\mu_{k}^{*}(x'_{k})^{r})|_{x_{j}\mapsto0}=M_{r}*x_{k}^{-r}\in\kk[x_{i}]_{i\in I}[x_{k}^{-1}]$
for $M_{r}$ a polynomial in $x_{i}$, $i\neq j,k$. We deduce that
$\mu_{k}^{*}P'|_{x_{j}\mapsto0}=\sum_{r}c_{r,0}*M_{r}*x_{k}^{-r}\neq0$.
Therefore, $\nu_{j}(\mu_{k}^{*}P')=0$. The desired statement follows
by applying this result to the reduced form of $Z$ in $\cF(\sd')$.

\end{proof}

Observe that $\bUpClAlg=\{z\in\upClAlg|\nu_{j}(z)\geq0,\forall j\in I_{\fv}\}$.
Assume $\tB(\sd)$ is of full rank now.

\begin{Lem}\label{lem:f-bounds-d}

For any $m$-pointed $z\in\upClAlg(\sd)\subset\LP(\sd)$, if some
$k\in I_{\ufv}$ is not contained in its support (Definition \ref{def:support}),
we have $\nu_{k}(z)\geq0$.

\end{Lem}

\begin{proof}

We can assume that $\sd$ is the seed $\sd^{\prin}$ of principal
coefficients (Section \ref{subsec:From-principal-coefficients}).
Otherwise, we replace the pointed element $z$ by the similar element
$z'$ in $\LP(\sd^{\prin})$ such that $\nu_{k}(z)=\nu_{k}(z')$ as
in Definition \ref{def:similar-elements}. Then we have have $\supp z=\supp z'$.
Moreover, Lemma \ref{lem:similar-alg-element} implies $z'\in\upClAlg(\sd^{\prin})$.

When $\sd=\sd^{\prin}$, one can prove the claim from the condition
$z\in\LP(\sd)\cap\LP(\mu_{k}\sd)$. We have the following proof from
\cite{cao2024exchange}. 

Denote $I_{\ufv}:=I_{\ufv}(\sd)$. Since $\sd=\sd^{\prin}$, we have
$I_{\fv}:=I_{\fv}(\sd)=\{k'|k\in I_{\ufv}\}$. Denote $U_{\hk}:=\kk[x_{i}^{\pm},x_{i'}^{\pm}]_{i\in I_{\ufv},i\neq k}$
and $V_{\hk}:=U_{\hk}[x_{k'}^{\pm}]$. Since $z\in\LP(\sd)$, we have
$z=\sum_{s\geq0}L_{s}*x{}_{k}{}^{s}+\sum_{s>0}x{}_{k}{}^{-s}*L_{-s}$
in $\LP(\sd)=V_{\hk}[x_{k}^{\pm}]$, where $L_{\pm s}\in V_{\hk}$. 

Denote $\sd'=\mu_{k}\sd$ and $x_{k}'=x_{k}(\sd')$. When $s\in\N$,
the Laurent expansion of $x_{k}^{s}$ in $\LP(\sd')$ takes the form
$x_{k}^{s}=\tP_{s}*(x'_{k})^{-s}$, where $\tP\in U_{\hk}[x_{k'}]$
is a degree $s$ polynomial in $x_{k'}$ with non-vanishing constant
term.

Then we obtain $z=\sum_{s\geq0}L_{s}*\tP_{s}*(x'_{k})^{-s}+\sum_{s>0}(x'{}_{k})^{s}*\tP_{s}{}^{-1}*L_{-s}$
in $\LP(\sd')=V_{\hk}[(x'_{k})^{\pm}]$. So we must have $\tP_{s}{}^{-1}*L_{-s}\in V_{\hk}=U_{\hk}[x_{k'}^{\pm}]$.
Denote $Q_{s}=\tP_{s}{}^{-1}*L_{-s}\in U_{\hk}[x_{k'}^{\pm}]$, we
have $\tP_{s}*Q_{s}=L_{-s}\in U_{\hk}$. Recall that $\tP_{s}$ is
a polynomial in $x_{k'}$ of degree $s$ with non-vanishing constant
term. By analyzing the degree of $x_{k'}$, we deduce that $Q_{s}=L_{-s}=0$,
$\forall s>0$. Equivalently, $\nu_{k}(z)\geq0$.

\end{proof}

Let $j$ denote any frozen vertex of a seed $\sd$. We say $j$ is
optimized in $\sd$ if $b_{jk}(\sd)\geq0$ for any $k\in I_{\ufv}$,
or, equivalently, $\nu_{j}(y_{k}(\sd_{j}))\geq0$ . We say $j$ can
be optimized if it is optimized in some seed $\sd_{j}\in\Delta_{\sd}^{+}$,
and $\sd$ can be optimized if its frozen vertices can be. We will
call $\sd_{j}$ a $j$-optimized seed or an optimzied seed of $j$.

\begin{Lem}\label{lem:optimized_pointed_elem}

Take any $j\in I_{\fv}$ with an $j$-optimized seed $\sd_{j}$. If
$z\in\LP(\sd_{j})$ has degree $m$, then $\nu_{j}(z)=\nu_{j}(m)$.

\end{Lem}

The following result is analogous to \cite[Proposition 9.7]{gross2018canonical}.

\begin{Prop}\label{prop:compactified_basis}

Let $\cS$ denote any basis of $\upClAlg$ such that it is $\tropSet$-pointed.
Assume that there is a $j$-optimized seed $\sd_{j}$ for each $j\in I_{\fv}$.
Then $\cS\cap\bUpClAlg$ is a basis of $\bUpClAlg$.

\end{Prop}

\begin{proof}

For any element $Z\in\bUpClAlg$, we have a finite sum $Z=\sum_{i=1}^{r}c_{i}S_{i}$
for basis elements $S_{i}\in\cS$ and $0\neq c_{i}\in\kk$. It suffices
to show that, for any $j\in I_{\fv}$, $\nu_{j}(S_{i})\geq0$. 

Working inside $\LP(\sd_{j})$ and denote $m_{i}=\deg^{\sd_{j}}S_{i}$.
Let $m_{k}$ denote a $\prec_{\sd_{j}}$-maximal element of $\{m_{i}|i\in[1,r]\}$.
Then $x^{m_{k}}$ must appear in the Laurent expansion of $Z$ in
$\LP(\sd_{j})$. So it is regular at $x_{j}=0$. By Lemma \ref{lem:optimized_pointed_elem},
$\nu_{j}(S_{k})\geq0$. Replace $Z$ by $Z-S_{k}$ and consider $Z-S_{k}=\sum_{i\neq k}c_{i}S_{i}$.
Repeating this process, we obtain that all $S_{i}$ appearing are
regular at $x_{j}=0$, i.e., $\nu_{j}(S_{i})\geq0$.

\end{proof}

\begin{Def}\label{def:domCone}

Choose any subset $\cZ=\{z_{[m]}|[m]\in\tropSet\}$ of $\upClAlg$
such that $z_{[m]}$ are $[m]$-pointed. We define the associated
set of dominant tropical points to be $\domTropSet:=\{[m]\in\tropSet|z_{[m]}\in\bUpClAlg\}$.
Denote $\domCone(\sd):=\{m\in\cone(\sd)|[m]\in\domTropSet\}.$

\end{Def}

By Lemma \ref{lem:optimized_pointed_elem}, $\domTropSet$ can be
calculated as below. 

\begin{Lem}\label{lem:optimized-domCone}

If there exist $j$-optimized seeds $\sd_{j}$ for all $j\in I_{\fv}$,
$\domCone(\sd)$ equals $\{m|(\phi_{\sd_{j},\sd}m)_{j}\geq0,\forall j\in I_{\fv}\}$.

\end{Lem}

In general, we will choose $\cZ$ to be the set of theta functions
\cite{gross2018canonical}, because their existence has been verified
in vast generalities.

\begin{Rem}

$\cZ$ can be taken as the theta basis, the generic basis, or triangular
basis, when they exist. We conjecture that the resulting $\domTropSet$
are equivalent for many (if not all) cluster algebras.

\end{Rem}

\begin{Def}\label{def:non-essential-frozen}

A frozen vertex $j$ is called non-essential if $\tB_{\{j\}\times I_{\ufv}}=0$,
and essential otherwise.

\end{Def}

If $j$ is non-essential, let $\sd'$ denote the (quantum) seed naturally
obtained from $\sd$ by removing $j$. Then we have $\bUpClAlg=\kk[x_{j}]\bUpClAlg$
and $\bClAlg=\kk[x_{j}]\bClAlg$.

\subsection{Triangular bases}

Take $\kk=\Z[q^{\pm\Hf}]$. Let $\alg$ denote $\clAlg$ or $\upClAlg$.
Assume $\sd$ is injective-reachable. 

\begin{Def}[{\cite{qin2017triangular}}]

Let $\can=\{\can_{m}|m\in\cone(\sd)\}$ be a $\kk$-basis of $\alg(\sd)$,
such that $\can_{m}$ are $m$-pointed. $\can$ is called the triangular
basis with respect to the seed $\sd$ if we have the following:

\begin{itemize}

\item $\overline{{\can_{m}}}=\can_{m}$.

\item $\can$ contains the cluster monomials in $\sd$ and $\sd[1]$.

\item$\forall x_{i}(\sd),\can_{m}$, the decomposition of the normalization
$[x_{i}(\sd)*\can_{m}]^{\sd}$ is $(\prec_{t},\text{\ensuremath{\mm})}$-unitriangular:
\begin{align}
[x_{i}(\sd)*\can_{g}]^{\sd} & =\can_{g+f_{i}}+\sum_{g'\prec_{\sd}g+f_{i}}b_{g'}\can_{g'},\ \text{for some }b_{g'}\in\mm:=q^{-\Hf}\Z[q^{-\Hf}].\label{eq:triangular_basis_triangularity}
\end{align}

\end{itemize}

\end{Def}

The decomposition on RHS of (\ref{eq:triangular_basis_triangularity})
is related to Leclerc's conjecture \cite{Leclerc03} \cite{qin2020analog}.
Note that $[x_{j}^{\pm}*\can_{m}]^{\sd}=\can_{m\pm f_{j}}$, $\forall j\in I_{\fv}$.
The triangular basis with respect to $\sd$ is determined by a Kazhdan-Lusztig
type algorithm, see \cite[Section 6.1]{qin2020dual}. In particular,
it is unique if it exists. 

\begin{Def}[{\cite[Definition 5.5, Remark 5.6]{qin2021cluster}\cite{qin2017triangular}}]

A $\kk$-basis $\can$ of $\alg$ is called the common triangular
basis if it is the triangular basis with respect to any seed $\sd\in\Delta^{+}$.

\end{Def}

By \cite[Proposition 6.4.3]{qin2020dual}, if $\can$ is the triangular
basis with respect to the seed $\sd$, then it is the common triangular
basis if and only if it contains the cluster monomials of any seeds
$\sd'\in\Delta^{+}$. In this case, the basis elements $\can_{m}$
are $[m]$-pointed by \cite[Proposition 6.4.3]{qin2020dual}. Then
\cite[Theorem 4.3.1]{qin2019bases} implies that $\can$ is a basis
for $\upClAlg$ and, consequently, $\alg=\upClAlg$.

Recall that we have the canonical anti-isomorphism $\iota:\LP(\sd)\simeq\LP(\sd\op)$.
By \cite[Proposition 4.9]{qin2020analog}, if $\alg(\sd)$ has the
common triangular basis $\can$, then $\iota(\can)$ is the common
triangular basis for $\alg(\sd\op)$.

\begin{Def}

If $\can$ is the common triangular basis (resp. the triangular basis
with respect to $\sd$) of $\alg$ such that $\can\cap\overline{\alg}$
is a basis of $\overline{\alg}$, we call $\can\cap\overline{\alg}$
the common triangular basis (resp. the triangular basis with respect
to $\sd$) of $\overline{\alg}$.

\end{Def}

\section{Freezing operators\label{sec:Freezing-operators}}

\subsection{Sub seeds\label{subsec:Sub-seeds}}

Take two seeds $\sd=(I,I_{\ufv},(d_{i}),\tB,(x_{i})_{i\in I})$, $\sd'=(I',I'_{\ufv},(d'_{i}),\tB',(x'_{i})_{i\in I'})$.
We do not assume $\tB$, $\tB'$ to be of full rank for the moment.

\begin{Def}\label{def:calibration}

Assume there is an inclusion $\iota:I'\hookrightarrow I$ with $\iota I'_{\ufv}\subset I_{\ufv}$,
$d'_{i}=d_{\iota(i)}$, $b'_{i,k}=b_{\iota(i),\iota(k)}$, $\forall i\in I',k\in I'_{\ufv}$.
It is called a cluster embedding from $\sd$ to $\sd'$. If $\sd$
and $\sd'$ are quantum seeds, we additionally require $\Lambda(\sd')_{i,j}=\Lambda(\sd)_{\iota(i),\iota(j)}$,
$\forall i,j\in I'$.

In this case, $\sd'$ is called a sub seed of $\sd$ via $\iota$,
denoted $\iota:\sd'\subset\sd$. It is called a good sub seed of $\sd$,
if we further have $b_{i,k}=0$, $\forall k\in\iota(I_{\ufv}')$,
$i\in I\backslash\iota(I')$.

\end{Def}

Assume a cluster embedding $\iota:I'\hookrightarrow I$ is given.
$\forall$ mutation sequence $\seq=\mu_{k_{r}}\cdots\mu_{k_{1}}$
on $I'_{\ufv}$, let $\iota(\seq)$ denote $\mu_{\iota(k_{r})}\cdots\mu_{\iota(k_{1})}$.
By the mutation rules, $\iota$ is also a cluster embedding from the
classical seed $\seq\sd'$ to $\iota(\seq)\sd$.

We obtain the inclusion $\iota:\cF(\sd')\subset\cF(\sd)$ sending
$x_{i}'$ to $x_{\iota(i)}$. If $\sd'$ is a good sub seed of a classical
seed $\sd$, and $\sd$ has a quantization matrix $\Lambda(\sd)$,
then we can choose the quantization for $\sd'$ such that $\Lambda(\sd')_{i,j}=\Lambda(\sd)_{\iota(i),\iota(j)}$.

We often omit $\iota$ for simplicity. The mutation rules imply the
following result.

\begin{Prop}

Assume that $\sd'$ is a good sub seed of $\sd$ via $\iota$. Then
$\iota$ is a cluster embedding from the (quantum) seed $\seq\sd'$
to $\iota(\seq)\sd$, $\seq\sd'$ is a good sub seed of $\iota(\seq)\sd$
via $\iota$, $\iota(x_{i}(\seq\sd'))=x_{i}(\seq\sd)$, and $\iota(\bLP(\seq\sd'))\subset\bLP(\iota(\seq)\sd)$
for any $i\in I'$ and any mutation sequence $\seq$ on $I'_{\ufv}$.
Particularly, we have $\iota(\bClAlg(\sd'))\subset\bClAlg(\sd)$.

\end{Prop}

\begin{Prop}\label{prop:good-sub-up-cl-alg}

When $\sd'$ is a good sub seed of $\sd$ via $\iota$, we have $\bUpClAlg'\subset\bUpClAlg$
under the assumption that $\sd$ is of full rank.

\end{Prop}

\begin{proof}

We omit the symbol $\iota$ for simplicity. Since $\sd$ is of full
rank, we have $\upClAlg=\LP\cap(\cap_{k\in I_{\ufv}}\LP(\mu_{k}\sd))$
by \cite[Corollary 1.9]{BerensteinFominZelevinsky05}\cite[Theorem 5.1]{BerensteinZelevinsky05}.
It follows that $\bUpClAlg=\bLP\cap(\cap_{k\in I_{\ufv}}\bLP(\mu_{k}\sd))$.
Take any $z\in\bUpClAlg'$. It belongs to $\bLP'\cap(\cap_{k'\in I'_{\ufv}}\bLP(\mu_{k'}\sd'))\subset\bLP\cap(\cap_{k'\in I'_{\ufv}}\bLP(\mu_{k'}\sd))$.
Moreover, $\forall k\in I_{\ufv}\backslash I'_{\ufv}$, $z\in\bLP'\subset\bLP$
is also contained in $\bLP(\mu_{k}\sd)$ since $\nu_{k}(z)\geq0$.
Therefore, $z\in\bUpClAlg$.

\end{proof}

\subsection{Freezing}

Assume $\tB$ to be of full rank. Choose any subset $F\subset I_{\ufv}$.
By freezing $F$, we obtain a good sub seed of $\sd$, denoted $\frz_{F}\sd:=(I,I_{\ufv}\backslash F,(d_{i}),\tB_{I\times(I_{\ufv}\backslash F)},(x_{i})_{i\in I})$.
If $\sd$ is a quantum seed, $\frz_{F}\sd$ is also quantized where
we choose $\Lambda(\frz_{F}\sd):=\Lambda(\sd)$.

\begin{Rem}\label{rem:inclusion_frozen_subalg}

Note that $\cone(\frz_{F}\sd)=\cone(\sd)$, $m'\preceq_{\frz_{F}\sd}m$
implies $m'\preceq_{\sd}m$, $x_{i}(\frz_{F}\sd)=x_{i}(\sd)$, and
$\cF(\frz_{F}\sd)=\cF(\sd)$. The identification $\cone(\frz_{F}\sd)=\cone(\sd)$
induces a natural identification between the set of tropical points
for $\upClAlg(\frz_{F}\sd)$ and $\upClAlg(\sd)$, which we both denote
by $\tropSet$.

\end{Rem}

\begin{Def}[Freezing operator]\label{def:freezing_operator}

For any degree $m\in\cone(\sd)$ and $F\subset I_{\ufv}$, we define
the freezing operator $\frz_{F,m}^{\sd}$ to be the $\kk$-linear
map from $\hLP_{\preceq m}(\sd)$ to $\hLP_{\preceq m}(\frz_{F}\sd)$
such that 
\begin{align*}
\mathfrak{\frz}_{F,m}^{\sd}(\sum_{n\in\yCone^{\oplus}}c_{n}x^{m}\cdot y^{n}) & =\sum_{n:\supp n\cap F=\emptyset}c_{n}x^{m}\cdot y^{n},\ \forall c_{n}\in\kk.
\end{align*}

\end{Def}

In general, by freezing $F'\subset I$, we mean freezing $F:=F'\cap I_{\ufv}$,
i.e., $\frz_{F'}\sd:=\frz_{F}\sd$, $\frz_{F',m}^{\sd}:=\frz_{F,m}^{\sd}$.
When the context is clear, we omit the scripts $\sd$, $F$ and/or
$m$ for simplicity. In particular, we often denote $\frz_{F}^{\sd}(z)=\frz_{F,m}^{\sd}(z)$
when $z$ has a unique $\prec_{\sd}$-maximal degree $m$. Following
this convention, for any collection $\cZ$ of elements $z$ with unique
maximal degrees, let $\frz_{F}^{\sd}\cZ$ denote the set consisting
of $\frz_{F}^{\sd}z$.

\begin{Lem}\label{lem:freezing_properties}

(1) $\forall$$m'\preceq_{\sd}m$ such that $m'=m+\tB n$ and $z\in\hLP_{\preceq m'}(\sd)\subset\hLP_{\preceq m}(\sd)$,
we have $\frz_{m}(z)=\begin{cases}
\frz_{m'}(z) & \supp n\cap F=\emptyset\\
0 & \supp n\cap F\neq\emptyset
\end{cases}$.

(2) $\forall$$m_{1,}m_{2}$ and $z_{1}\in\hLP_{\preceq m_{1}}(\sd)$,
$z_{2}\in\hLP_{\preceq m_{2}}(\sd)$, we have
\begin{align*}
\frz_{m_{1}}(z_{1})*\frz_{m_{2}}(z_{2}) & =\frz_{m_{1}+m_{2}}(z_{1}*z_{2}).
\end{align*}

\end{Lem}

\begin{proof}

The claims follow from definition of the freezing operators.

\end{proof}

\begin{Eg}

Let $\sd$ denote a seed such that $I=\{1,2\}$, $I_{\ufv}=\{1\}$,
$\tB=\left(\begin{array}{cc}
0 & -1\\
1 & 0
\end{array}\right)$. The upper cluster algebra $\upClAlg(\sd)$ equals $\kk[x_{2}^{\pm}][x_{1},x_{1}']$
where $x_{1}':=x_{1}^{-1}\cdot(1+y_{1})$. Its localized cluster monomials
are $s_{m}=x_{1}^{m_{1}}\cdot x_{2}^{m_{2}}$ when $m_{1}\geq0$,
and $s_{m}=(x_{1}')^{-m_{1}}\cdot x_{2}^{m_{2}}$ when $m_{1}<0$.
Choose $F=\{1\}$ and construct the seed $\frz_{F}\sd$ by freezing
$F$. We have $\upClAlg(\frz_{F}\sd)=\kk[x_{1}^{\pm},x_{2}^{\pm}]$.
Its localized cluster monomials $s'_{m}:=x^{m}$ satisfy $s'_{m}=\frz_{F,m}s_{m}$.

\end{Eg}

\begin{Prop}\label{prop:project_up_cl_alg}

Take any $z\in\LP_{\preceq m}(\sd)$. If $z\in\upClAlg(\sd)$, then
$\frz_{m}z\in\upClAlg(\frz_{F}\sd)$.

\end{Prop}

\begin{proof}

Take any sequence $\uk$ on $I_{\ufv}\backslash F$ and denote $\sd'=\seq_{\uk}\sd$.
Since $z\in\LP(\sd')$, we have $x(\sd')^{d'}*z=\sum_{g'}b_{g'}x(\sd')^{g'}$
for some $d',g'\in\N^{I}$, $b_{g'}\in\kk$. Now work in $\LP(\sd)$
and denote $d=\deg^{\sd}x(\sd')^{d'}$, $g=\deg^{\sd}x(\sd')^{g'}$.
Since $z\in\LP_{\preceq m}(\sd)$, we have $g\preceq_{\sd}d+m$ whenever
$b_{g'}\neq0$ (see \cite[Lemma 3.1.10(iii)]{qin2017triangular}).
Applying $\frz_{F,d+m}$, we obtain $\frz_{d}x(\sd')^{d'}*\frz_{m}z=\sum_{g'}b{}_{g'}\frz_{d+m}x(\sd')^{g'}$.
Note that, since $\uk$ is a sequence on $I_{\ufv}\backslash F$,
$\forall i$, we have $x_{i}(\seq_{\uk}\sd)=x_{i}(\seq_{\uk}\frz_{F}\sd)$,
which has support in $I_{\ufv}\backslash F$. Thus $\frz_{d}x(\sd')^{d'}$
equals $x(\seq_{\uk}\frz_{F}\sd)^{d'}$, and $\frz_{d+m}x(\sd')^{g'}$
is either $\frz_{g'}x(\sd')^{g'}=\frz_{g'}x(\seq_{\uk}\frz_{F}\sd)^{g'}=x(\seq_{\uk}\frz_{F}\sd)^{g'}$
or $0$. We deduce that $\frz_{m}z\in\LP(\seq_{\uk}\frz_{F}\sd)$.

\end{proof}

\subsection{Properties and applications\label{subsec:freeze_mutations}}

We will prove some key properties of the freezing operators. More
results are in \cite{qin2023freezing}. The following lemma is crucial.

\begin{Lem}[Mutation compatibility]\label{lem:freeze_mutation}

Take seeds $\sd=\mu_{k}\sd'$ for some $k\in I_{\ufv}\backslash F$
and $z\in\LP(\sd)\cap(\mu_{k}^{*})^{-1}\LP(\sd')$. Assume that $z$
has a unique $\prec_{\sd}$-maximal degree $m$ in $\LP(\sd)$ and
$\mu_{k}^{*}z$ has a unique $\prec_{\sd'}$-maximal degree $m'$
in $\LP(\sd')$. Then we have $\frz_{F,m'}^{\sd'}(\mu_{k}^{*}z)=\mu_{k}^{*}\frz_{F,m}^{\sd}(z)$.

\end{Lem}

\begin{proof}

For any $n=\sum n_{i}e_{i}(\sd)\in\yCone^{\oplus}(\sd)$ with $n_{i}\in\N$,
denote $n_{F}:=(n_{j})_{j\in F}\in\N^{F}$. Decompose $z$ as $z=\sum_{n\in N_{1}}c_{n}x^{m}\cdot y^{n}+\sum_{n\in N_{2}}c_{n}x^{m}\cdot y^{n}$,
where $c_{n}\in\kk$, $N_{1}=\{n|c_{n}\neq0,n_{F}=0\}$, and $N_{2}=\{n|c_{n}\neq0,n_{F}>0\}$.
Then $\frz_{F,m}^{\sd}z=\sum_{n\in N_{1}}c_{n}x^{m}\cdot y^{n}$.

Consider the pointed element $\mu_{k}^{*}(y^{n})$ of $\hLP(\sd')$.
We have $\deg^{\sd'}\mu_{k}^{*}(y^{n})=\deg^{\sd'}(y')^{h(n)}$, where
$h(n)\in\yCone(\sd')$. Note that $h(n)=h(n')$ if and only if $n=n'$.
Moreover, $n_{i}=h(n)_{i}$, $\forall i\neq k$, see \cite[(2.4)]{qin2019bases},
which implies $h(n)_{F}=n_{F}$. 

Note that we have a finite decomposition $\mu_{k}^{*}z=\sum_{n:c_{n}\neq0}\mu_{k}^{*}(c_{n}x^{m}\cdot y^{n})$
in $\hLP(\sd')$, see \cite[Lemma 3.3.7(2)]{qin2019bases}. Since
$m'$ is the unique $\prec_{\sd'}$-maximal degree of $z$, it equals
$\deg^{\sd'}x^{m}+\deg^{\sd'}y^{h(\tn)}$ for some $c_{\tn}\neq0$
such that $h(n)>h(\tn)$, $\forall n\neq\tn$, $c_{n}\neq0$. Denote
$N=N_{1}'\sqcup N_{2}'$ for $N_{1}'=\{n|c_{n}\neq0,h(n)_{F}=h(\tn)_{F}\}$
and $N_{2}'=\{n|c_{n}\neq0,h(n)_{F}>h(\tn)_{F}\}$. We have $\frz_{F,m'}^{\sd'}\mu_{k}^{*}z=\sum_{n\in N_{1}'}\mu_{k}^{*}(c_{n}x^{m}\cdot y^{n})$.
Moreover, $0\in N_{1}$ implies $h(\tn)_{F}=h(0)_{F}=0$. Consequently,
$N_{1}'=\{n|c_{n}\neq0,h(n)_{F}=0\}=\{n|c_{n}\neq0,n_{F}=0\}=N_{1}$.
The desired claim follows.

\end{proof}

The following result follows from Proposition \ref{prop:project_up_cl_alg}
and Lemma \ref{lem:freeze_mutation}.

\begin{Thm}\label{thm:freeze_compatibly_pointed}

If an element $z\in\upClAlg(t_{0})$ is compatibly pointed at all
seeds in $\Delta_{t_{0}}^{+},$ then $\frz_{F}z\in\upClAlg(\frz_{F}t_{0})$
is compatibly pointed at all seeds in $\Delta_{\frz_{F}t_{0}}^{+}$.

\end{Thm}

Theorem \ref{thm:freeze_compatibly_pointed}, \cite[Theorem 4.3.1]{qin2019bases},
and Remark \ref{rem:inclusion_frozen_subalg} imply the result below.

\begin{Thm}\label{thm:freeze_good_bases}

Assume that $\cS$ is a $\tropSet$-pointed $\kk$-basis of $\upClAlg(\sd)$.
If $\frz_{F}\sd$ is injective-reachable, then $\frz\cS$ is a $\tropSet$-pointed
$\kk$-basis of $\upClAlg(\frz_{F}\sd)$.

\end{Thm}

In \cite[Theorem 1.4.1]{muller2015existence}, Muller showed that,
if $\sd$ is injective-reachable, then $\frz_{F}\sd$ is injective-reachable
too. His proof also yields the following result. Details could be
found in \cite{qin2023freezing}.

\begin{Lem}\label{lem:freeze_cluster_deg}

If $m\in\cone(\sd)$ is the degree of a localized cluster monomial
for $\upClAlg(\sd)$, then $m$ is also the degree of a localized
cluster monomial for $\upClAlg(\frz_{F}\sd)$.

\end{Lem}

\begin{Thm}\label{thm:projection_quantum_cluster_monomial}

Assume that $\frz_{F}\sd$ is injective-reachable. Let $z$ be any
localized (quantum) cluster monomial of $\upClAlg(\sd)$. Then $\frz_{F}z$
is a localized (quantum) cluster monomial of $\upClAlg(\frz_{F}\sd)$.

\end{Thm}

\begin{proof}

By Lemma \ref{lem:freeze_cluster_deg}, $\deg^{\frz_{F}\sd}\frz_{F}z$
equals $\deg^{\frz_{F}\sd}u$ for some localized cluster monomial
of $\upClAlg(\frz_{F}\sd)$. By Theorem \ref{thm:freeze_compatibly_pointed},
$\frz_{F}z$ is compatibly pointed at $\Delta_{\frz_{F}\sd}^{+}$.
Then Lemma \ref{lem:tropicalized_same_cluster_monomial} implies $\frz_{F}z=u$.

\end{proof}

In \cite{qin2023freezing}, we will prove Theorem \ref{thm:projection_quantum_cluster_monomial}
without the assumption on $\frz_{F}\sd$. From now on, we will use
it in this generality.

\begin{Def}\label{def:nice-cluster-decomposition}

Take any $z\in\clAlg(\sd)$. Assume there is a finite decomposition
$z=\sum_{i}b_{i}z_{i}$, $b_{i}\in\kk$, such that $z_{i}$ have distinct
degrees in $\LP(\sd)$ and are products of cluster variables and $x_{j}^{-1}$,
$j\in I_{\fv}$. The decomposition is called a nice cluster decomposition
in $\sd$.

\end{Def}

\begin{Lem}\label{lem:freeze-ord-cluster-alg}

Take any $z\in\LP_{\preceq m}(\sd)$. If it has a nice cluster decomposition
in $\sd$, $\frz_{m}z$ belongs to $\clAlg(\frz_{F}\sd)$.

\end{Lem}

\begin{proof}

Denote the nice cluster decomposition by $z=\sum_{i}b_{i}z_{i}$,
$b_{i}\in\kk$. Since $z_{i}$ have distinct degrees, $z_{i}\in\LP_{\preceq m}(\sd)$,
see \cite[Lemma 4.1.1]{qin2019bases}. Denote $z_{i}=s_{1}*s_{2}*\cdots*s_{l}$,
where $s_{j}$ are cluster variables or $x_{j}^{-1}$, $j\in I_{\fv}$.
By Lemma \ref{lem:freezing_properties}, $\frz_{F,m}(z_{i})$ equals
$\frz_{F,\deg s_{1}}s_{1}*\cdots*\frz_{F,\deg s_{l}}s_{l}$ or $0$.
By Theorem \ref{thm:projection_quantum_cluster_monomial}, $\frz_{F,\deg s_{i}}s_{i}$
are localized cluster monomials. Therefore, $z_{i}\in\clAlg(\frz_{F}\sd)$
and the claim follows.

\end{proof}

Combining Theorem \ref{thm:freeze_good_bases} and Lemma \ref{lem:freeze-ord-cluster-alg},
we obtain the following useful result.

\begin{Cor}\label{cor:freezing-A-U}

Assume that $\upClAlg(\sd)$ has a $\tropSet$-pointed $\kk$-basis
$\cS=\{s_{m}|m\in\cone(\sd)\}$ such that $s_{m}$ are $[m]$-pointed.
If $\clAlg(\sd)=\upClAlg(\sd)$, $\frz_{F}\sd$ is injective-reachable,
and each $s_{m}$ has a nice cluster decomposition, then $\clAlg(\frz_{F}\sd)=\upClAlg(\frz_{F}\sd)=\Span_{\kk}\{\frz_{F,m}s_{m}|m\in\cone(\sd)\}$.
Moreover, each $\frz_{F,m}s_{m}$ has a nice cluster decomposition.

\end{Cor}

\begin{Thm}\label{thm:sub_cluster_triangular_basis}

Assume that $\sd$ is injective-reachable and $\upClAlg(\sd)$ possesses
the common triangular basis $\can=\{\can_{m}|m\in\cone(\sd)\}$, where
$\can_{m}$ are $[m]$-pointed. Then $\frz_{F}\can:=\{\frz_{F,m}\can_{m}|m\in\cone(\sd)\}$
is the common triangular basis for $\upClAlg(\frz_{F}\sd)$.

\end{Thm}

\begin{proof}

Since $\can$ is the common triangular basis, it contains all cluster
monomials of $\upClAlg(\sd)$. Note that any cluster monomials of
$\upClAlg(\frz_{F}\sd)$ must be a cluster monomial of $\upClAlg(\sd)$
with support in $I_{\ufv}\backslash F$, see Remark \ref{rem:inclusion_frozen_subalg}.
So $\frz\can$ contains all cluster monomials of $\upClAlg(\frz_{F}\sd)$.
We claim that $\frz\can$ is the triangular basis for $\upClAlg(\frz_{F}\sd)$
with respect to the seed $\frz_{F}\sd$. If so, it follows from \cite[Proposition 6.4.3, Theorem 6.5.3]{qin2020dual}
that $\frz_{F}\can$ is the common triangular basis for $\upClAlg(\frz_{F}\sd)$.

By Theorem \ref{thm:freeze_good_bases}, $\frz_{F}\can$ is a $\cone(\sd)$-pointed
basis for $\upClAlg(\frz_{F}\sd)$. Its bar-invariance is clear. We
have also seen that it contains the cluster monomials of $(\frz_{F}\sd)[1]$.
It remains to check the triangularity. Applying $\frz_{F,m+f_{i}}$
to (\ref{eq:triangular_basis_triangularity}) and using Lemma \ref{lem:freezing_properties},
we get $[x_{i}*\frz_{F,m}\can_{m}]^{\sd}=\can_{m+f_{k}}+\sum_{m'\prec_{\sd}m+f_{k}}b_{m'}\frz_{F,m+f_{k}}\can_{m'}$,
$b_{m}'\in\mm$, $\forall i$, where $\frz_{F,m+f_{i}}\can_{m'}$
is $0$ or $\frz_{F,m'}\can_{m'}$. So $\frz_{F}\can$ is the triangular
basis with respect to $\sd$. 

\end{proof}

\begin{Thm}\label{thm:freezing_common_triangular}

Let $\can=\{\can_{m}|m\in\cone(\sd)\}$ denote the common triangular
basis. Then, for any $\can_{m}$ and $k\in F$, when $d_{k}\in\N$
are large enough, we have
\begin{align*}
\frz_{F,m}\can_{m} & =x^{-\sum_{k\in F}d_{k}f_{k}}\cdot\can_{m+\sum_{k\in F}d_{k}f_{k}},
\end{align*}
and $x_{j}\cdot\can_{m+\sum_{k\in F}d_{k}f_{k}}\in\can$ for any $j\in F$.

\end{Thm}

We will prove Theorem \ref{thm:freezing_common_triangular} in the
end of this subsection.

\begin{Cor}\label{cor:tri-freezing-equal-restriction}

In the situation of Theorem \ref{thm:freezing_common_triangular},
the following claims are true.

(1) $\frz_{F}\can=\{x^{-p}\cdot b|b\in\can\cap\bUpClAlg,\supp b\cap F=\emptyset,p\in\N^{F\sqcup I_{\fv}}\}$.

(2) $\frz_{F}\can=\{x^{-p}\cdot b|b\in\can\cap\bUpClAlg(\frz_{F}\sd),p\in\N^{F\sqcup I_{\fv}}\}$.

(3) $\frz_{F}\can\cap\bUpClAlg(\frz_{F}\sd)=\can\cap\bUpClAlg(\frz_{F}\sd)$.

(4) $\can\cap\bUpClAlg(\frz_{F}\sd)=\{b\in\can\cap\bUpClAlg,\supp b\cap F=\emptyset\}$.

\end{Cor}

\begin{proof}

(1) Take any $x^{-p}\cdot b$ from the set on the right hand side.
We have $x^{-p}\cdot b=x^{-p}\cdot\frz_{F}b\in x^{-p}\cdot\frz_{F}\can$.
Since $\frz_{F}\can$ is the common triangular basis of $\upClAlg(\frz_{F}\sd)$
by Theorem \ref{thm:sub_cluster_triangular_basis}, we have $x^{-p}\cdot\frz_{F}\can=\frz_{F}\can$.
Thus $x^{-p}\cdot b\in\frz_{F}\can$.

Conversely, take any $\can_{m}\in\can$. By Theorem \ref{thm:freezing_common_triangular},
when $d$ is large enough, $\frz_{F}\can_{m}=x^{-d\sum_{k\in F}f_{k}}\cdot\can_{m+d\sum_{k\in F}f_{k}}$
with $\supp\can_{m+d\sum_{k\in F}f_{k}}\cap F=\emptyset$. Rewrite
$\frz_{F}\can_{m}=x^{-d\sum_{i\in F\sqcup I_{\fv}}f_{i}}\cdot\can_{m+d\sum_{i\in F\sqcup I_{\fv}}f_{i}}$.
We have $\can_{m+d\sum_{i\in F\sqcup I_{\fv}}f_{i}}=\frz_{F}\can_{m+d\sum_{i\in F\sqcup I_{\fv}}f_{i}}\in\frz_{F}\can$.
For $d$ large enough, $\nu_{i}(\frz_{F}\can_{m+d\sum_{i\in F\sqcup I_{\fv}}f_{i}})\geq0$,
$\forall i\in F\sqcup I_{\fv}$. So we can choose $b:=\can_{m+d\sum_{i\in F\sqcup I_{\fv}}f_{i}}\in\bUpClAlg$,
which satisfies $\supp b\cap F=(\supp\frz b)\cap F=\emptyset$. We
deduce that $\frz_{F}\can_{m}$ belongs to the set on the right. 

(2) In the proof of (1), we have shown $\frz_{F}\can\subset\{x^{-p}\cdot b|b\in\can\cap\bUpClAlg(\frz_{F}\sd),\supp b\cap F=\emptyset,p\in\N^{F\sqcup I_{\fv}}\}$.
Therefore, $\upClAlg(\frz_{F}\sd)=\Span\frz_{F}\can\subset\Span\{x^{-p}\cdot b|b\in\can\cap\bUpClAlg(\frz_{F}\sd),p\in\N^{F\sqcup I_{\fv}}\}\subset\upClAlg(\frz_{F}\sd)$.
The claim follows.

(3) The claim follows from (2).

(4) By (2), any $b\in\can\cap\bUpClAlg(\frz_{F}\sd)$ is contained
in $\frz_{F}\can$. It follows that $\supp b\cap F=\emptyset$. Note
that $\bUpClAlg(\frz_{F}\sd)\subset\bUpClAlg(\sd)$ by Proposition
\ref{prop:good-sub-up-cl-alg}. So $\can\cap\bUpClAlg(\frz_{F}\sd)$
is a subset of the set on the right. Conversely, for any $b$ from
the set on the right, we have $\frz_{F}b=b$, thus $b\in\upClAlg(\frz_{F}\sd)$.
For any $i\in F\sqcup I_{\fv}$, Lemma \ref{lem:f-bounds-d} implies
$\nu_{i}(\frz_{F}b)=\nu_{i}(b)\geq0$ since $i\notin\supp b$. We
deduce that $b=\frz_{F}b\in\bUpClAlg(\frz_{F}\sd)$. So we obtain
the inverse inclusion.

\end{proof}

Before proving Theorem \ref{thm:freezing_common_triangular}, let
us explore some general properties. Take any $k\in I$. Let $\cZ$
denote a $\kk$-linearly independent subset of $\LP(\sd)$. Assume
that, for any $z\in\cZ$, there is a finite decomposition 
\begin{align}
x_{k}*z & =\sum_{\tz_{j}\in q^{\frac{\Z}{2}}\cZ}b_{j}\tz_{j},\ b_{j}\in\Z.\label{eq:decomposition-Z}
\end{align}
\begin{Def}\label{def:stabilization}

Assume that, $\forall$$z\in\cZ$, the family of decompositions for
$d\in\N$ 
\begin{align}
x_{k}^{d}*z & =\sum_{\tz_{j}^{(d)}\in q^{\frac{\Z}{2}}\cZ}b_{j}^{(d)}\tz_{j}^{(d)},\ b_{j}^{(d)}\in\Z.\label{eq:decomposition-xk-z}
\end{align}
has the following property: when $d$ is large enough, $x_{k}^{-d}*\tz_{j}^{(d)}$
stabilizes, and $x_{k}*\tz_{j}^{(d)}\in q^{\Hf\Z}\cZ$ for any $z_{j}^{(d)}$
appearing. Then we say $\cZ$ has the stabilization property with
respect to $x_{k}$. We will also say $q^{\Z}\cZ$ and $q^{\frac{\Z}{2}}\cZ$
have this property.

\end{Def}

\begin{Lem}\label{lem:positive-to-stabilization}

If the Laurent coefficients of the elements of $\cZ$ and the decomposition
coefficients in (\ref{eq:decomposition-Z}) always belong to $\kk_{\geq0}$,
$\cZ$ has the stabilization property with respect to $x_{k}$.

\end{Lem}

\begin{proof}

Consider the decomposition (\ref{eq:decomposition-xk-z}) in $\Z[q^{\pm\Hf},(x_{i}^{\pm})_{i\in I}]$.
As $d$ grows, the sum $(\sum b_{j}^{(d)})|_{q\mapsto1}$ weakly increases.
Since the total number of Laurent monomials $q^{\alpha}x^{m}$ appearing
in $x_{k}^{d}*z$ is constant for varying $d$, the sum must stabilizes
when $d$ is large enough. In this case, for any $\tz_{j}^{(d)}$
appearing, we have $x_{k}*\tz_{j}^{(d)}=\tz_{j}^{(d+1)}$ for some
$\tz_{j}^{(d+1)}\in q^{\frac{\Z}{2}}\cZ$. The desired claim follows.

\end{proof}

We work at the quantum level by choosing $\kk=\Z[q^{\pm\Hf}]$ from
now on.

\begin{Def}\label{def:multiplication-triangularity}

Assume that the elements of $\cZ$ are pointed at distinct degrees.
We say $\cZ$ satisfies the triangularity condition with respect to
$x_{k}$ if, $\forall z\in\cZ$ , we have a finite decomposition
\begin{align}
[x_{k}*z] & =z_{0}+\sum_{z_{j}\in\cZ,\ \deg z_{j}\prec_{\sd}\deg z_{0}}b_{j}z_{j}\label{eq:triangular_set_decomposition}
\end{align}
such that $b_{j}\in q^{-\Hf}\Z[q^{-\Hf}]$ for $j\neq0$.

\end{Def}

\begin{Lem}\label{lem:stabilization-freezing}

Assume the elements of $\cZ$ are bar-invariant and pointed at distinct
degrees. If it has the stabilization property and satisfies triangularity
condition with respect to $x_{k}$, then we have $\frz_{\{k\}}z=[x_{k}^{-d}*z_{0}^{(d)}]$
for $d\in\N$ large enough, where $[x_{k}^{d}*z]=z_{0}^{(d)}+\sum_{\deg z_{j}^{(d)}\prec_{\sd}\deg z_{0}^{(d)}}b_{j}^{(d)}z_{j}^{(d)}$.

\end{Lem}

\begin{proof}

The case $k\in I_{\fv}$ is trivial. Now assume $k\in I_{\ufv}$.
By the stabilization property, we can take $d$ large enough such
that $[x_{k}*z_{0}^{(d)}]=z_{0}^{(d+1)}$. Then the bar-invariance
of $z_{0}^{(d+1)}$ implies that $z_{0}^{(d)}$ $q$-commutes with
$x_{k}$, i.e., $\frz_{\{k\},\deg z_{0}^{(d)}}z_{0}^{(d)}=z_{0}^{(d)}$.
Take any $z_{i}^{(d)}$ such that $\deg z_{i}^{(d)}$ is $\prec_{\sd}$-maximal
in $\{\deg z_{j}^{(d)}|j\neq0,b_{j}\neq0\}$. Then the Laurent coefficient
of $x^{\deg z_{i}^{(d)}}$ in $[x_{k}^{d}*z]$, contributed from $z_{0}^{(d)}$
and $b_{i}^{(d)}z_{i}^{(d)}$, is not bar-invariant since $b_{i}^{(d)}\in q^{-\Hf}\Z[q^{-\Hf}]$.
So we must have $\deg z_{i}^{(d)}=\deg z_{0}^{(d)}+p^{*}n^{(i)}$
with $n_{k}^{(i)}>0$. We deduce that, for any $1\leq j\leq s$, $\deg z_{j}^{(d)}=\deg z_{0}^{(d)}+\tB n^{(j)}$
with $n_{k}^{(j)}>0$. Applying $\frz_{\deg z_{0}^{(d)}}=\frz_{\deg z+df_{k}}$
to both sides of the decomposition of $[x_{k}^{d}*z]$, we obtain
$\frz_{\deg z_{0}^{(d)}}z_{j}^{(d)}=0$, $\forall j\neq0$, see Lemma
\ref{lem:freezing_properties}. Thus $\frz_{\deg z_{0}^{(d)}}[x_{k}^{d}*z]=z_{0}^{(d)}$
and the claim follows.

\end{proof}

\begin{Lem}\label{lem:triangular-stabilization}

Assume that $\cZ$ satisfies triangularity condition with respect
to $x_{k}$ and the elements of $\cZ$ are pointed at distinct degrees
and copointed at distinct codegrees. Then it has the stabilization
property with respect to $x_{k}$.

\end{Lem}

\begin{proof}

The case $k\in I_{\fv}$ is trivial. Now assume $k\in I_{\ufv}$ and
$z\in\cZ$. In the finite decomposition (\ref{eq:triangular_set_decomposition}),
we have $\codeg z+f_{k}\preceq_{\sd}\codeg z_{j}$ and the equality
holds for a unique $z_{s}$, since $z_{j}$ have distinct codegrees,
see \cite[Lemma 3.1.10(iii)]{qin2017triangular}. 

First assume $z_{s}=z_{0}$. Then the Laurent monomial $x^{\codeg z+f_{k}}$
of $[x_{k}*z]$ has coefficient $1$, i.e., $\lambda(f_{k},\deg z)=\lambda(f_{k},\codeg z)$.
So $(\suppDim z)_{k}=0$. Thus $x_{k}$ and $z$ $q$-commute. Applying
the bar-involution to (\ref{eq:triangular_set_decomposition}), we
deduce that $[x_{k}*z]=z_{0}$. 

Otherwise, we always have $\codeg z+f_{k}\preceq_{\sd}\codeg z_{j}$
and $\deg z_{j}\preceq_{\sd}\deg z_{0}+f_{k}$ and the equalities
do not hold at the same time. So all $z_{j}$ appearing satisfy $\suppDim z_{j}<\suppDim z$.
By induction on the support dimension, we can assume that $z_{j}$
satisfy the stabilization property. Then we deduce that $z$ satisfies
this property as well.

\end{proof}

\begin{proof}[Proof of Theorem \ref{thm:freezing_common_triangular}]

Note that the common triangular basis elements are copointed at distinct
codegrees \cite{qin2020analog}. Theorem \ref{thm:freezing_common_triangular}
follows from Lemmas \ref{lem:triangular-stabilization} and \ref{lem:stabilization-freezing}.

\end{proof}

\section{Base changes\label{sec:coeff-change}}

We interpret changes in the frozen part of cluster algebras as base
changes in schemes. Many of our claims will be evident when we apply
base change later, but we will provide a general treatment for the
sake of completeness.

\subsection{Similarity and the correction technique\label{subsec:Similarity-and-correction}}

Let $\sd$ denote any seed. Recall that, in the quantum case, we have
$\lambda(f_{i},p^{*}e_{k})=-\delta_{ik}\diag_{k}$ for some $\diag_{k}\in\N_{>0}$.

\begin{Def}[{\cite{Qin12}\cite{qin2017triangular}}]

Let there be given two seeds $\sd,\sd'$ not necessarily related by
a mutation sequence, $I'_{\ufv}=I_{\ufv}$, and $\sigma$ a permutation
on $I_{\ufv}$. We say $\sd$ and $\sd'$ are similar (up to $\sigma$)
if $b'_{\sigma i,\sigma j}=b_{ij}$ and $d'_{\sigma i}=d_{i}$, $\forall i,j\in I_{\ufv}$.
In the quantum case, we further require that there exists some $\rho\in\Q_{>0}$
such that $\diag'_{\sigma k}=\rho\diag_{k}$, $\forall k\in I_{\ufv}$.

\end{Def}

Let $\sigma$ denote any permutation of $I$ such that $\sigma I_{\ufv}=I_{\ufv}$.
If we are given a permutation $\sigma$ on $I_{\ufv}$, extend $\sigma$
by letting it acts on $I_{\fv}$ trivially. We define the permuted
seed $\sigma\sd$ by relabeling the vertices using $\sigma$, such
that $b_{\sigma i,\sigma j}(\sigma\sd)=b_{ij}$, $x_{\sigma i}(\sigma\sd)=x_{i}$,
$d_{\sigma i}(\sigma\sd)=d_{i}$. In the quantum case, we endow $\sigma\sd$
with the bilinear form $\sigma\lambda$ such that $\sigma\lambda(\sigma m,\sigma g):=\lambda(m,g)$,
where $(\sigma m)_{\sigma i}:=m_{i}$ $\forall m=(m_{i})\in\Z^{I}$.
Apparently, $\sd$ and $\sigma\sd$ are similar.

From now on, let $\sd$ and $\sd'$ denote two seeds similar up to
$\sigma$. Modifying $\sd'$ by permuting its vertices if necessary,
we can assume $\sigma$ is trivial. Note that $\seq\sd$ and $\seq\sd'$
are similar, $\forall\seq$, see \cite{qin2017triangular}. 

When we treat quantum cluster algebras later in this paper, it is
often sufficient to consider the case $\rho=1$. For completeness,
let us treat the general case $\rho=\frac{d_{\rho}'}{d_{\rho}}$ for
coprime $d_{\rho},d_{\rho}'\in\N_{>0}$. Recall $\kk=\Z[q^{\pm\Hf}]$.
Define $\kk':=\Z[q^{\pm\frac{1}{2d_{\rho}}}]$ and $\kk'':=\Z[q^{\pm\frac{1}{2d_{\rho}'}}]$.
Then $\kk'=\oplus_{j=0}^{d_{\rho}-1}\kk\cdot v_{j}$ is a free $\kk$-module,
where $v_{j}:=q^{\frac{j}{2d_{\rho}}}$. Define $\LP'_{\kk'}$, $\hLP'_{\kk'},$
$\clAlg'_{\kk'}$, and $\upClAlg'_{\kk'}$ associated to $\sd'$ using
$\kk'$ instead of $\kk$ in their constructions. We have $\LP'_{\kk'}=\LP'\otimes\kk'$,
$\hLP'_{\kk'}=\hLP'\otimes\kk'$, and $\clAlg'_{\kk'}=\clAlg'\otimes\kk'$.

\begin{Prop}\label{prop:change-base-up-alg}

We have $\upClAlg'_{\kk'}=\upClAlg'\otimes\kk'$.

\end{Prop}

\begin{proof}

The skew-field of fractions $\cF(\sd')_{\kk'}$ for $\LP(\sd')_{\kk'}$
contains $\cF(\sd')\otimes\kk'$. Take any $\ssd=\seq\sd'\in\Delta_{\sd'}^{+}$.
We have $\seq^{*}:\cF(\ssd)_{\kk'}\simeq\cF(\sd')_{\kk'}$. Since
$\seq^{*}(v_{j})=v_{j}$, $(\seq^{*}\LP(\ssd))\otimes\kk'$ equals
$\seq^{*}\LP(\ssd)_{\kk'}$ in $\cF(\sd')_{\kk'}$.

For any $\ssd'=\seq'\sd'\in\Delta_{\sd'}^{+}$, we have $(\seq^{*}\LP(\ssd)\cdot v_{j})\cap((\seq')^{*}\LP(\ssd')\cdot v_{j'})\subset(\cF(\sd')\cdot v_{j})\cap(\cF(\sd')\cdot v_{j'})=0$
whenever $j\neq j'$. Therefore, $\upClAlg'_{\kk'}=\cap_{\ssd=\seq\sd'\in\Delta_{\sd'}^{+}}(\seq^{*}\LP(\ssd)\otimes\kk')=\oplus_{j}(\cap_{\ssd}(\seq^{*}\LP(\ssd)\cdot v_{j}))=\oplus_{j}(\cap_{\ssd}\seq^{*}\LP(\ssd))\cdot v_{j}=\upClAlg'\otimes\kk'$.

\end{proof}

\begin{Lem}\label{lem:coefficent-kk}

For any subset $\cZ$ of $\hLP$, $\hLP\cap\Span_{\kk'}\cZ$ equals
$\Span_{\kk}\cZ$ in $\hLP\otimes\kk'$.

\end{Lem}

\begin{proof}

Recall $\hLP\otimes\kk'=\oplus_{j=0}^{d_{\rho}-1}\hLP\cdot v_{j}$,
$v_{j}=q^{\frac{j}{2d_{\rho}}}$. Obviously, $\hLP\cap\Span_{\kk'}\cZ$
contains $\Span_{\kk}\cZ$. For any $w\in\Span_{\kk'}\cZ$, we have
$w=\sum_{i}(\sum_{j=0}^{d_{\rho}-1}c_{ij}v_{j})z_{i}$ for $z_{i}\in\cZ$,
$c_{ij}\in\kk$. Then $w=\sum_{j}w_{j}$ for $w_{j}:=\sum_{i}c_{ij}v_{j}z_{i}\in\hLP\cdot v_{j}$.
If $w\in\hLP$, then $w=w_{0}$. Therefore, $\Span_{\kk}\cZ$ contains
$\hLP\cap\Span_{\kk'}\cZ$.

\end{proof}

\begin{Def}[{\cite{Qin12}}]\label{def:similar-elements}

Take any $z=x^{m}\cdot(1+\sum_{n\in\N^{I_{\ufv}}}c_{n}y^{n})\in\hLP$,
$c_{n}\in\Z[q^{\pm\Hf}]$. The elements $z'=(x')^{m'}\cdot(1+\sum_{n\in\N^{I_{\ufv}}}c_{n}'(y')^{n})\in\hLP'_{\kk'}$,
$c_{n}'\in\kk'$, are said to be similar to $z$ if $\pr_{I_{\ufv}}m'=\pr_{I_{\ufv}}m$
and $c'_{n}=c_{n}|_{q^{\Hf}\mapsto q^{\Hf\rho}}$. If $q^{\Hf}=1$,
we require $c'_{n}=c_{n}$.

\end{Def}

Note that if $z$ is similar to $z'$, $z'$ is similar to $z$ as
well. Moreover, when $z$ is given, any $z'$ similar to $z$ is determined
by $\deg^{\sd'}z'$. 

Let us extend the construction of similar elements. $\forall m\in\cone(\sd)$,
define a $\widehat{\kk[N_{\ufv}]}$-bimodule and a $\kk[N_{\ufv}]$-bimodule:
\begin{align}
\hV_{m}:=x^{m}\cdot\widehat{\kk[N_{\ufv}]},\ V_{m}:=x^{m}\cdot\kk[N_{\ufv}].
\end{align}
Then $\hV_{m}\cap\hV_{h}\neq0$ if and only if $h\in m+p^{*}N_{\ufv}$,
which implies $\hV_{m}=\hV_{h}$. Therefore, $\hLP=\oplus_{m_{i}}\hV_{m_{i}}$,
where $m_{i}$ are any chosen representatives of $\cone/p^{*}N_{\ufv}$.
Similarly, we have $\LP=\oplus_{m_{i}}V_{m_{i}}$.

Choose $m'\in\cone(\sd')$ such that $\pr_{I_{\ufv}}m=\pr_{I_{\ufv}}m'$.
Similarly define $\hV'_{m'}\subset\hLP'$ for $\sd'$. We have the
$\Z$-algebra isomorphism $\var:\widehat{\kk''[N_{\ufv}]}\rightarrow\widehat{\kk'[N'_{\ufv}]}$
such that $\var(y^{n}):=(y')^{n}$, $\var(q^{\frac{1}{2d_{\rho}'}}):=q^{\frac{1}{2d_{\rho}'}\rho}=q^{\frac{1}{2d_{\rho}}}$.
It induces a bijective $\Z$-linear map 
\begin{align}
\var & :=\var_{m',m}:\hV_{m}\otimes\kk''\simeq(\hV'_{m'})\otimes\kk',
\end{align}
such that $\var(q^{\frac{1}{2d_{\rho}'}}):=q^{\frac{1}{2d_{\rho}}}$,
$\var(x^{m}\cdot y^{n}):=(x')^{m'}\cdot(y')^{n}$. We have $\var(a*z*b)=\var(a)*\var(z)*\var(b)$,
$\forall a,b\in\widehat{\kk''[N_{\ufv}]}$, $z\in V_{m}$.

We observe that, for any $(m+\tB n)$-pointed element $z\in\hLP$,
$n\in N_{\ufv}$, the image $\var_{m',m}(z)$ is $(m'+\tB'n)$-pointed
and similar to $z$ in the sense of Definition \ref{def:similar-elements}.

In general, any $z\in\hLP$ belongs to some $\oplus_{m_{i}}\hV_{m_{i}}$.
If $z$ is further contained in $\LP$, then $z\in\oplus_{m_{i}}V_{m_{i}}$.
Denote $z=\sum_{i}z_{i}$ with $z_{i}\in\hV_{m_{i}}$. By choosing
$m'_{i}$ such that $\pr_{I_{\ufv}}m'_{i}=\pr_{I_{\ufv}}m_{i}$, we
obtain an element $\var_{\um',\um}(z):=\sum_{i}\var_{m'_{i},m_{i}}(z_{i})$,
where $\um=(m_{i})_{i}$ and $\um'=(m'_{i})_{i}$. We could say $\var_{\um',\um}(z)$
is similar to $z$.

\begin{Lem}\label{lem:component-univ-Laurent}

Assume $z=\sum_{i}z_{i}$, $z_{i}\in V_{m_{i}}$, $V_{m_{i}}\cap V_{m_{j}}=0$
when $i\neq j$. If $z\in\upClAlg$, then $z_{i}\in\upClAlg$.

\end{Lem}

\begin{proof}

It suffices to show $z_{i}\in\LP(\ssd)$ for any seed $\ssd=\seq^{-1}\sd$.
Following \cite[Section 3.3]{qin2019bases}, define the bijective
linear map $\psi:=\psi_{\ssd,\sd}:\cone\rightarrow\cone(\ssd)$ such
that $\psi(\deg^{\sd}x_{i})=\deg^{\ssd}\seq^{*}x_{i}$, $\forall i$.
Then, $\forall m\in\cone$, $\seq^{*}x^{m}$ is a $\psi m$-pointed
element in $\hLP(\ssd)$. Moreover, we can deduce from the mutation
rule of $y$-variables (\cite[(2.4)]{qin2019bases}) that $\psi$
restricts to a bijection from $p^{*}N_{\ufv}$ to $p^{*}N_{\ufv}(\ssd)$.

We have $\seq^{*}z=\sum_{i}\seq^{*}z_{i}$ in $\hLP(\ssd)$, where
$\seq^{*}z_{i}\in\hV(\ssd)_{h_{i}}$, $h_{i}=\psi m_{i}$. Moreover,
if $h_{i}\in h_{j}+p^{*}N_{\ufv}(\ssd)$, we will have $m_{i}\in m_{j}+p^{*}N_{\ufv}$
by applying $\psi^{-1}$. In this case, we obtain $V_{m_{i}}=V_{m_{j}}$
and thus $i=j$. Therefore, $\hV(\ssd)_{h_{i}}\cap\hV(\ssd)_{h_{j}}=0$
if $i\neq j$. Combining with $\seq^{*}z\in\LP(\ssd)$, we obtain
$\seq^{*}z_{i}\in V(\ssd)_{h_{i}}\subset\LP(\ssd)$.

\end{proof}

We slightly generalize the correction technique in \cite{Qin12,qin2017triangular}
as below.

\begin{Thm}\label{thm:correction}

Given $m_{s}$-pointed elements $z_{s}\in\hLP$ and $m'_{s}$-pointed
elements $z'_{s}\in\hLP'_{\kk'}$, for $r\in\N,1\leq s\leq r$, such
that $z'_{s}$ are similar to $z_{s}$. We do not assume $m_{s}$
to be distinct. Take any $z\in V_{m}$ and $z'\in(V'_{m'})\otimes\kk'$
such that $\pr_{I_{\ufv}}m=\pr_{I_{\ufv}}m'$, $z'=\var_{m',m}(z)$. 

(1) If $z=[z_{1}*\cdots*z_{r}]^{\sd}$, then $z'=p'\cdot[z'_{1}*\cdots*z'_{r}]^{\sd'}$
for $p'$ in the frozen group $\frGroup'$, such that $\deg^{\sd'}p'=m'-\sum_{s=1}^{r}m'_{s}$.

(2) If $z=\sum_{s=1}^{r}b_{s}z_{s}$, $b_{s}\in\kk$, we have $z=\sum_{s\in J}b_{s}z_{s}$
for $J=\{s|\exists n_{s}\in\Z^{I_{\ufv}},\ m_{s}=m+\tB n_{s}\}$.
Moreover, we have $z'=\sum_{s\in J}b'_{s}p'_{s}\cdot z'_{s}$, where
$b'_{s}=b_{s}|_{q^{\Hf}\mapsto q^{\frac{\rho}{2}}}$ and $p'_{s}=(x')^{m'+\tB'n_{s}-m'_{s}}\in\frGroup'$.

\end{Thm}

\begin{proof}

(1) follows from direct calculation, see \cite[Theorem 9.1.2]{Qin12}.

(2) The first claim follows from the decomposition of $z$ in $\oplus_{g}\hV_{g}$
. Using the $\Z$-linear map $\var:=\var_{m',m}:V_{m}\rightarrow V'_{m'}\otimes\kk'$,
we obtain $z'=\sum_{s\in J}\var(b_{s}z_{s})=\sum_{s\in J}b'_{s}\var(z_{s})$,
where $\var(z_{s})$ is similar to $z_{s}$ with $\deg\var(z_{s})=m'+\tB'n_{s}$.
The claim follows by rewriting $\var(z_{s})=p'_{s}\cdot z'_{s}$.

\end{proof}

By \cite{FominZelevinsky07}\cite{Tran09}\cite{gross2018canonical},
for any mutation sequence $\seq$ and $i\in I$, $x_{i}(\seq\sd)\in\LP$
and $x_{i}(\seq\sd')\in\LP'\subset\LP'_{\kk'}$ are similar.

\begin{Lem}\label{lem:similar-alg-element}

Let $z\in\LP$ and $z'\in\LP'_{\kk'}$ be similar pointed elements.

(1) If $z\in\clAlg$, then $z'\in\clAlg'_{\kk'}$. (2) If $z\in\upClAlg$,
then $z'\in\upClAlg'_{\kk'}$.

\end{Lem}

\begin{proof}

(1) $z$ is a polynomial of cluster variables of $\clAlg$ and $x_{j}^{-1}$,
$j\in I_{\fv}$, with coefficients in $\kk$. Then Theorem \ref{thm:correction}
implies that $z'$ is a polynomial of the similar cluster variables
of $\clAlg'$ and $(x'_{j})^{-1}$, $j\in I'_{\fv}$, with coefficients
in $\kk'$. The claim follows.

(2) For any mutation sequence $\seq$, denote $\ssd=\seq\sd$ and
$\ssd'=\seq\sd'$. We want to show $z'\in\LP(\ssd')_{\kk'}$. Since
$z\in\LP(\ssd)$, we have a finite decomposition $z=\sum_{m}b_{m}[x(\ssd)^{-d}*x(\ssd)^{m}]^{\sd}$
in $\hLP(\sd)$, where $d,m\in\N^{I}$ and $b_{m}\in\kk$. Note that
the normalized products of similar elements $[x(\ssd)^{-d}*x(\ssd)^{m}]^{\sd}\in\hLP(\sd)$
and $[x(\ssd')^{-d}*x(\ssd')^{m}]^{\sd'}\in\hLP(\sd')$ are again
similar. By Theorem \ref{thm:correction}, we have a finite decomposition
$z'=\sum_{m}b'_{m}p'_{m}\cdot[x(\ssd')^{-d}*x(\ssd')^{m}]^{\sd'}$
in $\hLP(\sd')$ where $b'_{m}\in\kk'$ and $p'_{m}\in\frGroup'$.
Therefore, $z'\in\LP(\ssd')_{\kk'}$.

\end{proof}

\subsection{Base changes\label{subsec:Base-changes}}

Let $\sd$ and $\sd'$ denote two similar seeds. Given any subset
$\cZ$ of $\LP(\sd)$ such that its elements are pointed and it is
closed under the $\frGroup$ commutative multiplication, i.e., $x_{j}^{\pm}\cdot\cZ\subset\cZ,\forall j\in I_{\fv}.$
In practice, $\cZ$ will be $\cone$-pointed.

Let $\cZ'$ denote the set of elements in $\LP(\sd')_{\kk'}$ similar
to the elements of $\cZ$. Then $\cZ'$ has pointed elements and is
closed under the $\frGroup'$ commutative multiplication. Let $\cZ''$
denote the set of the elements in $\LP(\sd)_{\kk''}$ that are similar
to those of $\cZ'$. 

\begin{Lem}\label{lem:inverse-construction-similar}

(1) $\cZ''$ coincides with $\cZ$.

(2) $\cZ$ is $\kk$-linearly independent in $\LP(\sd)$ if and only
if it is $\kk''$-linearly independent in $\LP(\sd)_{\kk''}$. 

(3) $\cZ$ is $\kk''$-linearly independent if and only if $\cZ'$
is $\kk'$-linearly independent .

\end{Lem}

\begin{proof}

For any $m''$-pointed $z''\in\cZ''$, there exists some $m'$-pointed
$z'\in\cZ'$ similar to it. Similarly, there exists some $m$-pointed
$z\in\cZ$ similar to $z'$. Note that $\pr_{I_{\ufv}}m''=\pr_{I_{\ufv}}m'=\pr_{I_{\ufv}}m$.
Denote $p=x^{m''-m}$. Then $p\cdot z\in\cZ$ and it coincides with
$z''$, since they are both similar to $z'$ and have the same degree.

Claim (2) can be proved using similar arguments as those in the proof
of Lemma \ref{lem:coefficent-kk}. Finally, Claim (3) follows from
Theorem \ref{thm:correction}(2) and Claim (1).

\end{proof}

\begin{Lem}\label{lem:similar-cluster-basis}

If $\cZ$ is a $\kk$-basis of $\clAlg$, then $\cZ'$ is a $\kk'$-basis
of $\clAlg'_{\kk'}$. 

\end{Lem}

\begin{proof}

Lemma \ref{lem:similar-alg-element} and Lemma \ref{lem:inverse-construction-similar}
imply $\cZ'\subset\clAlg'_{\kk'}$ and $\cZ'$ is linearly independent.

Let $z'$ denote any product of cluster variables and $(x')_{j}^{-1}$,
$j\in I'_{\fv}$ of $\clAlg'$. It is similar to a product $z$ of
the similar cluster variables of $\clAlg$ and $x_{j}^{-1}$, $j\in I_{\fv}$,
which belongs to $\Span_{\kk}\cZ$. By Theorem \ref{thm:correction},
$z'$ belongs to $\Span_{\kk'}\cZ'$. We deduce that $\Span_{\kk'}\cZ'=\clAlg'_{\kk'}$.

\end{proof}

\begin{Lem}\label{lem:similar-basis}

If $\cZ$ is a $\kk$-basis of $\upClAlg$, then $\cZ'$ is a $\kk'$-basis
of $\upClAlg'_{\kk'}$.

\end{Lem}

\begin{proof}

Lemma \ref{lem:similar-alg-element} implies $\cZ'\subset\upClAlg'_{\kk'}$.
Lemma \ref{lem:inverse-construction-similar} implies that $\cZ'$
is linearly independent. By Lemma \ref{lem:component-univ-Laurent},
it suffices to show that, $\forall m'\in\cone(\sd')$, any $u'\in\upClAlg'_{\kk'}\cap(V'_{m'}\otimes\kk')$
is contained in $\Span_{\kk'}\cZ'$.

Choose any $m\in\cone(\sd)$ such that $\pr_{I_{\ufv}}m=\pr_{I_{\ufv}}m'$.
Recall that we have bijective $\Z$-linear maps $\var_{m,m'}:V'_{m'}\otimes\kk'\rightarrow V_{m}\otimes\kk''$
and $\var_{m',m}:V_{m}\otimes\kk''\rightarrow V'_{m'}\otimes\kk'$
as before, such that $\var_{m,m'}$ and $\var_{m',m}$ are inverse
to each other. Denote $u:=\var_{m,m'}(u')\in V_{m}\otimes\kk''$.
Since $u'\in\upClAlg'_{\kk'}$, we have $u\in\upClAlg_{\kk''}$ by
Lemma \ref{lem:similar-alg-element}. In addition, $\upClAlg_{\kk''}$
equals $\upClAlg\otimes\kk''$ by Proposition \ref{prop:change-base-up-alg},
which has the $\kk''$-basis $\cZ$. So $u\in\Span_{\kk''}\cZ$. Therefore,
$u'=\var_{m',m}(u)$ is contained in $\Span_{\kk'}\cZ'$ by Theorem
\ref{thm:correction}.

\end{proof}

\begin{Lem}\label{lem:inverse-condition-basis}

Let $\alg$ be $\clAlg$ or $\upClAlg$. If $\cZ'$ is a $\kk'$-basis
of $\alg'_{\kk'}$, $\cZ$ is a $\kk$-basis of $\alg$.

\end{Lem}

\begin{proof}

$\cZ''$ is a $\kk''$-basis of $\alg_{\kk''}$ by Lemmas \ref{lem:similar-cluster-basis}
and \ref{lem:similar-basis}. Note that $\alg_{\kk''}=\alg\otimes\kk''$
(Proposition \ref{prop:change-base-up-alg}). By Lemmas \ref{lem:inverse-construction-similar}
and \ref{lem:coefficent-kk}, $\cZ$ equals $\cZ''$ and it is a $\kk$-basis
of $\alg$.

\end{proof}

Lemma \ref{lem:similar-cluster-basis}, Lemma \ref{lem:similar-basis},
and Theorem \ref{thm:correction} imply the following.

\begin{Cor}\label{cor:base-change-A-U}

(1) If $\clAlg=\upClAlg=\Span_{\kk}\cZ$, then $\clAlg'_{\kk'}=\upClAlg'_{\kk'}=\Span_{\kk'}\cZ'$. 

(2) If the elements of $\cZ$ have nice cluster decompositions in
$\sd$, so do those of $\cZ'$ in $\sd'$.

\end{Cor}

From now on, assume there is a $\Z$-algebra homomorphism $\var:\LP_{\kk''}\rightarrow\LP'_{\kk'}$
such that $\var(x_{k})=x'_{k}\cdot p'_{k}$, $\var(y_{k})=y_{k}'$,
$\var(q^{\frac{1}{2d_{\rho}'}})=q^{\frac{1}{2d_{\rho}}}$, for $k\in I_{\ufv}$,
$p'_{k}\in\frGroup'$, and $\var(\frGroup)\subset\frGroup'$, called
a variation map following \cite{kimura2022twist} (see \cite{fraser2016quasi}
for $\kk=\Z$).

Note that $\LP_{\kk''}$ is a free $\frRing_{\kk''}$-bimodule via
multiplication. Similarly, $\frRing'_{\kk'}$ is a $\frRing'_{\kk'}$-bimodule,
and thus an $\frRing_{\kk''}$-bimodule via $\var$. Introduce the
left $\frRing'_{\kk'}$-module homomorphism 
\begin{align}
\varphi & :\frRing'_{\kk'}\otimes_{\frRing_{\kk''}}\LP_{\kk''}\simeq\LP'_{\kk'},
\end{align}
such that $\varphi(p'\otimes x^{m}):=p'*\var(x^{m})$, $\forall m\in\Z^{I},p'\in\frGroup'$.
Then $\varphi$ is an isomorphism.

\begin{Def}\label{def:base-change}

$\varphi$ is called the base change from $\LP_{\kk''}$ to $\LP'_{\kk'}$.

\end{Def}

The multiplication on $\LP'_{\kk'}$ induces that on $\frRing'_{\kk'}\otimes_{\frRing_{\kk''}}\LP_{\kk''}$
such that, $\forall h\in\Z^{I}$, $r'\in\frGroup'$, $(p'\otimes x^{m})*(r'\otimes x^{h})=q^{\lambda'(\deg^{\sd'}(\var x^{m}),\deg^{\sd'}r')}(p'*r')\otimes(x^{m}*x^{h}).$
Then $\varphi$ is a $\Z$-algebra isomorphism. It is a $\kk$-algebra
isomorphism when $\rho=1$.

Note that $\cZ'=\frGroup'\cdot\var(\cZ)=\{[p'*\var(z)]^{\sd'}\ |\ \forall z\in\cZ,p'\in\frGroup'\}$,
also denoted as $\cZ'=[\varphi(\frGroup'\otimes_{\frGroup}\cZ)]^{\sd'}$.
We say that $\cZ'$ is obtained from $\cZ$ by the base change $\varphi$.

\begin{Thm}\label{thm:cluster-base-change}

$\varphi$ restricts to a $\Z$-algebra isomorphism $\frRing'_{\kk'}\otimes_{\frRing_{\kk''}}\clAlg_{\kk''}\simeq\clAlg'_{\kk'}$.

\end{Thm}

\begin{proof}

Let $p*z_{1}*z_{2}*\cdots*z_{l}$ denote any product of unfrozen cluster
variables of $\clAlg$ and $p\in\frGroup$. Take any $p'\in\frGroup'$.
Then $\var(p'\otimes p*z_{1}*\cdots*z_{l})=p'*\var(p)*\var(z_{1})*\cdots*\var(z_{l})$.
It is contained in $\clAlg'$ since $\var(z_{i})$ are localized cluster
monomials. Therefore, $\varphi(\frRing'_{\kk'}\otimes_{\frRing_{\kk''}}\clAlg_{\kk''})\subset\clAlg'_{\kk'}$. 

Conversely, take any $p'\in\frGroup'$ and any product $z_{1}'*\cdots*z_{l}'$
of unfrozen cluster variables in $\clAlg'$. Let $z_{i}$ denote the
cluster variables of $\clAlg$ similar to $z'_{i}$. Then Theorem
\ref{thm:correction} implies that $p'*z_{1}'*\cdots*z_{l}'=q^{\alpha}r'*\varphi(1\otimes z_{1}*\cdots*z_{l})=\varphi(q^{\alpha}r'\otimes z_{1}*\cdots*z_{l})$
for some $r'\in\frGroup'$, $\alpha\in\frac{\Z}{2d_{\rho}}$. Therefore,
$\clAlg'_{\kk'}\subset\varphi(\frRing'_{\kk'}\otimes_{\frRing_{\kk''}}\clAlg_{\kk''})$.

\end{proof}

\begin{Thm}\label{thm:up-cluster-base-change}

If $\cZ$ is a $\kk''$-basis of $\upClAlg_{\kk''}$ or $\cZ'$ is
a $\kk'$-basis of $\upClAlg'_{\kk'}$, then $\varphi$ restricts
to a $\Z$-algebra isomorphism $\frRing'_{\kk'}\otimes_{\frRing_{\kk''}}\upClAlg_{\kk''}\simeq\upClAlg'_{\kk'}$.

\end{Thm}

\begin{proof}

Note that $\cZ$ is a $\kk$-basis of $\upClAlg$ and $\cZ'$ is a
$\kk'$-basis of $\upClAlg'_{\kk'}$ by Proposition \ref{prop:change-base-up-alg}
and Lemmas \ref{lem:coefficent-kk}, \ref{lem:similar-basis} and
\ref{lem:inverse-condition-basis}. The claim follows from the equality
$\cZ'=\frGroup'\cdot\var(\cZ)$.

\end{proof}

If $\kk=\Z$, we obtain a base change of schemes $\Spec\upClAlg'=\Spec\frRing'\times_{\Spec\frRing}\Spec\upClAlg$.

\begin{Rem}

Let $\alg$ denote $\clAlg$ or $\upClAlg$. We claim that the construction
of $\varphi:\frRing'_{\kk'}\otimes_{\frRing_{\kk''}}\alg_{\kk''}\simeq\alg'_{\kk'}$
is independent of the choice of $\sd$. To see this, take any $k\in I_{\ufv}$,
denote $\ssd=\mu_{k}\sd$, $\ssd'=\mu_{k}\sd'$. Then $\var$ induces
a new variation map $\var_{\ssd}:\LP(\ssd)_{\kk''}\rightarrow\LP(\ssd')_{\kk'}$
such that $\var_{\ssd}|_{\alg(\ssd)_{\kk''}}$ coincides with the
composition $\alg(\ssd)_{\kk''}\xrightarrow{\mu_{k}^{*}}\alg(\sd)_{\kk''}\xrightarrow{\var}\alg(\sd')_{\kk'}\xrightarrow{(\mu_{k}^{*})^{-1}}\alg(\ssd')_{\kk'}$.
For more details, see \cite[Section 3.3 and Proposition 3.1.4]{kimura2022twist},
for which we replace $\upClAlg$ by $\alg$ and drop the condition
$\sd'\in\Delta_{\sd}^{+}$ without loss of generality. We denote $\mu_{k}^{*}:\frRing'_{\kk'}\otimes_{\frRing_{\kk''}}\alg(\ssd)_{\kk''}\xrightarrow{\sim}\frRing'_{\kk'}\otimes_{\frRing_{\kk''}}\alg(\sd)_{\kk''}$
such that $\mu_{k}^{*}(p'\otimes z):=\mu_{k}^{*}(p')\otimes\mu_{k}^{*}(z)=p'\otimes\mu_{k}^{*}(z)$. 

Introduce $\varphi_{\ssd}:\frRing'_{\kk'}\otimes_{\frRing_{\kk''}}\alg(\ssd)_{\kk''}\simeq\alg(\ssd')_{\kk'}$
associated with $\var_{\ssd}$. Then we have $\varphi_{\ssd}=(\mu_{k}^{*})^{-1}\varphi\mu_{k}^{*}$.
Therefore, $\varphi_{\ssd}$ is identified with $\varphi$ via the
mutation maps $\mu_{k}^{*}$.

\end{Rem}

Let $\alg$ denote $\clAlg$ or $\upClAlg$. Let $\sd^{\prin}$ denote
the seed of principal coefficients. For some choice of compatible
Poisson structure $\lambda_{1}$ for $\sd^{\prin}$ with $\diag_{k}(\sd^{\prin})=\diag_{k}(\sd)$,
the variation map $\var_{1}:\LP(\sd^{\prin})\rightarrow\LP(\sd)$
exists, See Section \ref{subsec:From-principal-coefficients}. Denote
the associated quantum cluster algebra $\alg(\sd^{\prin})$ by $\alg(\sd^{\prin},\lambda_{1})$.
Similarly, the variation map $\var_{2}:\LP(\sd^{\prin})\rightarrow\LP(\sd')$
exists for some choice of $\lambda_{2}$ for $\sd^{\prin}$ with $\diag_{k}(\sd^{\prin})=\diag_{k}(\sd')$.

If $\alg=\clAlg$, $\alg$ and $\alg'$ are related by base changes
$\varphi_{1},\varphi_{2}$ and a quantization change, where the squares
associated with base changes are commutative:
\begin{align}
\begin{array}{ccccccc}
\alg & \xleftarrow{\varphi_{1}} & \alg(\sd^{\prin},\lambda_{1}) & \stackrel[\text{change}]{\text{quantization}}{\sim} & \alg(\sd^{\prin},\lambda_{2}) & \xrightarrow{\varphi_{2}} & \alg'\\
\uparrow & \text{} & \uparrow &  & \uparrow &  & \uparrow\\
\frRing & \stackrel{\varphi_{1}}{\longleftarrow} & \frRing(\sd^{\prin},\lambda_{1}) & \stackrel[\text{change}]{\text{quantization}}{\sim} & \frRing(\sd^{\prin},\lambda_{2}) & \xrightarrow{\varphi_{2}} & \frRing'
\end{array},\label{eq:base-change}
\end{align}

Note that the quantum seeds $(\sd^{\prin},\lambda_{1})$ and $(\sd^{\prin},\lambda_{2})$
are similar with $\rho\in\Q$. If we want to replace the middle square
by a base change, we need to extend the algebras over $\kk''$ and
$\kk'$ respectively. More precisely, we have the $\Z$-algebra isomorphism
$\var:\LP(\sd^{\prin},\lambda_{1})_{\kk''}\simeq\LP(\sd^{\prin},\lambda_{2})_{\kk'}$,
sending $q^{\alpha}x^{m}$ to $q^{\alpha\rho}x^{m}$ such that $\var\frRing(\sd^{\prin},\lambda_{1})_{\kk''}=\frRing(\sd^{\prin},\lambda_{2})_{\kk'}$
. It induces the base change $\varphi:\frRing'(\sd^{\prin},\lambda_{2})_{\kk'}\otimes_{\frRing(\sd^{\prin},\lambda_{1})_{\kk''}}\alg(\sd^{\prin},\lambda_{1})_{\kk''}\simeq\alg(\sd^{\prin},\lambda_{2})_{\kk'}$
between $\Z$-algebras. 

When $\alg=\upClAlg$, we still have (\ref{eq:base-change}) under
the following assumptions:
\begin{enumerate}
\item $\alg$ has a $\kk$-basis $\base$, or $\alg(\sd^{\prin},\lambda_{1})$
has a $\kk$-basis $\base_{1}$, such that it consists of pointed
elements and is closed under the $\frGroup$ commutative multiplication.
\item The similar assumption for $\alg'$ and $\alg(\sd^{\prin},\lambda_{2})$.
\end{enumerate}
Under the above assumptions, $\alg$ and $\alg(\sd^{\prin},\lambda_{1})$
have bases $\base$, $\base_{1}$ respectively, which are related
by a base change: $\base=\frGroup\cdot\var_{1}(\base_{1})=[\varphi_{1}(\frGroup\otimes_{\frGroup(\sd^{\prin},\lambda_{1})}\base_{1})]$.
Similar, $\alg'$ and $\alg(\sd^{\prin},\lambda_{2})$ have bases
$\base'$, $\base_{2}$ related by a base change. Sometimes, we could
choose $\base_{1}$ and $\base_{2}$ related by a quantization change:
$\base_{1}|_{q^{\Hf}\mapsto q^{\frac{1}{2d_{\rho}}}}=\base_{2}$,
see Section \ref{subsec:Triangular-bases-base-change}.

\subsection{Localization\label{subsec:Localization}}

Let $\alg$ denote either $\bClAlg$ or $\bUpClAlg$. Assume it has
a $\kk$-basis $\cZ$ whose elements are pointed in $\LP$ and closed
under the $\frMonoid$ commutative multiplication: we have $p\cdot z\in\cZ$,
$\forall p\in\frMonoid$, $z\in\cZ$. Let $\widetilde{\alg}$ denote
the localization $\alg[x_{j}^{-1}]_{j\in I_{\fv}}$. 

Similar to the previous discussion, we have an isomorphism $\varphi:\frRing\otimes_{\bFrRing}\bLP\xrightarrow{\sim}\LP$,
still called a base change. Note that $\widetilde{\cZ}:=\frGroup\cdot\cZ=\{p\cdot z|p\in\frGroup,\ z\in\cZ\}$
is a basis of $\widetilde{\alg}$, denoted $\widetilde{\cZ}=[\varphi(\frGroup\otimes_{\frMonoid}\cZ)]$.
We deduce that the base change $\varphi$ restricts to an isomorphism
$\frRing\otimes_{\bFrRing}\alg\simeq\widetilde{\alg}$.

\subsection{Triangular bases\label{subsec:Triangular-bases-base-change}}

As before, given $\cZ\subset\LP$ whose elements are pointed. Let
$\cZ'$ be the subset of $\LP'{}_{\kk'}$ consisting of the elements
similar to those of $\cZ$. 

\begin{Lem}\label{lem:similar_triangular_basis}

Let $\alg$ denote $\clAlg$ or $\upClAlg$. If $\cZ\subset\LP(\sd)$
is the triangular basis of $\alg$ with respect to the seed $\ssd:=\seq\sd$
for some $\seq$, then $\alg'$ has the triangular basis with respect
to $\ssd':=\seq\sd'$, and it equals $\cZ'$.

\end{Lem}

\begin{proof}

By Lemma \ref{lem:similar-basis}, $\cZ'$ is a $\kk'$-basis of $\alg'_{\kk'}$.
$\cZ'$ is bar-invariant and contains the cluster monomials in $\ssd'=\seq\sd'$,
$\ssd'[1]=(\seq\sd')[1]$. Let us check the triangularity condition
(last condition).

For any $m$-pointed $z_{m}\in\cZ\subset\LP(\sd)$ and similar $m'$-pointed
$z'_{m'}\in\cZ'\subset\LP(\sd')_{\kk'}$, denote $z_{m}=q^{\alpha}x^{m}*F_{m}$,
$z'_{m'}=q^{\alpha'}(x')^{m'}*F_{m'}'$, where $q^{\alpha},q^{\alpha'}$
are determined by the bar-invariance, $\pr_{I_{\ufv}}m=\pr_{I_{\ufv}}m'$,
$F_{m}\in\kk[\yCone^{\oplus}]$, and $\var(F_{m})=F_{m'}'$ under
the map $\var:\kk[\yCone^{\oplus}]\rightarrow\kk'[(\yCone')^{\oplus}]$
sending $y_{k}$ to $y'_{k}$, $q^{\frac{1}{2}}$ to $q^{\frac{\rho}{2}}$.
By \cite[Lemma 4.2.2]{qin2017triangular}, $\seq\sd$ and $\seq\sd'$
are similar. Moreover, as elements in $\hLP(\ssd)$ and $\hLP(\ssd')_{\kk'}$
respectively, $x^{m}$ and $(x')^{m'}$ are similar, and $F_{m}$
and $F_{m'}'$ are similar. Then Theorem \ref{thm:correction} implies
that $z_{m}$ and $z'_{m'}$ are similar as elements in $\LP(\ssd)$
and $\LP(\ssd')_{\kk'}$ respectively. 

Since $[x_{i}(\ssd)*z_{m}]^{\ssd}$ has a $(\prec_{\ssd},q^{-\Hf}\Z[q^{-\Hf}])$-decomposition
into $\cZ$ in $\LP(\ssd)$, $[x_{i}(\ssd')*z'_{m'}]^{\ssd'}$ has
a $(\prec_{\ssd'},q^{-\frac{\rho}{2}}\Z[q^{-\frac{\rho}{2}}])$-decomposition
into $\cZ'$ in $\LP(\ssd')_{\kk'}$ by Theorem \ref{thm:correction}. 

Therefore, $\cZ'$ is the triangular basis for $\alg'_{\kk'}$ with
respect to $\ssd'$ over $\kk'$. We claim that $\cZ'\subset\LP(\ssd')$.
If so, $\cZ'$ is a $\kk$-basis of $\alg'$ by Lemma \ref{lem:coefficent-kk}.

Note that any $z'\in\cZ'$ has a $(\prec_{\ssd},q^{-\frac{\rho}{2}}\Z[q^{-\frac{\rho}{2}}])$-decomposition
into $[x(\ssd')^{g_{+}}*x(\ssd'[1])^{g_{-}}]^{\ssd'}$, for $g_{+}\in\N^{I}$,
$g_{-}\in\N^{I_{\ufv}}$, see \cite[Lemmas 6.2.1(iii) and 3.1.11]{qin2017triangular}.
Since $[x(\ssd')^{g_{+}}*x(\ssd'[1])^{g_{-}}]^{\ssd'}$ is contained
in $\LP(\ssd')$ and $\cZ'$ is bar-invariant, we deduce that $z'\in\hLP(\ssd')$.
The desired claim follows.

\end{proof}

Lemma \ref{lem:similar_triangular_basis} implies the following result.

\begin{Prop}\label{prop:similar-common-tri-basis}

If $\cZ$ is the common triangular basis for $\upClAlg$, then the
common triangular basis for $\upClAlg'$ exists and it equals $\cZ'$.

\end{Prop}

\section{Based cluster algebras and monoidal categorifications\label{sec:monoidal_categorification}}

\begin{Def}\label{def:based-cluster-algebra}

Let $\alg$ denote $\bClAlg$, $\clAlg$, $\bUpClAlg$, or $\upClAlg$,
and $\base$ its $\kk$-basis satisfying:

\begin{enumerate}

\item $\base$ contains all the cluster monomials.

\item $\forall b\in\base,j\in I_{\fv}$, $x_{j}\cdot b\in\base$.

\item At the quantum level, the elements of $\base$ are invariant
under the bar involution $q^{\Hf}\mapsto q^{-\Hf}$.

\item (Homogeneity) Each $b\in\base$ is contained in some $x^{m}\cdot\kk[N_{\ufv}]$.

\end{enumerate}

The pair $(\alg,\base)$ is called a based cluster algebra. 

\end{Def} 

We call $(\alg,\base)$ $\Theta$-pointed if $\base$ is, where $\Theta$
is a subset of $\cone(\sd)$ or $\tropSet$. 

Localization of $(\alg,\base)$ produces a based cluster algebra $(\alg[x_{j}^{-1}]_{j\in I_{\fv}},\base')$
where $\base'=\{p\cdot b|p\in\frGroup,b\in\base\}$.

\subsection{Monoidal categories\label{subsec:Monoidal-categories}}

Let $\F$ be a field. Let $\cT$ denote a $\F$-linear tensor category
in the sense of \cite[Section A.1]{kang2018symmetric}. In particular,
$\cT$ is endowed with an exact bifunctor $(\ )\otimes(\ )$ called
the tensor functor, its objects have finite lengths, and its Grothendieck
ring $K_{0}(\cT)$ is unital and associative. Let $[X]$ denote the
isoclass of $X\in\cT$. We have $[X\otimes Y]=\sum_{S\in\simpObj}c_{XY}^{S}[S]$
for $X,Y\in\cT$ and $c_{XY}^{S}\in\N$. Recall that the multiplication
of $K_{0}(\cT)$ is induced from the tensor product $\otimes$. Its
unit is $[\unitObj]$ for some unit object $\unitObj$. Let $\simpObj$
denote the set of simple objects of $\cT$. Simple objects are often
referred to as simples.

When $\kk=\Z$, we assume $K_{0}(\cT)$ is commutative. Denote $K:=K_{0}(\cT)$.

When $\kk=\Z[q^{\pm\Hf}]$, we make one of the following assumptions.

\begin{enumerate}

\item We assume that $\cT$ has an object $Q$ satisfying \cite[(6.2)(ii)(iii)]{Kang2018}.
In particular, $\bq:=Q\otimes(\ )$ is an auto-equivalence functor,
such that $[\bq(X\otimes Y)]=[(\bq X)\otimes Y]=[X\otimes(\bq Y)]$.
We denote $X\sim Y$ if $X\simeq\bq^{d}(Y)$ for some $d\in\Z$. Then
$K_{0}(\cT)$ is a $\Z[q^{\pm}]$-algebra such that $q^{\pm}[X]:=[\bq^{\pm}X]$,
where $q$ is a formal parameter. Denote $K:=K_{0}(\cT)\otimes\kk$.

\item Or we assume that the $\kk$-space $K:=K_{0}(\cT)\otimes_{\Z}\kk$
has a $q$-twisted product $*$ such that, for $X,Y\in\cT$, its multiplication
satisfies $[X]*[Y]=\sum_{S\in\simpObj}c(q^{\Hf})_{XY}^{S}[S]$, where
$c(q^{\Hf})_{XY}^{S}\in\N[q^{\pm\Hf}]$ and $c(1)_{XY}^{S}=c_{XY}^{S}$.

\end{enumerate}

The $\kk$-algebra $K$ will be called the (deformed) Grothendieck
ring of $\cT$.

\begin{Rem}\label{rem:simple-to-commute}

Assume that there is an order $2$ anti-automorphism $\overline{(\ )}$
on $K$ such that $\overline{[S]}\sim[S]$ for any simple $S$. Take
another simple $V$ such that $V\otimes S$ is simple. We claim that
$[S]$ and $[V]$ $q$-commute, $S\otimes V$ is simple, and $S\otimes V\sim V\otimes S$.

To see this, denote $[V\otimes S]=[V]*[S]=[W]$ for some simple $W$.
Applying $\overline{(\ )}$, we obtain $[S\otimes V]=[S]*[V]\sim[W]$.
The claim follows.

\end{Rem}

Recall that $\cT$ has a natural isomoprhism $a_{X,Y,Z}:(X\otimes Y)\otimes Z\simeq X\otimes(Y\otimes Z)$,
called the associativity isomorphism. Following \cite[Definition 2.1.5]{etingof2015tensor},
we define the opposite monoidal category $\cT\op$ to be the abelian
category $\cT$ edowed with the new tensor product $X\otimes\op Y:=Y\otimes X$
and the new associativity isomorphism $a_{X,Y,Z}\op:=a_{Z,Y,X}^{-1}$.
Then $K(\cT\op)$ is the opposite algebra of $K(\cT)$. 

\subsection{Categorifications of based cluster algebras\label{subsec:Categorifications-of-based}}

We will use the following notion of monoidal categorification, which
is weaker than the one in \cite{HernandezLeclerc09}.

\begin{Def}\label{def:categorification}

We say $\cT$ admits a cluster structure $\alg$ if the following
holds.

\begin{enumerate}

\item There exists a $\kk$-algebra embedding $\kappa:\alg\hookrightarrow K$.

\item For any cluster monomial $z$, there exists a simple $S$ such
that $\kappa(z)\sim[S]$.

\end{enumerate}

We further say $\cT$ categorifies $\alg$ if $\kappa$ is isomorphic.
We say $\cT$ categorifies $\alg$ after localization if $\kappa$
extends to $\kappa:\alg[x_{j}^{-1}]_{j\in I_{\fv}}\simeq K[\kappa(x_{j})^{-1}]_{j\in I_{\fv}}$.
We can omit the symbol $\kappa$ in this case.

\end{Def}

\begin{Def}\label{def:categorification-based}

We say $\cT$ categorifies $(\alg,\base)$ if the following holds.

\begin{enumerate}

\item $\cT$ categorifies $\alg$.

\item For any simple object $S$, $q^{\alpha}[S]\in\kappa\base$
for some $\alpha$.

\end{enumerate}

Categorification after localization is defined in a similar manner.
In addition, the notion of categorification for cluster algebras defined
over the any extension of $\kk$ can also be similarly established.

\end{Def}

\begin{Def}\label{def:quasi-categorification}

We say $\cT$ quasi-categorifies a based cluster algebra $(\alg',\base')$
if there exists $(\alg,\base)$, which is related to $(\alg',\base')$
by a finite sequence of base changes and quantization changes (see
(\ref{eq:base-change}), section \ref{subsec:Localization}), such
that $\cT$ categorifies $(\alg,\base)$ after localization.

\end{Def}

Note that, in Definition \ref{def:quasi-categorification}, when the
elements of $\base$ are pointed elements in some $\LP(\sd)$ up to
$q^{\frac{\Z}{2}}$-multiples, the structure constants of $\base'$
are determined from those for $\base$, see Theorem \ref{thm:correction}.

Now assume that $\cT$ categories $(\alg,\base)$. For any cluster
variables $x_{i}(\sd)$, choose a simple object $S_{i}(\sd)$ such
that $q^{\alpha}[S_{i}(\sd)]=\kappa(x_{i}(\sd))$. We will often omit
the symbols $\kappa$ and $\sd$ for simplicity. Let us present some
immediate consequences.

Apply the bar-involution $\overline{(\ )}$ in $\alg$ to $[S]$,
$S\in\simpObj$, which is the identity map at the classical level.
We obtain $\overline{[S]}\sim[S]$.

The simple objects $S_{j}$, $j\in I_{\fv}$, are said to be frozen.
Since $q^{\frac{\Z}{2}}\base$ is closed under multiplication by the
frozen variables $x_{j}$, for any simple object $S$, $S_{j}\otimes S$
and $S\otimes S_{j}$ remain simple.

Finally, the structure constants of $\base$ are non-negative. 

\begin{Rem}\label{rem:ss-categorification}

If $\base$ has positive structure constants, we can construct a semisimple
monoidal category $\cT$ as below. The set of simple objects of $\cT$
is $q^{\frac{\Z}{2}}\base$. For any simple objects $S,S'\in\cT$,
we define their tensor product by multiplication: $S\otimes S':=\oplus_{j}S_{j}^{\otimes c_{j}}$
whenever $S*S'=\sum_{S_{j}\in q^{\frac{\Z}{2}}\base}c_{j}S_{j}$,
$c_{j}\in\N$, as elements in $\alg$. Then $\cT$ categorifies $(\alg,\base)$.
Usually, we desire categories with richer structures than $\cT$.

\end{Rem}

\subsection{Triangular bases}

Work at the quantum case $\kk=\Z[q^{\pm\Hf}]$. 

\begin{Cond*}[(C)]\label{condition:category_commuting}

The category $\cT$ is said to satisfy the commuting condition, or
Condition (C), with respect to a simple $X$ if, for any simple $S$,
$X\otimes S$ and $S\otimes X$ are simple whenever $[X]$ and $[S]$
$q$-commute.

\end{Cond*}

Now, assume we have an inclusion $K\subset\LP(\sd)$ such that, for
any simple $S$, $\overline{[S]}\sim[S]$ (see also Remark \ref{rem:simple-to-commute}).
Further assume that for each simple $S\in\cT$, $q^{\beta}S$ is pointed
in $\LP(\sd)$ for some $\beta\in\Hf\Z$. Moreover, $\deg S\neq\deg S'$
whenever $S\nsim S'$ for simples $S,S'$.

\begin{Cond*}[(T)]\label{condition:category_triangular}The category
$\cT$ is said to satisfy the $\prec_{\sd}$-triangularity condition,
or Condition (T), with respect to a simple $X$ if we have the following
decomposition in $K$ for any simple $S$:
\begin{align}
[X]*[S] & =\sum_{j=0}^{s}c_{j}q^{\alpha_{j}}q^{\beta_{j}}[S_{j}],\label{eq:tensor_triangular_decomposition}
\end{align}
where $S_{j}$ are simple, $\overline{q^{\beta_{j}}[S_{j}]}=q^{\beta_{j}}[S_{j}]$,
$c_{0}=1$, $c_{j}\in\N$, and $\alpha_{j}<\alpha_{0}$ for $j>0$.

\end{Cond*}

\begin{Rem}\label{rem:condition_T_to_C}

We claim that Condition (T) implies Condition (C) with respect to
$X$. To see this, denote $[X]*[S]=q^{\gamma}[S]*[X]$. Then $[X]*[S]=q^{\gamma}\overline{[X]*[S]}$.
Decomposing $[X]*[S]$ as in (\ref{eq:tensor_triangular_decomposition}),
we deduce that $q^{\gamma}\overline{c_{j}}=c_{j}$. So $\gamma=2\alpha_{0}$
and $c_{j}=0$ for $j>0$. Therefore, $X\otimes S$ is simple. By
Remark \ref{rem:simple-to-commute}, $S\otimes X$ is simple too.

\end{Rem}

\begin{Thm}\label{thm:category_tri_basis}

Take $\kk=\Z[q^{\pm\Hf}]$ and $\alg=\clAlg$ or $\upClAlg$. Assume
that $\alg$ is injective-reachable, $\cT$ categorifies $(\alg,\base)$
after localization, and the elements of $\base$ are pointed at distinct
degrees in $\LP(\sd)$ for some seed $\sd$. If $\cT$ satisfies the
triangularity condition with respect to $S_{k}(\sd)$, $k\in I_{\ufv}$,
then $\base$ is the common triangular basis of $\alg$.

\end{Thm}

\begin{proof}

Note that $\cT$ always satisfy the Condition (T) with respect to
the frozen simples $S_{j}$, $j\in I_{\fv}$. 

By definition of based cluster algebras, $\base$ contains all the
cluster monomials, including those in $\sd$ and $\sd[1]$. Then the
Condition (T) with respect to $S_{i}(\sd)$, $i\in I$, implies that
$\base$ is the triangular basis with respect to $\sd$. Since this
triangular basis contains all the cluster monomials, it is the common
triangular basis by \cite[Proposition 6.4.3, Theorem 6.5.3]{qin2020dual}.

\end{proof}

\begin{Rem}\label{rem:triangularity_condition}

For symmetric $C$, Condition (T) holds for categories of KLR algebras'
finite dimensional modules with respect to the real simple modules,
see \cite[Corollary 3.8]{Kang2018}. In type $ADE$, it also holds
for finite dimensional modules categories of quantum affine algebras
with respect to the cluster variables, see \cite{qin2017triangular}.

Conversely, if $\base$ is the triangular basis with respect to any
seed $\sd$, then Condition (T) holds with respect to the simples
corresponding to the cluster monomials in $\sd$.

\end{Rem}

\subsection{Monoidal subcategories\label{subsec:Monoidal-subcategories}}

Work at the quantum case $\kk=\Z[q^{\pm\Hf}]$. Let $\sd$ denote
a seed of $\clAlg$. Assume we have an inclusion of $\kk$-algebras
$K\hookrightarrow\LP(\sd)$ such that $q^{\alpha_{k}}[S_{i}]=x_{i}$
for some simples $S_{i}=S_{i}(\sd)$, where $i\in I$, and $\alpha_{i}\in\Hf\Z$.
Assume $\overline{[S]}\sim[S]$, $\forall S\in\simpObj$.

\begin{Prop}[{\cite[Proposition 2.2.(i)]{HernandezLeclerc09}}]\label{prop:monoidal_positivie}

For any object $V$ of $\cT$ , the Laurent expansion of $[V]$ in
$\LP(\sd)$ have non-negative coefficients.

\end{Prop}

Choose any $F\subset I_{\ufv}$ and introduce the following set:
\begin{align}
\simpObj' & :=\{S\in\simpObj|[S_{k}]\text{ \ensuremath{q}-commutes with }[S],\ \forall k\in F\}.\label{eq:commuting_simple}
\end{align}
If $\cT$ satisfies Condition (C) with respect to $S_{k}$, $k\in F$,
we have
\begin{align*}
\simpObj' & =\{S\in\simpObj|S_{k}\otimes S\text{ is simple},\ \forall k\in F\}.
\end{align*}

Let $\cT'$ denote the full-subcategory of $\cT$ whose objects' composition
factors belong to $\simpObj'$. Then $\cT'$ is closed under taking
biproducts, extensions, kernels, and quotients. In particular, $\cT'$
is an abelian subcategory. If $\cT$ is equipped with the auto-equivalent
functor $Q\otimes(\ )$ as in Section \ref{subsec:Monoidal-categories},
$\cT'$ is closed under this auto-equivalence. The (deformed) Grothendieck
group of $\cT'$, denoted $K'$, is defined to be the submodule of
$K$ spanned by $\{[S]|S\in\simpObj'\}$.

\begin{Lem}\label{lem:factor_commute}

For any $k\in F$ and object $V\in\cT$, if $[V]$ $q$-commutes with
$[S_{k}]$, then $[W]$ $q$-commutes with $S_{k}$ for every subquotient
$W$ of $V$.

\end{Lem}

\begin{proof}

We have $[S_{k}]*[V]=q^{\alpha}[V]*[S_{k}]$ for some $\alpha\in\Z$.
Consider the corresponding Laurent polynomials in $\LP(\sd)$ and
note that $[S_{k}]\sim x_{k}$. Then all Laurent monomials $x^{m}$
appearing in $[V]$ satisfies $x^{m}*x_{k}=q^{\alpha}x_{k}*x^{m}$.

We have $[V]=\sum b_{i}[V_{i}]$ for $b_{i}\in\N$, $V_{i}$ composition
factors of $V$. By Proposition \ref{prop:monoidal_positivie}, $[V_{i}]$
have positive Laurent expansions. So the Laurent monomials appearing
in $[V_{i}]$ already appear in $[V]$. Therefore, $x_{k}*[V_{i}]=q^{\alpha}[V_{i}]*x_{k}$.
The claim follows.

\end{proof}

\begin{Lem}\label{lem:abelian_sub}

The subcategory $\cT'$ is closed under the tensor product.

\end{Lem}

\begin{proof}

For any two simples $V,W$ in $\cT'$, $[V\otimes W]=[V]*[W]$ $q$-commutes
with $S_{k}$, $k\in F$. By Lemma \ref{lem:factor_commute}, all
composition factors of $V\otimes W$ $q$-commute with $S_{k}$. Therefore,
$V\otimes W$ remains in $\cT'$. In general, for any object $V'$
and $W'$ in $\cT'$, we have $[V']=\sum b_{i}[V_{i}]$ and $[W']=\sum c_{j}[W_{j}]$,
$b_{i},c_{j}\in\N$, for composition factors $V_{i},W_{j}$. Note
that $[V'\otimes W']=[V']*[W']=\sum b_{i}c_{j}[V_{i}]*[W_{j}]=\sum b_{i}c_{j}[V_{i}\otimes W_{j}]$.
Any composition factor of $V'\otimes W'$ must be a composition factor
of some $V_{i}\otimes W_{j}$, which we have seen is contained in
$\cT'$. So $V'\otimes W'$ belongs to $\cT'$.

\end{proof}

\begin{Rem}\label{rem:sub_cluster_monomial}

Any cluster monomial $z$ of $\bClAlg(\frz_{F}\sd)$ $q$-commutes
with $x_{k}$, $\forall k\in F$. So any simple $S\in\cT$ such that
$[S]\sim z$ belongs to $\cT'$. Therefore, if $\cT$ admits a cluster
structure $\bClAlg$, $\cT'$ admits a cluster structure $\bClAlg(\frz_{F}\sd)$.

\end{Rem}

By Proposition \ref{prop:monoidal_positivie}, the set of simples
is a positive $\N$-basis of $K\subset\LP(\sd)$ that has positive
Laurent expansions. Therefore, it has the following stabilization
property, see Lemma \ref{lem:positive-to-stabilization}.

\begin{Prop}\label{prop:stabilization-simple}

Take any object $V$ in $\cT$ and $i\in I$. Consider the decomposition
$[S_{i}^{\otimes d}\otimes V]=\sum b_{j}^{(d)}[V_{j}^{(d)}]$ where
$V_{j}^{(d)}$ are composition factors and $b_{j}^{(d)}\in\N$. When
$d$ is large enough, $x_{i}^{-d}*[V_{j}^{(d)}]$ and $b_{j}^{(d)}$
stabilize, and $V_{j}^{(d)}\in\cT'$. 

\end{Prop}

\subsection{Categorifications via monoidal subcategories\label{subsec:Categorifications-via-monoidal-sub}}

For the moment, assume $\cT$ categorifies $(\alg,\base)$ after localization
and it has the triangularity property with respect to $S_{k}=S_{k}(\sd)$,
$k\in F$. Assume $\base=\{b_{m}|m\in\cone(\sd)\}$ with $m$-pointed
$b_{m}$ contained in $\LP(\sd)$. For those $b_{m}$ in $\kappa^{-1}K$,
we can choose and fix a simple $S_{m}$ such that $\kappa b_{m}\sim[S_{m}]$.
Omit $\kappa$ for simplicity.

\begin{Lem}\label{lem:leading_term_shift}

When $d\in\N$ is large enough, we have $[S_{m+df_{i}}]*[S_{k}]\sim[S_{m+(d+1)f_{i}}]$.

\end{Lem}

\begin{proof}

$[S_{m+df_{i}}]$ appears in the decomposition of $[S_{k}]^{d}*[S_{m}]$.
So, when $d$ is large enough, the claim follows from the stabilization
property, see Proposition \ref{prop:stabilization-simple}.

\end{proof}

\begin{Thm}\label{thm:monoidal-freezing}

For any simple $S_{m}\in\cT$, the freezing operator acts on $[S_{m}]$
by 
\begin{align*}
\frz_{F}[S_{m}] & \sim x^{-d\sum_{k\in F}f_{k}}*[S_{m+d\sum_{k\in F}f_{k}}]
\end{align*}
for $d\in\N$ large enough. In this case, $S_{m+d\sum_{k\in F}f_{k}}\in\cT'$.

\end{Thm}

\begin{proof}

By Lemma \ref{lem:leading_term_shift}, $S_{m+d\sum_{k\in F}f_{k}}\in\cT'$.
The claim follows from Lemma \ref{lem:stabilization-freezing}.

\end{proof}

\begin{Thm}\label{thm:sub_category_upClAlg}

If $\base$ is $\tropSet$-pointed and $\frz_{F}\sd$ is injective-reachable,
then $\cT'$ categorifies $(\upClAlg(\frz_{F}\sd),\frz_{F}\base)$
after localization. 

\end{Thm}

\begin{proof}

For any $S_{m}\in\simpObj'\subset\cT'$, we have $b_{m}=\frz_{F}b_{m}$,
which is contained in $\upClAlg(\frz_{F}\sd)$, see Proposition \ref{prop:project_up_cl_alg}.
Therefore, $K'\subset\upClAlg(\frz_{F}\sd)$.

Since $\base$ is $\tropSet$-pointed, $\frz_{F}\base$ is a $\tropSet$-pointed
basis of $\upClAlg(\frz_{F}\sd)$ by Theorem \ref{thm:freeze_good_bases}.
Take any $b_{m}\in\kappa^{-1}K$ with $b_{m}\sim[S_{m}]$. By Theorem
\ref{thm:monoidal-freezing}, $\frz_{F}b_{m}=x^{-d\sum_{k\in F}f_{k}}\cdot q^{\alpha}[S_{m+d\sum_{k\in F}f_{k}}]$
for $d\in\N$ large enough, $\alpha\in\frac{\Z}{2}$, and $S_{m+d\sum_{k\in F}f_{k}}\in\simpObj'\subset\cT'$.
We deduce that $\upClAlg(\frz_{F}\sd)\subset K'[[S_{i}]^{-1}]_{i\in I_{\fv}\sqcup F}$. 

\end{proof}

From now on, we drop the assumptions in the beginning of Section \ref{subsec:Categorifications-via-monoidal-sub}.

\begin{Thm}\label{thm:subcategory-b-upcluster}

Let $(\bUpClAlg,\base)$ be a based cluster algebra such that $\base$
is $\cone(\sd)$-pointed in $\LP(\sd)$. Assume that $\base':=\base\cap\bUpClAlg(\frz_{F}\sd)$
is a basis for $\bUpClAlg(\frz_{F}\sd)$. If $\cT$ categorifies $(\bUpClAlg,\base)$,
then $\cT'$ categorifies $(\bUpClAlg(\frz_{F}\sd),\base')$.

\end{Thm}

\begin{proof}

Since $\base'$ contains all cluster monomials of $\bUpClAlg(\frz_{F}\sd)$,
it remains to show $K'=\bUpClAlg(\frz_{F}\sd)$.

$\forall S\in\simpObj'\subset\cT'$, since $[S]$ $q$-commute with
$x_{k}$, $k\in F$, we have $\supp[S]\cap F=\emptyset$. Then Lemma
\ref{lem:f-bounds-d} implies $\nu_{k}([S])\geq0$, $\forall k\in F$.
Note that $\nu_{j}([S])\geq0$ for $j\in I_{\fv}$ since $[S]\in\bUpClAlg$.
Combining with $[S]=\frz_{F}[S]\in\upClAlg(\frz_{F}\sd)$, we obtain
$[S]\in\bUpClAlg(\frz_{F}\sd)$. We deduce that $K'\subset\bUpClAlg(\frz_{F}\sd)$.

We have $\Span_{\kk}\{[S]|S\in\simpObj,[S]\in\bUpClAlg(\frz_{F}\sd)\}=\Span_{\kk}(\base\cap\bUpClAlg)$,
which coincides with $\bUpClAlg(\frz_{F}\sd)$ by our assumption.
Take any $S\in\simpObj$ such that $[S]\in\bUpClAlg(\frz_{F}\sd)$.
Then $[S]$ $q$-commutes with $x_{k}$, $k\in F$, since $x_{k}$
are frozen variables in $\bUpClAlg(\frz_{F}\sd)$ and $[S]$ is pointed
in $\LP(\frz_{F}\sd)$. So we obtain $S\in\simpObj'\subset\cT'$.
We deduce that $\bUpClAlg(\frz_{F}\sd)$ is contained in $K'$.

\end{proof}

Now assume $\upClAlg$ has the common triangular basis $\can$. Then
$\upClAlg(\frz_{F}\sd)$ has the common triangular basis $\can'=\frz_{F}\can$
(Theorem \ref{thm:sub_cluster_triangular_basis}). Assume that $\sd$
and $\frz_{F}\sd$ can be optimized. Then $\overline{\can}:=\can\cap\bUpClAlg$
is a basis of $\bUpClAlg$ and $\overline{\can}'=\can'\cap\bUpClAlg(\frz_{F}\sd)$
is a basis of $\bUpClAlg(\frz_{F}\sd)$.

\begin{Cor}\label{cor:subcategorify-combine}

If $\cT$ categorifies $(\bUpClAlg,\overline{\can})$, then $\cT'$
categorifies $(\bUpClAlg(\frz_{F}\sd),\overline{\can}')$.

\end{Cor}

\begin{proof}

Proposition \ref{prop:good-sub-up-cl-alg} implies $\bUpClAlg(\frz_{F}\sd)\subset\bUpClAlg(\sd)$.
Corollary \ref{cor:tri-freezing-equal-restriction}(3) implies that
the basis $\overline{\can}'=\frz_{F}\can\cap\bUpClAlg(\frz_{F}\sd)$
equals $\can\cap\bUpClAlg(\frz_{F}\sd)$, which is $\can\cap(\bUpClAlg\cap\bUpClAlg(\frz_{F}\sd))=(\can\cap\bUpClAlg)\cap\bUpClAlg(\frz_{F}\sd)=\overline{\can}\cap\bUpClAlg(\frz_{F}\sd)$.
The desired claim follows from Theorem \ref{thm:subcategory-b-upcluster}.

\end{proof}

\section{Cluster algebras associated with signed words\label{sec:Cluster-algebras-signed-words}}

\subsection{Seeds from trapezoids}

\subsubsection*{Signed words\label{subsec:Signed-words}}

A signed word $\ubi=(\bi_{1},\bi_{2},\ldots,\bi_{l})$ is a sequence
of the elements of $\Z\backslash\{0\}$. For $1\leq j\leq k\leq l$,
denote $\ubi_{[j,k]}=(\bi_{j},\bi_{j+1},\ldots,\bi_{k})$ . For any
$k\in[1,l]$, define its successor and predecessor by
\begin{align*}
k[1]:=\min(\{k'\in[k+1,l],\ |\bi_{k'}|=|\bi_{k}|\}\cup\{+\infty\}),\\
k[-1]:=\min(\{k'\in[1,k-1],\ |\bi_{k'}|=|\bi_{k}|\}\cup\{-\infty\}).
\end{align*}
Inductively, define $k[d\pm1]:=k[d][\pm1]$ when $k[d]\in\Z$, $d\in\Z$.
$\forall a\in\N_{>0}$, define 
\begin{gather*}
O^{\ubi}([j,k];a)=|\{s\in[j,k],\ |\bi_{s}|=a\}|,\ O^{\ubi}(a):=O^{\ubi}([1,l];a),\\
o_{-}^{\ubi}(k):=O^{\ubi}([1,k-1];|\bi_{k}|),\ o_{+}^{\ubi}(k):=O^{\ubi}([k+1,l];|\bi_{k}|).
\end{gather*}
Denote $k^{\min}=k[o_{-}^{\ubi}(k)]$, $k^{\max}=k[o_{+}^{\ubi}(k)]$.

Define the length $l(\ubi):=l$, $-\ubi:=(-\bi_{1},\ldots,-\bi_{l})$,
and the opposite signed word $\ubi\op=(\bi_{l},\cdots,\bi_{2},\bi_{1})$.
Define the support $\supp\ubi:=\{|\bi_{k}|\ |\ k\in[1,l]\}$. For
any $s\in\N$, we also denote $s$-copies of $\ubi$ by $\ubi^{s}:=(\ubi,\ldots,\ubi)$.

Introduce $I(\ubi):=\{\binom{a}{d}^{\ubi}|a\in\supp\ubi,0\leq d<O^{\ubi}(a)\}$.
Then we have an isomorphism $[1,l]\simeq I(\ubi)$, such that $k$
is identified $\binom{|\bi_{k}|}{o_{-}^{\ubi}(k)}^{\ubi}$. The order
$<$ on $[1,l]$ induces the order $<_{\ubi}$ on $I(\ubi)$. We also
introduce $\ddI(\ubi):=\{\binom{a}{-1}|a\in J\}\sqcup I(\ubi)$.

We often omit the symbol $\ubi$ for simplicity. We will identify
$[1,l]$ and $I(\ubi)$ and use their symbols interchangeably.

If $\ubi'$ is another signed word such that $O^{\ubi}(a)=O^{\ubi'}(a)$,
$\forall a\in J$, we will also identify $I(\ubi)$ with $I(\ubi')$
such that $\binom{a}{d}^{\ubi}=\binom{a}{d}^{\ubi'}$. But $<_{\ubi}$
and $<_{\ubi'}$ are different.

\begin{Eg}

Take $\ubi=(1,2,1,-2,-1,-2)$. Then the ordered set $I(\ubi)=\{\binom{1}{0}<\binom{2}{0}<\binom{1}{1}<\binom{2}{1}<\binom{1}{2}<\binom{2}{2}\}$
is isomorphic to $[1,6]$. For $\ubi'=(2,1,-2,-2,-1,-1)$, we have
$I(\ubi)=I(\ubi')$ but their isomorphisms with $[1,6]$ are different.

\end{Eg}

\subsubsection*{Positive Braids}

Choose and fix a finite subset $J\subset\Z_{>0}$. Let $C=(C_{ab})_{a,b\in J}$
denote a generalized Cartan matrix with symmetrizers $\symm_{a}\in\Z_{>0}$:
$C_{aa}=2$, $C_{ab}\in\Z_{\leq0}$ $\forall a\neq b$, and $\symm_{b}C_{ab}=\symm_{a}C_{ba}$.

The monoid $\Br^{+}$ of positive braids is generated by $\sigma_{a}$,
$\forall a\in J$, such that
\begin{align*}
\sigma_{a}\sigma_{b}=\sigma_{b}\sigma_{a} & \text{ if }C_{ab}C_{ba}=0, & \sigma_{a}\sigma_{b}\sigma_{a}=\sigma_{b}\sigma_{a}\sigma_{b} & \text{ if }C_{ab}C_{ba}=1,\\
(\sigma_{a}\sigma_{b})^{2}=(\sigma_{b}\sigma_{a})^{2} & \text{ if }C_{ab}C_{ba}=2, & (\sigma_{a}\sigma_{b})^{3}=(\sigma_{b}\sigma_{a})^{3} & \text{ if }C_{ab}C_{ba}=3.
\end{align*}
Let $e$ be its identity element. The generalized braid group $\Br$
is generated by $\sigma_{a}^{\pm}$. 

A word $\ueta=(\eta_{1},\eta_{2},\ldots,\eta_{l})$ is a sequence
formed from the letters of $J$. We associate with it the positive
braid $\beta_{\ueta}:=\sigma_{\eta_{1}}\cdots\sigma_{\eta_{l}}\in\Br^{+}$.
The Weyl group $W$ is the quotient of $\Br$ by the relations $\sigma_{a}^{2}=\Id$,
$\forall a$. Denote the image of $\beta\in\Br$ in $W$ by $[\beta]$.
Denote $w_{\ueta}:=[\beta_{\ueta}]$ and $s_{a}:=[\sigma_{a}]$. The
length of an element $w\in W$, denoted $l(w)$, is defined to be
the minimum of the length of $\ueta$ such that $w=w_{\ueta}$ in
$W$, in which case we say $\ueta$ is a reduced word.

Let $(\uzeta,\ueta)$ denote any pair of words. We will recall the
seed $\dsd(\uzeta,\ueta,\Delta)$ associated to any triangulation
$\Delta$ of a trapezoid defined in \cite{shen2021cluster}. We will
also denote it by $\dsd(\ubi)$ for the corresponding signed word
$\ubi$. See Example \ref{eg:seed_triangulated_trapezoid}.

\subsubsection*{Triangulated trapezoids}

Let $\Sigma_{\ueta}^{\uzeta}$ denote a trapezoid in $\R^{2}$ with
horizontal bases, whose top base is subdivided into intervals labeled
by $-\uzeta$ from left to right (i.e., along the $(1,0)$-direction),
and its bottom base is analogously subdivided and labeled by $\ueta$.
The marked points are the endpoints of these intervals. A diagonal
is an interior line segment in $\Sigma_{\ueta}^{\uzeta}$ connecting
a top marked point with a bottom one. A triangulation $\Delta$ is
a maximal collection of non-intersecting distinct diagonals. 

Note that $\Delta$ produces a partition of $\Sigma_{\ueta}^{\uzeta}$
into $(l(\ueta)+l(\uzeta))$-many triangles. For any triangle $T$
in the partition, denoted $T\in\Delta$, we define $\bi_{T}\in\pm J$
to be the label of its horizontal edge. Define $\sign(T):=\sign(\bi_{T})$. 

\subsubsection*{Seeds}

Draw distinct $a$-labeled horizontal lines $L_{a}$ between the bases
of $\Sigma_{\ueta}^{\uzeta}$ for $a\in J$, called layers. We insert
one relative interior point $p_{T}$ in $L_{|\bi_{T}|}\cap T$ for
each $T\in\Delta$. The collection of $L_{a}$ and $p_{T}$ is called
a string diagram, denoted $D(\uzeta,\ueta,\Delta)$.

Each layer $L_{a}$ is subdivided by the nodes into intervals, which
are labeled by the vectors $\binom{a}{-1},\binom{a}{0},\ldots,\binom{a}{O(a)-1}$
from left to right. Let $\ddI$ consist of these intervals, $I_{\ufv}:=\ddI_{\ufv}$
the bounded ones, $\ddI_{\fv}$ the unbounded ones, $I_{\fv}=\{\binom{a}{O(a)-1},\forall a\}$
the right-unbounded ones, and $I_{\fv}^{-}=\{\binom{a}{-1},\forall a\}$
the left-unbounded ones. Any triangle $T_{k}$ with $|\bi_{k}|=a$
intersects one interval $\ell_{b}=\binom{b}{O([1,k-1];b)-1}\subset L_{b}$
for each $b\ne a$, an interval $\ell_{a}^{-}=\text{\ensuremath{\binom{a}{o_{-}(k)-1}}}\subset L_{a}$
on the left of $p_{T}$, and an interval $\ell_{a}^{+}=\binom{a}{o_{-}(k)}\subset L_{a}$
on the right of $p_{T}$. It contributes a skew-symmetric $\ddI\times\ddI$
$\frac{\Z}{2}$-matrix $B_{\Gamma}^{(T)}$ whose non-zero entries
are:
\begin{itemize}
\item $(B_{\Gamma}^{(T)})_{\ell_{a}^{-},\ell_{a}^{+}}=-(B_{\Gamma}^{(T)}){}_{\ell_{a}^{+},\ell_{a}^{-}}=\sign(T)$.
\item $\forall b\neq a$, $(B_{\Gamma}^{(T)})_{\ell_{b},\ell_{a}^{-}}=-(B_{\Gamma}^{(T)})_{\ell_{b},\ell_{a}^{+}}=\Hf\sign(T)$,
$(B_{\Gamma}^{(T)})_{\ell_{a}^{+},\ell_{b}}=-(B_{\Gamma}^{(T)})_{\ell_{a}^{-},\ell_{b}}=\Hf\sign(T)$.
\end{itemize}
Denote $B_{\Gamma}:=\sum_{T\in\Delta}B_{\Gamma}^{(T)}$ and its quiver
by $\Gamma$. Define $\ddB=(b_{\ell,\ell'})_{\ell,\ell'\in\ddI}$
by
\begin{itemize}
\item $b_{\ell,\ell'}=(B_{\Gamma})_{\ell,\ell'}$ if $\ell,\ell'\subset L_{a}$
for some $a\in J$.
\item $b_{\ell,\ell'}=(B_{\Gamma})_{\ell,\ell'}(-C_{ab})$ if $\ell\subset L_{a}$,
$\ell'\subset L_{b}$, for some $a\neq b\in J$.
\end{itemize}
Then $\ddB$ is skew-symmetrizable with symmetrizers $d_{l}$, such
that $d_{l}^{\vee}=\symm_{a}$ for $\ell\subset L_{a}$. In addition,
the submatrix $\ddB_{\ddI\times I_{\ufv}}$ is a $\Z$-matrix (see
(\ref{eq:dBS_B_matrix}) or \cite[Remark 3.6]{shen2021cluster}).

The above data gives rise to a seed $\dsd(\uzeta,\ueta,\Delta)$ whose
$B$-matrix is $\ddB_{\ddI\times I_{\ufv}}$. Removing the frozen
vertices in $I_{\fv}^{-}$, we further obtain the reduced seed $\rsd(\uzeta,\ueta,\Delta)$
whose vertex set is $I:=I_{\ufv}\sqcup I_{\fv}$.

\begin{Eg}\label{eg:seed_triangulated_trapezoid}

Choose $C=\left(\begin{array}{cc}
2 & -1\\
-1 & 2
\end{array}\right)$, $\uzeta=\ueta=(1,2,1)$, and a triangulation for $\Sigma_{\ueta}^{\uzeta}$.
Figure \ref{fig:triangulated_trapezoid}(A) shows the non-zero entries
of $B_{\Gamma}^{(T)}$ for single triangles $T$, Figure \ref{fig:triangulated_trapezoid}(B)
shows the string diagram, and Figures \ref{fig:triangulated_trapezoid}(C)(D)
show the quiver associated with $\ddB$ via $\ddI\simeq[-2,-1]\sqcup[1,6]$.
Circular nodes denote unfrozen vertices, and rectangular nodes denote
frozen vertices. Dashed arrows have weight $\frac{1}{2}$.

\begin{figure}[h]
\caption{The quiver for $\ddB(\dsd(1,2,1,-1,-2,-1))$. }

\label{fig:triangulated_trapezoid}

\subfloat[]{

\begin{tikzpicture}  [node distance=48pt,on grid,>={Stealth[length=4pt,round]},bend angle=45, inner sep=0pt,      pre/.style={<-,shorten <=1pt,>={Stealth[round]},semithick},    post/.style={->,shorten >=1pt,>={Stealth[round]},semithick},  unfrozen/.style= {circle,inner sep=1pt,minimum size=1pt,draw=black!100,fill=red!50},  frozen/.style={rectangle,inner sep=1pt,minimum size=12pt,draw=black!75,fill=cyan!50},   point/.style= {circle,inner sep=0pt, outer sep=1.5pt,minimum size=1.5pt,draw=black!100,fill=black!100},   boundary/.style={-,draw=cyan},   internal/.style={-,draw=red},    every label/.style= {black}]   
\draw[-] (-146.755pt,60pt) node (m1){}--(-133.0575pt,0) node (m2){}; \draw[-] (m1)--(-160.17pt,0) node(m3){}; \draw[-] (m2)--(m3);
\node[point] (m4) at (-146.14pt,21.4225pt) {}; \draw[-,draw=cyan!100] (-158.4925pt,21.4225pt)--(m4); \draw[-,draw=cyan!100] (-135.605pt,21.4225pt)--(m4); \draw[-,draw=cyan!100] (-158.4925pt,37.155pt)--(-135.5625pt,37.155pt) node(m6){};
\node at (-165.225pt,37.155pt) {$\color{cyan}L_b$}; \node at (-165.225pt,21.4225pt) {$\color{cyan}L_a$};
\draw[-] (-116.7125pt,0pt) node (n1){}--(-104.4375pt,60pt) node (n2){}; \draw[-] (n1)--(-128.705pt,60pt) node(n3){}; \draw[-] (n2)--(n3); \node[point] (n4) at (-116.0975pt,21.4225pt) {}; \draw[-,draw=cyan!100] (-129.8725pt,21.4225pt)--(n4);
\draw[-,draw=cyan!100] (-104.055pt,21.4225pt)--(n4);
\draw[-,draw=cyan!100] (-129.8725pt,37.155pt)--  (-104.1825pt,37.155pt) ;
\node at (-147.155pt,-5pt) {a};
\node at (-115.605pt,65pt) {-a};
\draw[->,draw=teal!100] (-150.845pt,22.4225pt) node (w8) {} .. controls (-146.0275pt,25.4225pt) .. (-142.3375pt,22.4225pt) node (w5) {}; \draw[->,dotted,draw=teal!100] (-149.4225pt,35.845pt) node (w10) {} -- (w8); \draw[<-,dotted,draw=teal!100]  (-143.6475pt,35.845pt) node(w11){} -- (w5);
\draw[<-,draw=teal!100] (-123.7325pt,22.4225pt) node (y8) {} .. controls (-116.07pt,25.4225pt) .. (-109.535pt,22.4225pt) node (y5) {}; \draw[<-,dotted,draw=teal!100] (-119.465pt,35.845pt) node (w10) {} -- (y8); \draw[->,dotted,draw=teal!100]  (-113.69pt,35.845pt) node(w11){} -- (y5);
\end{tikzpicture}

}$\qquad$\subfloat[]{

\begin{tikzpicture}  [node distance=48pt,on grid,>={Stealth[length=4pt,round]},bend angle=45, inner sep=0pt,      pre/.style={<-,shorten <=1pt,>={Stealth[round]},semithick},    post/.style={->,shorten >=1pt,>={Stealth[round]},semithick},  unfrozen/.style= {circle,inner sep=1pt,minimum size=1pt,draw=black!100,fill=red!50},  frozen/.style={rectangle,inner sep=1pt,minimum size=12pt,draw=black!75,fill=cyan!50},   point/.style= {circle,inner sep=0pt, outer sep=1.5pt,minimum size=1.5pt,draw=black!100,fill=black!100},   boundary/.style={-,draw=cyan},   internal/.style={-,draw=red},    every label/.style= {black}]     \node at (-42.155pt,65pt) {$-\underline{\zeta}$}; \node at (-64.915pt,-5pt) {$\underline{\eta}$};
  
                                  \node[point] (v1) at (-44.7175pt,21.4225pt) {};    \node[point] (v2) at (-19.225pt,21.4225pt) {};        \node[point] (v3) at (-7.76pt,21.4225pt) {};    \node[point] (v4) at (9.155pt,21.4225pt) {};            \node[point] (z1) at (-32.845pt,37.155pt) {};    \node[point] (z2) at (0.05,37.155pt) {};        
 \draw[-,draw=cyan!100] (-61.3375pt,21.4225pt)--(v1); \draw[-,draw=cyan!100] (28.62pt,21.4225pt)--(v4);
 \draw[-,draw=cyan!100] (-59.915pt,37.155pt)--(z1);  \draw[-,draw=cyan!100] (28.62pt,37.155pt)--(z2);  
\draw[-] (11.55pt,60pt) node (v6) {}--(30.1275pt,60pt) node (v7) {};
\draw[-] (-1.15,60pt) node (t1){}--(-60pt,0) node (b1){}; \draw[-] (t1)--(-40pt,0) node(b2){}; \draw[-] (t1)--(-22.845pt,0) node(b3){}; \draw[-] (t1)--(0,0) node (b4){}; \draw[-] (b3)--(b4); \draw[-] (b2)--(b3); \draw[-] (b1)--(b2); \draw[-] (t1)--(-8.45pt,60pt) node(t2){};
\draw[-] (t2) node (v9) {}--(v6);
\node at (-52.7175pt,45pt) {\color{cyan}$\binom{2}{-1}$};
\node at (-56.7725pt,13.5775pt) {\color{cyan}$\binom{1}{-1}$}; \node at (-35.6475pt,13.5775pt) {\color{red}$\binom{1}{0}$}; \node at (-15.69pt,45pt) {\color{red}$\binom{2}{0}$}; \node at (-0.5,13.5775pt) {\color{red}$\binom{1}{1}$}; \node at (0.705pt,13.5775pt) {\color{red}$\binom{1}{2}$}; \node at (14.9pt,45pt) {\color{cyan}$\binom{2}{1}$}; \node at (17.7875pt,13.5775pt) {\color{cyan}$\binom{1}{3}$};
\node at (-72.8025pt,21.4225pt) {\color{cyan}$L_1$}; \node at (-72.8025pt,37.155pt) {\color{cyan}$L_2$};
\draw[<-,draw=teal!100] (-11.62pt,22.4225pt)  node (v16) {} .. controls (-6.31pt,25.4225pt) .. (-4.5775pt,22.4225pt) node (v13) {}; \draw[<-,draw=teal!100] (6.1825pt,22.4225pt) node (v15) {} .. controls (9.5775pt,25.4225pt) .. (18.9575pt,22.4225pt) node (v18) {}; \draw[->,draw=teal!100] (-23.6475pt,22.4225pt) node (v12) {} .. controls (-17.4075pt,25.4225pt) .. (v16); \draw[->,draw=teal!100] (-50.675pt,22.4225pt) node (v8) {} .. controls (-45.8575pt,25.4225pt) .. (-39.3225pt,22.4225pt) node (v5) {};
\draw[<-,draw=teal!100] (-2.3525pt,35.845pt) node (v14) {} .. controls (2.465pt,32.845pt) .. (4.7325pt,35.845pt) node (v16) {}; \draw[->,draw=teal!100] (-39.295pt,35.845pt) node (v10) {} .. controls (-33.055pt,32.845pt) .. (-25.0975pt,35.845pt) node (v11) {};
       \node at (-22.7175pt,65pt) {-1};    \node at (-2.7175pt,65pt) {-2};    \node at (17.2825pt,65pt) {-1};            \node at (-30pt,-5pt){2};    \node at (-10pt,-5pt){1};    \node at (-50pt,-5pt){1};           \draw  (v9) edge (b4); \draw  (v6) edge (b4); \draw  (v7) edge (b4);
\draw[-,draw=cyan!100]   (z2) edge (z1); \draw[-,draw=cyan!100]   (v1) edge (v2); \draw[-,draw=cyan!100]  (v2) edge (v3); \draw[-,draw=cyan!100]   (v3) edge (v4);
\draw[->,teal]  (v5) edge (v10); \draw[->,teal]  (v11) edge (v12); \draw[<-,teal]  (v13) edge (v14); \draw[->,teal]  (v15) edge (v16); \node (v16) at (-50.69pt,35.845pt) {}; \node (v17) at (18.38pt,35.845pt) {}; \draw[->,dotted,teal]  (v16) edge (v8); \draw[->,dotted,teal]  (v17) edge (v18);
\node (v19) at (-11.775pt,35.845pt) {}; \node (v20) at  (-11.62pt,22.4225pt) {};
\draw[<-,teal]  (v19) edge (v20); 
\end{tikzpicture}}

\subfloat[]{

\begin{tikzpicture}  [node distance=48pt,on grid,>={Stealth[length=4pt,round]},bend angle=45, inner sep=0pt,      pre/.style={<-,shorten <=1pt,>={Stealth[round]},semithick},    post/.style={->,shorten >=1pt,>={Stealth[round]},semithick},  unfrozen/.style= {circle,inner sep=1pt,minimum size=1pt,draw=black!100,fill=red!50},  frozen/.style={rectangle,inner sep=1pt,minimum size=12pt,draw=black!75,fill=cyan!50},   point/.style= {circle,inner sep=0pt, outer sep=1.5pt,minimum size=1.5pt,draw=black!100,fill=black!100},   boundary/.style={-,draw=cyan},   internal/.style={-,draw=red},    every label/.style= {black}]     
\node[frozen] (q-2) at (99.575pt,42.675pt) {$\binom{2}{-1}$};
\node[frozen] (q-1) at (99.575pt,14.225pt) {$\binom{1}{-1}$}; \node[unfrozen] (q1) at (128.025pt,14.225pt) {$\binom{1}{0}$}; \node[unfrozen] (q2) at (142.25pt,42.675pt) {$\binom{2}{0}$}; \node[unfrozen] (q3) at (5.5,14.225pt) {$\binom{1}{1}$}; \node[unfrozen] (q4) at (184.925pt,14.225pt) {$\binom{1}{2}$}; \node[frozen] (q5) at (199.15pt,42.675pt) {$\binom{2}{1}$}; \node[frozen] (q6) at (213.375pt,14.225pt) {$\binom{1}{3}$};
\draw[->,dotted,teal]  (q-2) edge (q-1); \draw[->,dotted,teal]  (q5) edge (q6);
\draw[->,teal]  (q-1) edge (q1); \draw[->,teal]   (q1) edge (q-2); \draw[->,teal]   (q-2) edge (q2); \draw[->,teal]   (q2) edge (q1); \draw [->,teal]  (q1) edge (q3); \draw [->,teal]  (q4) edge (q3); \draw [->,teal]  (q2) edge (q4); \draw [->,teal]  (q5) edge (q2); \draw [->,teal]  (q3) edge (q2); \draw [->,teal]  (q4) edge (q5); \draw [->,teal]  (q6) edge (q4); \end{tikzpicture}}$\qquad$\subfloat[]{

\begin{tikzpicture}  [node distance=48pt,on grid,>={Stealth[length=4pt,round]},bend angle=45, inner sep=0pt,      pre/.style={<-,shorten <=1pt,>={Stealth[round]},semithick},    post/.style={->,shorten >=1pt,>={Stealth[round]},semithick},  unfrozen/.style= {circle,inner sep=1pt,minimum size=1pt,draw=black!100,fill=red!50},  frozen/.style={rectangle,inner sep=1pt,minimum size=12pt,draw=black!75,fill=cyan!50},   point/.style= {circle,inner sep=0pt, outer sep=1.5pt,minimum size=1.5pt,draw=black!100,fill=black!100},   boundary/.style={-,draw=cyan},   internal/.style={-,draw=red},    every label/.style= {black}]     
\node[frozen] (q-2) at (99.575pt,42.675pt) {$-2$};
\node[frozen] (q-1) at (99.575pt,14.225pt) {$-1$}; \node[unfrozen] (q1) at (128.025pt,14.225pt) {$1$}; \node[unfrozen] (q2) at (142.25pt,42.675pt) {$2$}; \node[unfrozen] (q3) at (5.5,14.225pt) {$3$}; \node[unfrozen] (q4) at (184.925pt,14.225pt) {$4$}; \node[frozen] (q5) at (199.15pt,42.675pt) {$5$}; \node[frozen] (q6) at (213.375pt,14.225pt) {$6$};
\draw[->,dotted,teal]  (q-2) edge (q-1); \draw[->,dotted,teal]  (q5) edge (q6);
\draw[->,teal]  (q-1) edge (q1); \draw[->,teal]   (q1) edge (q-2); \draw[->,teal]   (q-2) edge (q2); \draw[->,teal]   (q2) edge (q1); \draw [->,teal]  (q1) edge (q3); \draw [->,teal]  (q4) edge (q3); \draw [->,teal]  (q2) edge (q4); \draw [->,teal]  (q5) edge (q2); \draw [->,teal]  (q3) edge (q2); \draw [->,teal]  (q4) edge (q5); \draw [->,teal]  (q6) edge (q4); \end{tikzpicture}}
\end{figure}

\end{Eg}

\subsection{Quantum seeds associated to signed words}

Successively label the triangles in the triangulation $\Delta$ of
$\Sigma_{\ueta}^{\uzeta}$ by $T_{1},T_{2},\ldots$ from left to right,
we obtain $\ubi=(\bi_{1},\bi_{2},\ldots):=(\bi_{T_{1}},\bi_{T_{2}},\ldots)$.
Then $\ubi$ is a (riffle) shuffle of the signed words $-\uzeta$
and $\ueta$. Conversely, any signed word $\ubi$ determines a triangulation
$\Delta$. For example, $\ubi=(1,2,1,-1,-2,-1)$ in Figure \ref{fig:triangulated_trapezoid}.
Therefore, we can denote $\dsd(\uzeta,\ueta,\Delta)=\dsd(\ubi)$ and
$\rsd(\uzeta,\ueta,\Delta)=\rsd(\ubi)$. 

We identify the set $\ddI$ of intervals with $(-J)\sqcup I(\ubi)$,
sending $\binom{a}{-1}$ to $-a$ and $\binom{a}{d}$ to $\binom{a}{d}^{\ubi}$,
$\forall a\in J,d\in[0,O^{\ubi}(a)-1]$. Then we have $I=I(\ubi)\simeq[1,l]$,
such that the interval $\ell_{|\bi_{k}|}^{+}$ associated with $T_{k}$
is identified with $\binom{|\bi_{k}|}{o_{-}(k)}^{\ubi}$ and $k$.
We will denote the elements in $I$ by $\binom{a}{d}^{\ubi}$ and
$k\in[1,l]$ interchangeably. We also denote $\binom{a}{d}^{\ubi}$
by $\binom{a}{d}^{\rsd(\ubi)}$, $\binom{a}{d}^{\dsd(\ubi)}$, or
simply $\binom{a}{d}$.

Define $\varepsilon_{k}:=\sign(T_{k})$, $\forall k\in[1,l]$. Extend
$[1]:I\rightarrow I\sqcup\{+\infty\}$ to a map from $\ddI$ to $I\sqcup\{+\infty\}$
such that $\binom{a}{d}[1]:=\binom{a}{d+1}$, $\forall d\in[-1,O(a)-1]$,
where we understand $\binom{a}{O(a)}$ as $+\infty$.

\begin{Lem}\label{lem:dBS_B_matrix}

The matrix entries $b_{jk}$, $(j,k)\in\ddI\times I_{\ufv}$, of $\ddB$
are given as below:
\begin{align}
b_{jk} & =\begin{cases}
\varepsilon_{k} & k=j[1]\\
-\varepsilon_{j} & j=k[1]\\
\varepsilon_{k}C_{|\bi_{j}|,|\bi_{k}|} & \varepsilon_{j[1]}=\varepsilon_{k},\ j<k<j[1]<k[1]\\
\varepsilon_{k}C_{|\bi_{j}|,|\bi_{k}|} & \varepsilon_{k}=-\varepsilon_{k[1]},\ j<k<k[1]<j[1]\\
-\varepsilon_{j}C_{|\bi_{j}|,|\bi_{k}|} & \varepsilon_{k[1]}=\varepsilon_{j},\ k<j<k[1]<j[1]\\
-\varepsilon_{j}C_{|\bi_{j}|,|\bi_{k}|} & \varepsilon_{j}=-\varepsilon_{j[1]},\ k<j<j[1]<k[1]\\
0 & \text{otherwise}
\end{cases}.\label{eq:dBS_B_matrix}
\end{align}

\end{Lem}

\begin{proof}

The case $|\bi_{j}|=|\bi_{k}|$ is trivial. When $|\bi_{j}|\neq|\bi_{k}|$,
the following triangles contribute to $b_{jk}$.

(i) $T_{k}$ such that $j<k<j[1]$: We have $\ell_{|\bi_{k}|}^{+}=k$
and it contributes $\varepsilon_{k}\Hf C_{|\bi_{j}|,|\bi_{k}|}$.

(ii) $T_{k[1]}$ such that $j<k[1]<j[1]$ : We have $\ell_{|\bi_{k}|}^{-}=k$
and contributes $-\varepsilon_{k[1]}\Hf C_{|\bi_{j}|,|\bi_{k}|}$.

(iii) $T_{j}$ such that $k<j<k[1]$: We have $\ell_{|\bi_{j}|}^{+}=j$
and it contributes $-\varepsilon_{j}\Hf C_{|\bi_{j}|,|\bi_{k}|}$.

(iv) $T_{j[1]}$ such that $k<j[1]<k[1]$: We have $\ell_{|\bi_{j}|}^{-}=j$
and it contributes $\varepsilon_{j[1]}\Hf C_{|\bi_{j}|,|\bi_{k}|}$.

Note that (i) and (iii) cannot appear simultaneously, and similar
for (ii) and (iv). Then (i) and (iv) sum to the 3rd identity, (i)
and (ii) sum to the 4th identity, (iii) and (ii) sum to the 5th identity,
and (iii) and (iv) sum to the 6th identity.

\end{proof}

Let $\uc$ denote a word in which each $a$ appears exactly once,
where $a\in J$, called a Coxeter word. By Lemma \ref{lem:dBS_B_matrix},
we have $\ddB_{\ddI\times I_{\ufv}}=-\tB$ for the matrix $\tB$ associated
to $(-\uc,\ubi)$ in \cite[(8.7)]{BerensteinZelevinsky05}. Then \cite[(8.5)(8.9)]{BerensteinZelevinsky05},
after a sign change, provides us a quantization matrix $\ddLambda$
for $\dsd(\ubi)$. See Example \ref{eg:global-crystal-sl2} or \cite[Example 3.2]{BerensteinZelevinsky05}.
On the other hand, $\rsd(\ubi)$ has quantization by the following
Lemma.

\begin{Lem}\label{lem:dBS-fullrank}

The matrix $\tB(\rsd(\ubi))$ is of full rank.

\end{Lem}

\begin{proof}

(\ref{eq:dBS_B_matrix}) implies $\deg^{\rsd}y_{k}=-\varepsilon_{k[1]}f_{k[1]}+\sum_{i<k[1]}b_{ik}f_{i}$,
$\forall k\in I_{\ufv}$. Moreover, $[1]$ is injective on $I_{\ufv}$.
We deduce that $\deg^{\rsd}y_{k}$, $k\in I_{\ufv}$, are linearly
independent. The desired claim follows. Alternatively, the claim is
implied by \cite[Section 8.1]{casals2022cluster}.

\end{proof}

\subsection{Seeds operations\label{subsec:change-string-diagrams}}

\subsubsection*{Mutations and reflections}

Let $\varepsilon$ denote any sign. By \cite[Proposition 3.7]{shen2021cluster},
the following operations on $\sd=\dsd(\uzeta,\ueta,\Delta)=\dsd(\ubi)$
will generate new seeds $\sd'=\dsd(\uzeta,\ueta,\Delta')=\dsd(\ubi')$:
\begin{itemize}
\item (Flips) Change $\ubi=(\ldots,\varepsilon a,-\varepsilon b,\ldots)$
to $\ubi'=(\ldots,-\varepsilon b,\varepsilon a,\ldots)$, where $a,b\in J$,
$\varepsilon a=\bi_{k}$. Then $\sd'=\mu_{k}\sd$ when $a=b$, and
$\sd'=\sd$ otherwise. Note that the associated string diagrams are
related by the diagonal flip of the quadrilateral $T_{k}\cup T_{k+1}$:
\\
\begin{tikzpicture}  [scale=0.8,node distance=48pt,on grid,>={Stealth[length=4pt,round]},bend angle=45, inner sep=0pt,      pre/.style={<-,shorten <=1pt,>={Stealth[round]},semithick},    post/.style={->,shorten >=1pt,>={Stealth[round]},semithick},  interval/.style= {blue},    quiver/.style={teal},  unfrozen/.style= {circle,inner sep=1pt,minimum size=1pt,blue,draw=black!100,fill=red!50},  frozen/.style={rectangle,inner sep=1pt,minimum size=12pt,blue,draw=black!75,fill=cyan!50},   point/.style= {circle,inner sep=0pt, outer sep=1.5pt,minimum size=1.5pt,draw=black!100,fill=black!100},   boundary/.style={-,draw=cyan},   internal/.style={-,draw=red},    every label/.style= {black}]        
\node (v1) at (-3.15,1.5) {}; \node (v2) at (-3.15,0) {}; \node (v4) at (-1.65,1.5) {};
\node (v3) at (-1.65,0) {}; \draw  (v1) edge (v2); \draw  (v2) edge (v3); \draw  (v3) edge (v4); \draw  (v4) edge (v1); \draw  (v2) edge (v4); \node at (-2.4,-0.2) {a}; \node at (-2.4,1.7) {-a};
\node (w1) at (-6,1.5) {}; \node (w2) at (-6,0) {}; \node (w4) at (-4.5,1.5) {}; \node (w3) at (-4.5,0) {}; \node at (-5.15,-0.2) {a}; \node at (-5.2,1.7) {-a};
\draw  (w1) edge (w2); \draw  (w2) edge (w3); \draw  (w3) edge (w4); \draw  (w4) edge (w1); \draw  (w1) edge (w3);
\node at (-3.75,0.7) {$\longleftrightarrow$}; \node (v5) at (-6.15,0.5) {}; \node (v8) at (-4.3,0.5) {}; \node (v9) at (-3.3,0.5) {}; \node (v12) at (-1.5,0.5) {}; \node (v13) at (-6.15,1) {}; \node (v14) at (-4.3,1) {}; \node (v15) at (-3.3,1) {}; \node (v16) at (-1.5,1) {}; \node [point] (v6) at (-5.6,0.5) {}; \node [point] (v7) at (-4.8,0.5) {}; \node [point] (v10) at (-2.8,0.5) {}; \node [point] (v11) at (-2,0.5) {}; \draw[interval]  (v5) edge (v6); \draw[interval]  (v6) edge node[below,yshift=-1.2]{$k$} (v7); \draw[interval]  (v7) edge (v8); \draw[interval]  (v9) edge (v10); \draw[interval]  (v10) edge node[below,yshift=-1.2]{$k$} (v11); \draw[interval]  (v11) edge (v12); \draw[interval]  (v13) edge (v14); \draw[interval]  (v15) edge (v16); \node[quiver] at (-3.75,1) {$\mu_{\ell}$}; \node[interval] at (-6.65,0.5) {$L_a$}; \node[interval] at (-6.65,1) {$L_b$}; \node (v18) at (-5.2,1) {}; \node (v17) at (-5.2,0.5) {}; \node (v20) at (-5.85,0.5) {}; \node (v19) at (-5.85,1) {}; \node (v21) at (-4.55,1) {}; \node (v22) at (-4.55,0.5) {}; \node (v24) at (-2.35,0.5) {}; \node (v23) at (-2.35,1) {}; \node (v27) at (-1.75,0.5) {}; \node (v28) at (-1.75,1) {}; \node (v25) at (-3,0.5) {}; \node (v26) at (-3,1) {}; \draw[->,quiver]  (v17) edge (v18); \draw[->,quiver,dashed]  (v19) edge (v20); \draw[->,quiver,dashed]  (v21) edge (v22); \draw[->,quiver] (v20) .. controls (-5.6,0.75) .. (v17); \draw[->,quiver] (v22) .. controls (-4.8,0.75) .. (v17); \draw[->,quiver]  (v23) edge (v24); \draw[->,quiver,dashed]  (v25) edge (v26); \draw[->,quiver,dashed]  (v27) edge (v28); \draw[<-,quiver] (v25) .. controls (-2.7,0.75) .. (v24); \draw[<-,quiver] (v27) .. controls (-2,0.75) .. (v24); \end{tikzpicture} $\qquad$\begin{tikzpicture}  [scale=0.8,node distance=48pt,on grid,>={Stealth[length=4pt,round]},bend angle=45, inner sep=0pt,      pre/.style={<-,shorten <=1pt,>={Stealth[round]},semithick},    post/.style={->,shorten >=1pt,>={Stealth[round]},semithick},  interval/.style= {blue},    quiver/.style={teal},  unfrozen/.style= {circle,inner sep=1pt,minimum size=1pt,blue,draw=black!100,fill=red!50},  frozen/.style={rectangle,inner sep=1pt,minimum size=12pt,blue,draw=black!75,fill=cyan!50},   point/.style= {circle,inner sep=0pt, outer sep=1.5pt,minimum size=1.5pt,draw=black!100,fill=black!100},   boundary/.style={-,draw=cyan},   internal/.style={-,draw=red},    every label/.style= {black}]        
\node (v1) at (-3.15,1.5) {}; \node (v2) at (-3.15,0) {}; \node (v4) at (-1.65,1.5) {};
\node (v3) at (-1.65,0) {}; \draw  (v1) edge (v2); \draw  (v2) edge (v3); \draw  (v3) edge (v4); \draw  (v4) edge (v1); \draw  (v2) edge (v4); \node at (-2.4,-0.2) {a}; \node at (-2.4,1.7) {-b};
\node (w1) at (-6,1.5) {}; \node (w2) at (-6,0) {}; \node (w4) at (-4.5,1.5) {}; \node (w3) at (-4.5,0) {}; \node at (-5.15,-0.2) {a}; \node at (-5.2,1.7) {-b};
\draw  (w1) edge (w2); \draw  (w2) edge (w3); \draw  (w3) edge (w4); \draw  (w4) edge (w1); \draw  (w1) edge (w3);
\node at (-3.75,0.7) {$=$}; \node (v5) at (-6.15,0.5) {}; \node (v8) at (-4.3,0.5) {}; \node (v9) at (-3.3,0.5) {}; \node (v12) at (-1.5,0.5) {}; \node (v13) at (-6.15,1) {}; \node[point] (v14) at (-4.8,1) {}; \node (v15) at (-3.3,1) {}; \node (v16) at (-1.5,1) {}; \node [point] (v6) at (-5.6,0.5) {}; \node (v7) at (-4.3,1) {}; \node [point] (v10) at (-2.8,1) {}; \node [point] (v11) at (-2,0.5) {}; \draw[interval]  (v5) edge (v6); \draw[interval]  (v14) edge  (v7); \draw[interval]  (v6) edge (v8); \draw[interval]  (v9) edge (v11); \draw[interval]  (v11) edge (v12); \draw[interval]  (v13) edge (v14); \draw[interval]  (v15) edge (v10); \draw[interval]  (v16) edge (v10);
\node[interval] at (-6.65,0.5) {$L_a$}; \node[interval] at (-6.65,1) {$L_b$}; \node (v18) at (-5.2,1) {}; \node (v17) at (-5.2,0.5) {}; \node (v20) at (-5.85,0.5) {}; \node (v19) at (-5.85,1) {}; \node (v21) at (-4.55,1) {}; \node (v22) at (-4.55,0.5) {}; \node (v24) at (-2.35,0.5) {}; \node (v23) at (-2.35,1) {}; \node (v27) at (-1.75,0.5) {}; \node (v28) at (-1.75,1) {}; \node (v25) at (-3,0.5) {}; \node (v26) at (-3,1) {}; \draw[->,quiver,dashed]  (v19) edge (v20); \draw[<-,quiver,dashed]  (v21) edge (v22); \draw[->,quiver] (v20) .. controls (-5.6,0.75) .. (v17); \draw[->,quiver] (v21) .. controls (-4.8,0.75) .. (v18); \draw[<-,quiver,dashed]  (v25) edge (v26); \draw[->,quiver,dashed]  (v27) edge (v28); \draw[<-,quiver] (v26) .. controls (-2.7,0.75) .. (v23); \draw[<-,quiver] (v27) .. controls (-2,0.75) .. (v24); \end{tikzpicture} 
\item (Braid moves) Assume that $\beta_{\ueta}=\beta_{\ueta'}\in\Br^{+}$
and $\ubi_{[j,k]}=\varepsilon\ueta$. Replacing $\ubi_{[j,k]}$ by
$\varepsilon\ueta'$ in $\ubi$, we obtain $\ubi'$. Then $\sd'$
and $\sd$ are related by a sequence of mutations $\seq_{\ubi',\ubi}$
at vertices $r\in[j,k]$ such that $r[1]\leq k$ and a permutation
$\sigma_{\ubi',\ubi}$ on $[j,k]$. We use $\seq_{\ubi',\ubi}^{\sigma}$
to denote this permutation mutation sequence.
\end{itemize}
Therefore, for any given words $\ueta$, $\uzeta$, $\ueta'$, $\uzeta'$,
such that $\beta_{\ueta}=\beta_{\ueta'}$, $\beta_{\uzeta}=\beta_{\uzeta'}$,
and any shuffles $\ubi$ of $(-\uzeta,\ueta)$ and $\ubi'$ of $(-\uzeta',\ueta')$,
the seeds $\dsd(\ubi)$ and $\dsd(\ubi')$ are connected by permutations
and mutations. 

Following \cite[Section 2.3]{shen2021cluster}, we introduce the reflection
operations which will generate new seeds $\dsd(\ubi')=\dsd(\uzeta',\ueta',\Delta')$
similar to $\dsd(\ubi)=\dsd(\uzeta,\ueta,\Delta)$.
\begin{itemize}
\item Left reflection: change $\ubi=(\varepsilon a,\ldots)$ to $\ubi'=(-\varepsilon a,\ldots)$.
\item Right reflection: change $\ubi=(\ldots,\varepsilon a)$ to $\ubi'=(\ldots,-\varepsilon a)$.
\end{itemize}
Note that $\Sigma_{\ueta'}^{\uzeta'}\neq\Sigma_{\ueta}^{\uzeta}$
and $\dsd(\ubi')\neq\dsd(\ubi)$, but we have $\rsd(\ubi')=\rsd(\ubi)$
under left reflections. 

Using flips and left reflections, we obtain $\bClAlg(\rsd(\ubi))=\bClAlg(\rsd(\uzeta\op\ueta))$,
which only depends on $\beta_{\uzeta\op\ueta}\in\Br^{+}$ by braid
moves.

It is worth noting that the seeds connected by these operations provide
the same cluster algebra. More precisely, let there be given signed
words $\ubi^{(i)}$, $i=1,2,3$. Let $\sd^{(i)}$ denote either $\rsd(\ubi^{(i)})$
or $\dsd(\ubi^{(i)})$. In the former case, we assume $\ubi^{(i)}$
and $\ubi^{(j)}$, $\forall i,j\in[1,3]$, are connected by flips,
braid moves, and left reflections; in the latter case, we assume they
are connected by flips and braid moves. Let us choose permutation
mutation sequences $\seq_{\ubi^{(j)},\ubi^{(i)}}^{\sigma}$ associated
to these operations such that $\sd^{(j)}=\seq_{\ubi^{(j)},\ubi^{(i)}}^{\sigma}\sd^{(i)}$.
Consider the following diagram for $\clAlg^{(i)}:=\clAlg(\sd^{(i)})$.
\begin{align}
\begin{array}{ccc}
\clAlg^{(3)} & \overset{(\seq_{\ubi^{(3)},\ubi^{(1)}}^{\sigma})^{*}}{\xrightarrow{\sim}} & \clAlg^{(1)}\\
\simeqd(\seq_{\ubi^{(3)},\ubi^{(2)}}^{\sigma})^{*} &  & \parallel\\
\clAlg^{(2)} & \overset{(\seq_{\ubi^{(2)},\ubi^{(1)}}^{\sigma})^{*}}{\xrightarrow{\sim}} & \clAlg^{(1)}
\end{array}.\label{eq:mutation-different-words}
\end{align}

Diagram (\ref{eq:mutation-different-words}) is commutative at the
classical level by \cite[Theorem 1.1]{shen2021cluster}. By tracking
the degrees of the quantum cluster variables, which are the same as
those of their classical counterparts, we could easily deduce that
the diagram is commutative at the quantum level as well (details could
be found in \cite[Lemma 3.2]{qin2024infinite}).

\subsubsection*{Subwords}

Take any $1\leq j\leq k\leq l$. Define $\ubi'=(\bi'_{1},\ldots,\bi'_{k-j+1})$
to be $\ubi_{[j,k]}=(\bi_{j},\bi_{j+1},\ldots,\bi_{k})$, $\dsd':=\dsd(\ubi')$,
$\ddI':=I(\dsd')$, $I'_{\ufv}:=I_{\ufv}(\dsd')$, and $I':=I(\rsd(\ubi'))$.

\begin{Def}\label{def:calibration-words}

Define the embedding $\iota_{\ubi,\ubi'}:\ddI'\hookrightarrow\ddI$,
such that $\iota\binom{a}{d}^{\ubi'}=\binom{a}{d+O([1,j-1];a)}^{\ubi}$.

\end{Def}

Abbreviate $\iota_{\ubi,\ubi'}$ by $\iota$. Note that $\iota(I'_{\ufv})\subset I_{\ufv}$,
$\iota(I')\subset I$. Under the identification $I'\simeq[1,k-j+1]$
and $I\simeq[1,l]$, we have $\iota(s)=s+j-1$, $\forall s\in[1,k-j+1]$.
In particular, $s<s'$ if and only if $\iota(s)<\iota(s')$ for $s,s'\in I'$.
We have the following.

\begin{Lem}\label{lem:calibration-word}

The embedding $\iota$ is a cluster embedding from $\dsd(\ubi')$
to $\dsd(\ubi)$, such that $\dsd(\ubi')$ is a good sub seed of $\dsd(\ubi)$.
It restricts to a cluster embedding $\iota$ from $\rsd(\ubi')$ to
$\rsd(\ubi)$. When $\ubi'=\ubi_{[1,k]}$, $\rsd(\ubi')$ is a good
sub seed of $\rsd(\ubi)$.

\end{Lem}

\begin{proof}

Let $\Sigma(\ubi)$ denote the triangulated trapezoid associated with
$\ubi$, whose triangles are $T_{1},\ldots,T_{l(\ubi)}$. Then $T_{j},\ldots,T_{k}$
form a triangulated trapezoid $\Sigma(\ubi')$ associated with $\ubi'$,
such that $T_{r}$, $j\leq r\leq k$, are identified with the triangle
$T'_{r-j+1}$ of $\Sigma(\ubi')$. Under this identification, the
interval $\binom{a}{d}^{\ubi'}$ only overlaps $\iota\binom{a}{d}^{\ubi'}$.
Summing up the contributions from $T_{1},\ldots,T_{l(\ubi)}$, we
deduce $b'_{jk}=b_{\iota(j),\iota(k)}$, $\forall j\in\ddI',k\in I'_{\ufv}$
and, moreover, $b_{j',\iota(k)}=0$ $\forall j'\in\ddI\backslash\iota(\ddI')$.
Therefore, $\iota$ is a cluster embedding from $\dsd(\ubi')$ to
$\dsd(\ubi)$ such that $\dsd(\ubi')$ is a good sub seed, and it
restricts to a cluster embedding from $\rsd(\ubi')$ to $\rsd(\ubi)$.

If $\ubi'=\ubi_{[1,k]}$, $I\backslash\iota(I')$ equals $\{j'|j'>k\}=\ddI\backslash\iota(\ddI')$.
So $\rsd(\ubi')$ is a good sub seed.

\end{proof}

\begin{Eg}

Take $\uzeta=\ueta=(1,2,1)$. The quiver associated to $\dsd(\ubi)$,
$\ubi:=(\ueta,-\uzeta)$ is drawn in Figure \ref{fig:triangulated_trapezoid}.
The quiver associated to $\dsd(\ubi')$, $\ubi':=(\ueta_{[2,3]},-\uzeta)$
is drawn in Figure \ref{fig:quiver_codim1_cell}. Moreover, $\dsd(\ubi')$
is a good seed of $\dsd(\ubi)$ under the cluster embedding $\iota_{\ubi,\ubi'}$,
such that $\iota_{\ubi,\ubi'}(-1)=-1$, $\iota_{\ubi,\ubi'}(-2)=-2$,
and $\iota_{\ubi,\ubi'}(s)=s+1$, $\forall s\in[1,5]$.

\begin{figure}[h]
\caption{The quiver for a seed of $\kk[G^{w_{0},s_{1}w_{0}}]$, $w_{0}=s_{1}s_{2}s_{1}$}
\label{fig:quiver_codim1_cell}

\begin{tikzpicture}  [node distance=48pt,on grid,>={Stealth[length=4pt,round]},bend angle=45, inner sep=0pt,      pre/.style={<-,shorten <=1pt,>={Stealth[round]},semithick},    post/.style={->,shorten >=1pt,>={Stealth[round]},semithick},  unfrozen/.style= {circle,inner sep=1pt,minimum size=1pt,draw=black!100,fill=red!50},  frozen/.style={rectangle,inner sep=1pt,minimum size=12pt,draw=black!75,fill=cyan!50},   point/.style= {circle,inner sep=0pt, outer sep=1.5pt,minimum size=1.5pt,draw=black!100,fill=black!100},   boundary/.style={-,draw=cyan},   internal/.style={-,draw=red},    every label/.style= {black}]     
\node[frozen] (q-2) at (115.1375pt,42.155pt) {-2};
\node[frozen] (q1) at (135.0525pt,16.4225pt) {-1}; \node[unfrozen] (q2) at (150.7425pt,42.155pt) {1}; \node[unfrozen] (q3) at (5.45,16.4225pt) {2}; \node[unfrozen] (q4) at (171.405pt,16.4225pt) {3}; \node[frozen] (q5) at (182.755pt,42.155pt) {4}; \node[frozen] (q6) at (188.4875pt,16.4225pt) {5};
\draw[->,dashed,,teal]  (q5) edge (q6);
\draw[->,dashed,teal]   (q1) edge (q-2); \draw[->,teal]   (q-2) edge (q2); \draw[->,teal]   (q2) edge (q1); \draw [->,teal]  (q1) edge (q3); \draw [->,teal]  (q4) edge (q3); \draw [->,teal]  (q2) edge (q4); \draw [->,teal]  (q5) edge (q2);
\draw [->,teal]  (q4) edge (q5); \draw [->,teal]  (q6) edge (q4);
\draw [->,teal]  (q3) edge (q2);
\end{tikzpicture}
\end{figure}

\end{Eg}

\subsubsection*{Extensions}

Let $\tJ$ be any subset of $\Z_{>0}$ containing $J$ and $\tC$
any $\tJ\times\tJ$ generalized Cartan matrix, $\ubti$ any signed
word in the letters from $\pm\tJ$.

\begin{Def}[Extension]

If $\tC_{J\times J}=C$, $\tC$ is called an extension of $C$. $\ubti$
is called a letter extension of a signed word $\ubi$, if $\ubi$
can be obtained from $\ubti$ by removing the letters from $\pm\tJ\backslash J$. 

\end{Def}

Assume $\tC$ is an extension of $C$ and $\ubti$ a letter extension
of $\ubi$ from now on. Denote $\dsd(\ubti)=\dsd(\widetilde{\uzeta},\widetilde{\ueta},\widetilde{\Delta})$.
Then the string diagram $D(\uzeta,\ueta,\Delta)$ is contained in
$D(\widetilde{\uzeta},\widetilde{\ueta},\widetilde{\Delta})$. Note
that a vertex of $\dsd(\ubti)$ is either a vertex of $\dsd(\ubi)$,
or an interval on $L_{c}$ for $c\in\tJ\backslash J$. We have the
following observation.

\begin{Lem}\label{lem:extension_diagram}

The matrix $\ddB(\dsd(\ubi))$ equals $\ddB(\dsd(\ubti))_{\ddI\times\ddI}$.

\end{Lem}

The string diagram (or an associated quiver) of $\ubti$ is called
a layer extension of that of $\ubi$, see Figure \ref{fig:quiver_extension}.
Let $\tW$ denote the Weyl group associated to $\tC$. Then $W$ is
a subgroup of $\tW$.

\begin{Lem}[Extension of words]\label{lem:extension_words}

For any word $\ueta$ in $J$, there exist an extended generalized
Cartan matrix $\tC$ and a reduced word $\widetilde{\ueta}$ for $\tW$,
such that $\widetilde{\ueta}$ is a letter extension of $\ueta$.

\end{Lem}

\begin{proof}

Choose $c>\max J$. Extend $C$ to a $(J\sqcup\{c\})\times(J\sqcup\{c\})$-matrix
Cartan $\tC$ such that $C_{a,c}C_{a,c}\geq4$, $\forall a\in J$.
$\widetilde{\ueta}:=(\eta_{1},c,\eta_{2},c,\eta_{3},\ldots,c,\eta_{l})$
meets our expectation.

\end{proof}

\begin{Eg}

Take $\ueta=(1,1,1,1)$ where $C=\left(2\right)_{\{1\}\times\{1\}}$.
The quiver for the seed $\dsd(\ueta)$ is drawn in Figure \ref{fig:quiver_extension}(A).
Choose the letter extension $\widetilde{\ueta}=(1,2,1,2,1,2,1,2)$
for $\tC=\left(\begin{array}{cc}
2 & -2\\
-2 & 2
\end{array}\right)_{[1,2]\times[1,2]}$. The quiver for the seed $\dsd(\widetilde{\ueta})$ is drawn in Figure
\ref{fig:quiver_extension}(B). By removing its vertices $5,6,7,8$
(they correspond to the intervals on $L_{2}$), we recover Figure
\ref{fig:quiver_extension}(A). 

\end{Eg}

\begin{figure}[h]
\caption{A quiver and an extension}
\label{fig:quiver_extension}

\subfloat[$\dsd(1,1,1,1)$]{

 \begin{tikzpicture}  [scale=0.7,node distance=48pt,on grid,>={Stealth[round]},bend angle=45,      pre/.style={<-,shorten <=1pt,>={Stealth[round]},semithick},    post/.style={->,shorten >=1pt,>={Stealth[round]},semithick},  unfrozen/.style= {circle,inner sep=1pt,outer sep=2pt,minimum size=12pt,draw=black!100,fill=red!100},  frozen/.style={rectangle,inner sep=1pt,outer sep=2pt,minimum size=12pt,draw=black!75,fill=cyan!100},   point/.style= {circle,inner sep=1pt,minimum size=5pt,draw=black!100,fill=black!100},   boundary/.style={-,draw=cyan},   internal/.style={-,draw=red},    every label/.style= {black}]        \node[frozen](v1) at (0,0){1};    \node[unfrozen](v2) at (40pt,0){2};    \node[unfrozen](v3) at (80pt,0){3};    \node[frozen](v4) at (120pt,0){4};     \draw[->] (v1) -- (v2); \draw[->] (v2) -- (v3); \draw[->] (v3) -- (v4);
 \end{tikzpicture}}$\qquad$\subfloat[$\dsd(1,2,1,2,1,2,1,2)$]{

\begin{tikzpicture}  [scale=0.7,node distance=48pt,on grid,>={Stealth[round]},bend angle=45,      pre/.style={<-,shorten <=1pt,>={Stealth[round]},semithick},    post/.style={->,shorten >=1pt,>={Stealth[round]},semithick},  unfrozen/.style= {circle,inner sep=1pt,outer sep=2pt,minimum size=12pt,draw=black!100,fill=red!100},  frozen/.style={rectangle,inner sep=1pt,outer sep=2pt,minimum size=12pt,draw=black!75,fill=cyan!100},   point/.style= {circle,inner sep=1pt,minimum size=5pt,draw=black!100,fill=black!100},   boundary/.style={-,draw=cyan},   internal/.style={-,draw=red},    every label/.style= {black}]        \node[frozen](v1) at (0,0){5};    \node[unfrozen](v2) at (40pt,0){6};    \node[unfrozen](v3) at (80pt,0){7};    \node[frozen](v4) at (120pt,0){8};        \node[frozen](v5) at (-20pt,-40pt){1};    \node[unfrozen](v6) at (20pt,-40pt){2};    \node[unfrozen](v7) at (60pt,-40pt){3};    \node[frozen](v8) at (100pt,-40pt){4};         \draw[->] (v1) -- (v2); \draw[->] (v2) -- (v3); \draw[->] (v3) -- (v4);
\draw[->] (v5) -- (v6); \draw[->] (v6) -- (v7); \draw[->] (v7) -- (v8);
\draw[->] (v6.115)--(v1.275); \draw[->] (v6.90)--(v1.300) ;
\draw[->]  (v2.265)--(v6.65) ; \draw[->] (v2.240)--(v6.85) ;
\draw[->] (v7.115)--(v2.275); \draw[->] (v7.90)--(v2.300) ;
\draw[->]  (v3.265)--(v7.65) ; \draw[->] (v3.240)--(v7.85) ;
\draw[->] (v8.115)--(v3.275); \draw[->] (v8.90)--(v3.300) ;
 \end{tikzpicture}}
\end{figure}

From now on, let $\ubi$ be a shuffle of $(-\uzeta,\ueta)$ and denote
$\uxi:=(\uzeta\op,\ueta)$.

\begin{Prop}\label{prop:extension_string_diagram}

There exists a mutation sequence $\seq$ such that $\rsd(\ubi)=\seq\rsd(\uxi)$.
Moreover, we can find an extended Cartan matrix $\tC$, a letter extension
$\widetilde{\uxi}$ of $\uxi$, and $F\subset I_{\ufv}(\rsd(\widetilde{\uxi}))$,
such that $\widetilde{\uxi}$ is reduced, and the classical seed $\rsd(\uxi)$
can be obtain from $\frz_{F}\rsd(\widetilde{\uxi})$ by removing frozen
vertices.

\end{Prop}

\begin{proof}

Note that $\rsd(\ubi)$ can be obtained from $\rsd(\uxi)$ by iterating
flips (mutations) and left reflections (who act trivially on the seeds),
see \cite[Proposition 4.3]{cao2022exchange}.

By Lemma \ref{lem:extension_words}, we can find an extended $\tJ\times\tJ$
generalized Cartan matrix $\tC$ and a letter extension $\widetilde{\uxi}$
of the word $\uxi$ which is reduced for the Weyl group $\tW$. Note
that $\tB(\rsd(\uxi)$) equals $\tB(\rsd(\widetilde{\uxi}))_{I\times I{}_{\ufv}}$
by Lemma \ref{lem:extension_diagram}. Therefore, $\rsd(\uxi)$ can
be obtained from $\rsd(\widetilde{\uxi})$ by freezing and then removing
the vertices on the lines $L_{c}$, $c\in\tJ\backslash J$.

\end{proof}

\begin{Thm}\label{thm:bases_dBS}

Let $\sd$ denote the seed $\rsd(\ubi)$ or the seed $\dsd(\ubi)$
arising from (decorated) double Bott-Samelson cells. The following
statements hold for $\upClAlg(\sd)$:

(1) It possesses the common triangular basis, denoted $\can$.

(2) When $C$ is symmetric, $(\upClAlg(\sd),\can)$ admits quasi-categorification.

\end{Thm}

\begin{proof}

Since $\rsd(\ubi)$ are $\dsd(\ubi)$ are similar, it suffices to
check the claim for $\sd=\rsd(\ubi)$.

Take a shuffle $\ubi$ of $(-\uzeta,\ueta)$. Denote $\uxi=(\uzeta\op,\ueta)$.
Then $\rsd(\uxi)=\seq\rsd(\ubi)$ for a mutation sequence $\seq$,
see Proposition \ref{prop:extension_string_diagram}. Let $\widetilde{\uxi}$
be a letter extension of $\uxi$ as in Proposition \ref{prop:extension_string_diagram}
and denote $w=w_{\widetilde{\uxi}}\in\tW$. Note that, with an appropriate
quantization, $\rsd(\widetilde{\uxi})$ is a quantum seed for the
quantum cluster algebra $\qO[N^{w}]$. The dual canonical basis of
$\qO[N^{w}]$ is the common triangular basis \cite{qin2017triangular},
denoted $\can^{w}$. 

By Proposition \ref{prop:extension_string_diagram}, for some $F\subset I_{\ufv}(\rsd(\widetilde{\uxi}))$,
the classical seed $\sd$ is obtained from $\sd':=\frz_{F}\rsd(\widetilde{\uxi})$
by removing frozen vertices. Note that $\sd$ and $\sd'$ are similar
as classical seeds. In addition, $\upClAlg(\sd')$ has the common
triangular basis $\frz_{F}\can^{w}$ (Theorem \ref{thm:sub_cluster_triangular_basis}).
We deduce that the $\upClAlg(\sd)$ and $\upClAlg(\sd')$ are related
by base changes and quantization changes, see Section \ref{subsec:Base-changes},
(\ref{eq:base-change}):
\begin{align}
\upClAlg(\rsd(\widetilde{\uxi})) & \xrightarrow{\frz_{F}} & \upClAlg(\sd') & \xleftarrow{\text{\ensuremath{\varphi_{1}}}} & \upClAlg(\sd^{\prin},\lambda_{1}) & \stackrel[\text{change}]{\text{quantization}}{\sim} & \upClAlg(\sd^{\prin},\lambda_{2}) & \xrightarrow{\text{\ensuremath{\varphi_{2}}}}\upClAlg(\sd),\label{eq:q-gp-to-dBS}
\end{align}
where $(\sd^{\prin},\lambda_{1})$ and $(\sd^{\prin},\lambda_{2})$
are two principal coefficients seeds with different quantization.
By Proposition \ref{prop:similar-common-tri-basis}, the common triangular
basis $\can$ for $\upClAlg(\sd)$ exists and consists of the elements
similar to those of $\frz_{F}\can^{w}$.

If $C$ is symmetric, we can choose $\tC$ to be symmetric. Then $(\qO[N^{w}],\can^{w})$
is categorified by a monoidal category $\tilde{\cC}^{w}$ consisting
of KLR algebras' modules. Then $(\upClAlg',\frz_{F}\can^{w})$ is
categorified by a monoidal subcategory $\cC'$ of $\tilde{\cC}^{w}$
after localization (Theorem \ref{thm:sub_category_upClAlg}). Therefore,
$(\upClAlg(\sd),\can)$ is quasi-categorified by $\cC'$.

\end{proof}

\begin{Thm}[{Classical $\clAlg=\upClAlg$ by \cite{shen2021cluster}}]\label{thm:A-U-dBS}

For $\sd=\dsd(\ubi)$ or $\rsd(\ubi)$, we have $\clAlg(\sd)=\upClAlg(\sd)$,
and its common triangular basis elements have nice cluster decomposition
in $\seq\sd$, where $\seq$ satisfies $\seq(\rsd(\ubi))=\rsd((\uzeta\op,\ueta))$.

\end{Thm}

\begin{proof}

Recall that $\uxi=(\uzeta\op,\ueta)$ and $\bClAlg(\rsd(\widetilde{\uxi}))=\bUpClAlg(\rsd(\widetilde{\uxi}))=\qO[N(w)]$
(see \cite{GY13}). So $\clAlg(\rsd(\widetilde{\uxi}))=\upClAlg(\rsd(\widetilde{\uxi}))$
and the elements of $\can_{w}$ have nice cluster decompositions in
$\rsd(\tilde{\uxi})$, since $\qO[N(w)]$ has the dual PBW basis.
Note that all algebras in (\ref{eq:q-gp-to-dBS}) have the common
triangular bases related by freezing or similarity. Applying Corollaries
\ref{cor:freezing-A-U} and \ref{cor:base-change-A-U} to $(\clAlg(\rsd(\widetilde{\uxi})),\can_{w})$
in (\ref{eq:q-gp-to-dBS}), we deduce that $\clAlg(\rsd(\uxi))=\upClAlg(\rsd(\uxi))$
and its common triangular basis elements have nice cluster decompositions
in $\rsd(\uxi)$.

The claim for $\seq\dsd(\ubi)$ follows from its similarity with $\rsd(\uxi)$
and Corollary \ref{cor:base-change-A-U}.

\end{proof}

\subsection{Optimized seeds\label{subsec:Optimized-seeds}}

\begin{Lem}\label{lem:optimized_dBS_last}

Denote $\sd=\rsd(\ueta)$ for any word $\ueta$. Then, $\forall k\in I$,
$k^{\max}$ is optimized in $\Sigma_{k^{\min}}\sd$, where $\Sigma_{k^{\min}}:=\mu_{k^{\max}[-1]}\cdots\mu_{k^{min}[1]}\mu_{k^{\min}}$.

\end{Lem}

\begin{proof}

Let us give a short proof using \cite{cao2022exchange}. Extend $\ueta$
to $\ueta'=(\ueta,\eta_{k})$. Then $I(\sd')=I\sqcup\{l(\ueta)+1\}$,
$k^{\max}$ becomes unfrozen in $\sd':=\rsd(\ueta')$, and $\tB(\sd)=\tB(\sd')_{I\times I_{\ufv}}$.
By \cite[Lemma 4.6]{cao2022exchange}, we have $b_{j,k^{\max}}(\Sigma_{k^{\min}}\sd')\leq0$
for all $j\in I_{\ufv}$.\footnote{Our quiver is opposite to that of \cite{cao2022exchange} but the
$B$-matrices are the same.}

\end{proof}

\begin{Prop}\label{prop:optimize_double_sd}

For any signed word $\ubi$, $\dsd(\ubi)$ can be optimized.

\end{Prop}

\begin{proof}

(1) Let us prove the claim for any frozen vertex $k^{\max}$, $k\in I$.
For this purpose, it suffices to show for the seed $\rsd(\ubi)$ that
the frozen $k^{\max}$ can be optimized, which has less frozen vertices.
We can mutate $\rsd(\ubi)$ to a seed of the form $\rsd(-\uzeta,\ueta)$
for words $\uzeta,\ueta$ in $J$. It is connected to $\rsd(\uzeta\op,\ueta)$
by mutations \cite[Proposition 4.3]{cao2022exchange}. Then $k^{\max}$
can be optimized by Lemma \ref{lem:optimized_dBS_last}.

(2) Let us prove the claim for any frozen vertex $-\eta_{k}$, $k\in I$.
Observe that $\dsd(\ubi)$ is identified with $\dsd(-\ubi\op)$ after
we identify the vertices $\binom{a}{-1+d}$ of $\dsd(\ubi)$ with
$\binom{a}{O(a)-1-d}$ of $\dsd(-\ubi\op)$, $\forall d\in[0,O(a)]$.
Specifically, the frozen vertex $-\eta_{k}=\binom{\eta_{k}}{-1}$
of $\dsd(\ubi)$ becomes identified with the frozen vertex $k^{\max}=\binom{\eta_{k}}{O(\eta_{k})-1}$
of $\dsd(-\ubi\op)$, which can be optimized by (1).

\end{proof}

\begin{Thm}\label{thm:bases-compactified-dBS}

Let $\sd$ denote $\rsd(\ubi)$ or $\dsd(\ubi)$. Then the following
statements hold.

(1) $\bUpClAlg(\sd)$ possesses the common triangular basis, denoted
$\can$.

(2) When $C$ is symmetric, $(\bUpClAlg(\sd),\can)$ admits quasi-categorification.

\end{Thm}

\begin{proof}

(1) Since all frozen vertices of $\rsd(\ubi)$ and $\dsd(\ubi)$ can
be optimized, combining Theorem \ref{thm:bases_dBS} and Proposition
\ref{prop:compactified_basis}, we obtain the desired claim.

(2) Note that we allow localization in quasi-categorification. The
claim follows from Theorem \ref{thm:bases_dBS}(2).

\end{proof}

\section{Applications: cluster algebras from Lie theory\label{sec:Applications:-cluster-algebras-Lie}}

\subsection{Double Bott-Samelson cells}

As in Section \ref{sec:Cluster-algebras-signed-words}, let $C$ denote
a $J\times J$ generalized Cartan matrix. For any words $\uzeta$,
$\ueta$, let $\ddBS_{\beta_{\ueta}}^{\beta_{\uzeta}}$ denote the\emph{
}decorated double Bott-Samelson cell in \cite[Section 2]{shen2021cluster}.
Choose any shuffle $\ubi$ of $-\uzeta$ and $\ueta$. Then the coordinate
ring $\C[\ddBS_{\beta_{\ueta}}^{\beta_{\uzeta}}]$ is isomorphic to
$\upClAlg(\dsd(\ubi))$ for $\kk=\C$ by \cite[Theorem 1.1]{shen2021cluster}.
Similarly, let $\dBS_{\beta_{\ueta}}^{\beta_{\uzeta}}$ denote the
associated (half-decorated) double Bott-Samelson cell (dBS for short).\emph{
}Then $\C[\dBS_{\beta_{\ueta}}^{\beta_{\uzeta}}]$ is isomorphic to
$\upClAlg(\rsd(\ubi))$ for $\kk=\C$, see \cite[Section 2.4]{shen2021cluster}.
Theorems \ref{thm:bases_dBS} and \ref{thm:A-U-dBS} imply the following.

\begin{Thm}\label{thm:results-dBS}

Theorems \ref{thm:intro-quasi-categorification} and \ref{thm:intro-A-U}
hold for the coordinate rings of the decorated (or half-decorated)
double Bott-Samelson cells.

\end{Thm}

We will mainly consider $\upClAlg(\rsd(\uzeta))$ arising from $\dBS_{\beta_{\uzeta}}^{e}$.
For the reader's convenience, let us briefly recall $\dBS_{\beta_{\uzeta}}^{e}$
for $C$ of finite type following \cite{casals2022cluster}. Let $G$
be a connected, simply connected, complex semisimple algebraic group.
We choose its Borel subgroups $B_{\pm}$, unipotent subgroups $U_{\pm}=[B_{\pm},B_{\pm}]$,
and maximal torus $T=B_{+}\cap B_{-}$. We fix pinning of the group
$G$, i.e., for each $a\in J$, we select an isomorphism $\varphi_{a}:\mathrm{SL}_{2}(\C)\rightarrow U_{a}$,
where $U_{a}$ is the corresponding root subgroup. Define $s_{a}:=\varphi_{a}\left(\begin{array}{cc}
0 & -1\\
1 & 0
\end{array}\right)$. For $z\in\C$, define $B_{a}(z):=\varphi_{a}\left(\begin{array}{cc}
1 & z\\
0 & 1
\end{array}\right)s_{a}=\varphi_{a}\left(\begin{array}{cc}
z & -1\\
1 & 0
\end{array}\right)$. For $\uz=(z_{k})_{k\in[1,l]}\in\C^{l}$, define $B_{\beta_{\uzeta}}(z):=B_{\zeta_{1}}(z_{1})\cdots B_{\zeta_{l}}(z_{l})$.
Then we define the double Bott-Samelson variety as $\dBS_{\beta_{\uzeta}}^{e}=\{\uz\in\C^{l}|B_{\beta_{\uzeta}}(z)\in B_{-}B_{+}\}$.

\subsection{Braid varieties}

Assume $C$ is of finite type. We briefly recall the braid varieties
following \cite[Section 3.3, Corollary 3.7]{casals2022cluster}. Let
$\delta$ denote the Demazure product map from $\Br^{+}$ to $W$,
satisfying
\begin{align*}
\delta(\sigma_{a}) & =s_{a},\text{ and }\delta(\beta\sigma_{a})=\begin{cases}
\delta(\beta)s_{a} & \text{if }l(\delta(\beta)s_{a})=l(\delta(\beta))+1\\
\delta(\beta) & \text{if }l(\delta(\beta)s_{a})=l(\delta(\beta))-1
\end{cases}.
\end{align*}
We also have
\begin{align}
\delta(\sigma_{a}\beta)=\begin{cases}
s_{a}\delta(\beta) & \text{if }l(s_{a}\delta(\beta))=l(\delta(\beta))+1\\
\delta(\beta) & \text{if }l(s_{a}\delta(\beta))=l(\delta(\beta))-1
\end{cases}.\label{eq:Demazure-left}
\end{align}

Take any word $\ueta=(\eta_{1},\ldots,\eta_{l})$ and $\beta=\beta_{\ueta}$.
Let $\ow_{0}$ be a reduced word for the longest element $w_{0}$
of $W$. We define the braid variety $X(\beta)$ as $\{\uz\in\C^{l}|\delta(\beta_{\ow_{0}})^{-1}B_{\beta_{\ueta}}(\uz)\in B\}$.
It is a smooth, irreducible affine variety of dimension $l(\beta)-l(\delta(\beta))$.
Moreover, we have $X(\beta)\simeq X(\beta\sigma_{a})$ if $\delta(\beta\sigma_{a})=\delta(\beta)s_{a}$
(\cite[Lemma 3.4(1)]{casals2022cluster}). \cite[Theorem 1.1]{casals2022cluster}
shows that $\C[X(\beta)]$ is isomorphic to the cluster algebra $\upClAlg(\sd_{\leftweave}(\ueta))$
for a specific seed $\sd_{\leftweave}(\ueta)$. 

By \cite[Lemma 3.16]{casals2022cluster}, we can naturally identify
$X(\beta_{\ow_{0}}\beta_{\ueta})$ with the dBS $\dBS_{\beta_{\ueta}}^{e}$.
By \cite[Corollary 4.42, Proposition 5.7]{casals2022cluster}, we
further have an isomorphism $\C[X(\beta_{\ow_{0}}\beta_{\ueta})]\simeq\C[\dBS_{\beta_{\ueta}}^{e}]$
identifying $\sd_{\leftweave}(\ow_{0},\ueta)$ with $\rsd(\ueta)$.

\begin{Lem}\label{lem:dBS-to-braid}

Choose $\ow_{0}=(\ow'',\ow')$, such that $\ow'$ is a reduced word
for $w_{0}\delta(\beta)^{-1}$. Then the seed $\sd_{\leftweave}(\ueta)$
could be obtained from $\sd_{\leftweave}(\ow_{0},\ueta)$ by repeating
the process: freezing and then removing a non-essential frozen vertex.

\end{Lem}

\begin{proof}

By \cite[Lemma 4.44]{casals2022cluster}, for any word $\uzeta$,
if $\delta(\sigma_{a}\beta_{\uzeta})=\delta(\beta_{\uzeta})$, $\sd_{\leftweave}(a,\uzeta)$
has an optimized frozen vertex $j$, such that $\sd_{\leftweave}(\uzeta)$
can be obtained from $\sd_{\leftweave}(a,\uzeta)$ by freezing all
vertices incident to $j$ and then removing the non-essential frozen
vertex $j$; otherwise,~$\delta(\sigma_{a}\beta_{\uzeta})=s_{a}\delta(\beta_{\uzeta})$
and $\sd_{\leftweave}(a,\uzeta)=\sd_{\leftweave}(\uzeta)$. 

Since $\delta(\beta_{\ow'}\beta)=w_{0}$, in view of (\ref{eq:Demazure-left}),
we can obtain $\sd_{\leftweave}(\ow',\ueta)$ by repeatedly applying
the above freezing-removing process to $\sd_{\leftweave}(\ow_{0},\ueta)$.
Since $\ow'$ is a reduced word for $w_{0}\delta(\beta)^{-1}$, we
have $l(\ow'\delta(\beta))=l(\delta(\beta))+l(\ow')$. So we can identify
$\sd_{\leftweave}(\ow',\ueta)$ with $\sd_{\leftweave}(\ueta)$. The
claim follows.

\end{proof}

Apply operations to $\upClAlg(\sd_{\leftweave}(\ow_{0},\ueta))=\upClAlg(\rsd(\ueta))$
following Lemma \ref{lem:dBS-to-braid}. Then Theorems \ref{thm:bases_dBS}
\ref{thm:A-U-dBS} and Corollaries \ref{cor:freezing-A-U} \ref{cor:base-change-A-U}
imply the following.

\begin{Thm}\label{thm:results_braid_variety}

Theorems \ref{thm:intro-quasi-categorification} and \ref{thm:intro-A-U}
hold for the (quantum) cluster algebra $\upClAlg(\sd_{\leftweave}(\ueta))$.

\end{Thm}

Note that Theorem \ref{thm:results_braid_variety} apply to open Richardson
varieties as well, because they are special cases of braid varieties
\cite[Section 3.6]{casals2022cluster}.

\subsection{Double Bruhat cells\label{sec:Cluster-algebras-for-alg-grp}}

Assume that $C$ is of finite type. Take any pair of reduced words
$\uzeta,\ueta$ and denote $u:=w_{\uzeta}$, $v:=w_{\ueta}$. Define
the double Bruhat cell $G^{u,v}=B_{+}uB_{+}\cap B_{-}vB_{-}$. The
quantum double Bruhat cell $\qO[G^{u,v}]$ is isomorphic to the quantum
cluster algebra $\upClAlg(\sd_{\text{BZ}}(\ubi))$, where $\ubi=(\ueta,-\uzeta)$
and $\sd_{\text{BZ}}(\ubi)$ denote the Berenstein-Zelevinsky quantum
seed in \cite{BerensteinZelevinsky05}, see \cite[Section 3.3 Theorem 9.5]{goodearl2016berenstein}
for details. Recall that $\tB(\sd_{\text{BZ}}(\ubi))=-\tB(\dsd(\ubi))$
by Lemma \ref{lem:dBS_B_matrix}. So we can identify $\sd_{\text{BZ}}(\ubi)$
with the opposite seed $\dsd(\ubi)\op$. Then Theorems \ref{thm:bases_dBS}
and \ref{thm:A-U-dBS} imply the following.

\begin{Thm}

Theorems \ref{thm:intro-quasi-categorification} and \ref{thm:intro-A-U}
hold for the quantum double Bruhat cells.

\end{Thm}

\begin{proof}

Recall that the canonical anti-isomorphism $\iota:\upClAlg(\dsd(\ubi))\simeq\upClAlg(\sd_{\text{BZ}}(\ubi))$
sends the common triangular basis to the common triangular basis.
We also know that $\upClAlg(\dsd(\ubi))$ is quasi-categorified by
some monoidal category $\cC$. It follows that the opposite algebra
$\upClAlg(\sd_{\text{BZ}}(\ubi))$ is quasi-categorified by the opposite
monoidal category $\cC\op$, see Section \ref{subsec:Monoidal-categories}.

\end{proof}

\begin{Eg}

Take $G=SL_{3}$, $\ueta=\uzeta=(1,2,1)$, $\ubi=(\ueta,-\uzeta)$.
Then $\qO[G^{w_{0},w_{0}}]$ is isomorphic to $\upClAlg(\sd_{\text{BZ }}(\ubi)))$.
See Figure \ref{fig:triangulated_trapezoid} for the quiver of $-\tB(\sd_{\text{BZ}}(\ubi))=\tB(\dsd(\ubi))$.

\end{Eg}

\subsection{Algebraic groups}

Assume that $C$ is of finite type. Then $G$ coincides with the closure
$\overline{G^{w_{0},w_{0}}}$, where $w_{0}$ denote the longest element
in $W$. Let $\ueta,\uzeta$ denote any two chosen reduced words for
$w_{0}$. Recall that the quantum double Bruhat cell $\qO[G^{w_{0},w_{0}}]$
is anti-isomorphic to $\upClAlg(\dsd(\ubi))$, where $\ubi=(\ueta,-\uzeta)$.
We claim the following natural statements are true.

\begin{Claim}\label{claim:G_case}

(1) $\C[G]$ is anti-isomorphic to $\bUpClAlg(\dsd(\ubi))\otimes\C$
for $\kk=\Z$. 

(2) $\qO[G]\otimes\Q(q^{\Hf})$ is anti-isomorphic to $\bUpClAlg(\dsd(\ubi))\otimes\Q(q^{\Hf})$
for $\kk=\Z[q^{\pm\Hf}]$. 

\end{Claim}

Proofs will appear elsewhere. Here, the quantized coordinate ring
$\qO[G]$ is defined as the quantum function algebra spanned by the
matrix coefficients, see \cite{Kashiwara93}. Assuming these claims,
Theorem \ref{thm:bases-compactified-dBS} implies the following result.

\begin{Thm}

Theorem \ref{thm:intro-quasi-categorification} holds for the (quantized)
coordinate rings of $G$.

\end{Thm}

\begin{Eg}[Example $SL_2$]\label{eg:global-crystal-sl2}

Take the Cartan matrix $C=\left(2\right)$ and the seed $\sd=\dsd(\ubi)$,
where $\ubi=(1,-1)$. We have $\ddI=\{-1,1,2\}$ and $I_{\ufv}=\{1\}$.
Consider the quantum cluster algebra $\bUpClAlg(\sd\op)$ associated
with the opposite seed $\sd\op$. The $B$-matrix of $\sd\op$ is
$\left(\begin{array}{c}
-1\\
0\\
-1
\end{array}\right)$. Following \cite[Section 10.1]{BerensteinZelevinsky05}, we choose
the $\Lambda$-matrix of $\sd\op$ to be $\left(\begin{array}{ccc}
0 & -1 & 0\\
1 & 0 & 1\\
0 & -1 & 0
\end{array}\right)$. Then there is an isomorphism between $\bUpClAlg(\sd\op)$ and $\qO[SL_{2}]$,
such that the cluster variables $x_{-1}(\sd),x_{1}(\sd),x_{2}(\sd)$,
$x_{1}(\mu_{1}\sd)$ correspond to the quantum minors $u:=\Delta_{\varpi_{1},s_{1}\varpi_{1}}$,
$x:=\Delta_{\varpi_{1},\varpi_{1}}$, $v:=\Delta_{s_{1}\varpi_{1},\varpi_{1}}$,
$y:=\Delta_{s_{1}\varpi_{1},s_{1}\varpi_{1}}$, respectively. Here,
$\varpi_{1}$ denotes the fundamental weight. We refer the reader
to \cite{BerensteinFominZelevinsky05} \cite{BerensteinZelevinsky05}
for definitions and properties of (quantum) minors.

Note that the common triangular basis of $\bUpClAlg(\sd\op)$ consists
of the cluster monomials. On the other hand, by \cite[Sectioni 9]{Kashiwara93},
the global crystal basis for $\qO[SL_{2}]$ is $\{u^{n}x^{m}v^{l}|n,m,l\in\N\}\cup\{u^{n}y^{m}v^{l}|n,m,l\in\N,m>0\}$.
We observe that these global basis elements correspond to the cluster
monomials up to $q^{\frac{\Z}{2}}$multiples.

\end{Eg}

\subsection{Other coordinate rings}

Recall that open positroid varieties are special cases of double Bott-Samelson
cells. So Theorems \ref{thm:results-dBS} and \ref{thm:intro-A-U}
apply to them.

Recall that the coordinate ring of a Grassmannian is a cluster algebra
with a distinguished seed $\sd_{\Gr}$ \cite{scott2006grassmannians}.
Note that $\sd_{\Gr}$ is similar to $\rsd(\ugamma)$ for some \emph{adapted}
word $\ugamma$ in type $A$. It is already known that $\rsd(\ugamma)$
has the common triangular basis $\can$ and admits categorification
by \cite{qin2017triangular}, see Theorem \ref{thm:basis-for-HL}.
Combining with Theorem \ref{thm:A-U-dBS} and Corollary \ref{cor:base-change-A-U},
we deduce that Theorems \ref{thm:intro-quasi-categorification} and
\ref{thm:intro-A-U} hold for $\upClAlg(\sd_{\Gr})$.

\subsection{Cluster structures from quantum affine algebras\label{subsec:Cluster-structures-quantum-affine}}

\global\long\def\mybinom#1#2{\genfrac{\langle}{\rangle}{0pt}{}{#1}{#2}}%

Choose a Coxeter word $\uc=(c_{1},\ldots,c_{|J|})$, i.e., each $a\in J$
appears in it exactly once. Note that $\uc$ determines a skew-symmetric
$J\times J$-matrix $B^{\Delta}$ such that $B_{c_{j},c_{k}}^{\Delta}:=\sign(C_{c_{j},c_{k}})\in\{0,-1\}$
for $j<k$. A letter $c_{k}$ is called a sink of $\uc$ if $C_{c_{k},c_{j}}\geq0$
$\forall j<k$. 

Introduce $\W:=\{\mybinom ad|a\in J,d\in\Z\}$ as in \cite{KimuraQin14}.
Endow it with the order $<$ such that $\mybinom{c_{j}}d<\mybinom{c_{k}}r$
if $d<r$ or $d=r$ but $j<k$. Let $e_{\mybinom ad}$ denote the
$\mybinom ad$-th unit vector of $\N^{\W}$. 

For any sink letter $a$, define $\mu_{a}\uc:=(\uc\backslash\{a\},a)$.
Choose any $\uc$-adapted word $\ugamma=(\gamma_{i})$, i.e., $\gamma_{i}$
is a sink of $\mu_{\gamma_{i-1}}\cdots\mu_{\gamma_{1}}\uc$, $\forall i$.
Define $\beta_{k}:=\beta_{k}(\ugamma):=e_{\mybinom{\gamma_{k}}{o_{-}(\gamma_{k})}}$.

Assume that the quiver associated with $B^{\Delta}$ is a tree. Then
we can choose a $\Z$-valued function $\xi$ on $J$, such that $\xi_{c_{k}}=\xi_{c_{j}}-1$
when $C_{c_{j},c_{k}}<0$ and $k>j$. Denote $J_{\Z}(\xi):=\{(a,p)|a\in J,p\in\xi_{a}+2\Z\}$.
In this case, we identify $\W$ with $J_{\Z}(\xi)$ such that $\mybinom ad$
corresponds to $(a,\xi_{a}-2d)$. Then we can denote $e_{(a,\xi_{a}-2d)}$
by $e_{\mybinom ad}$.

\begin{Eg}\label{eg:A2-121212}

Assume $C=\left(\begin{array}{cc}
2 & -1\\
-1 & 2
\end{array}\right)$. We choose $\uc=(1,2)$, $\ugamma=(1,2,1,2,1,2)$. The quiver for
$\rsd(\ugamma)$ is drawn in Figure \ref{fig:quiver-aff}, where we
denote elements of $I(\rsd(\ugamma))$ by $k\in[1,6]$, $\mybinom ad\in\W$,
or $(a,\xi_{a}-2d)\in J_{\Z}(\xi)$.

In this case, the letter $1$ is a sink, $\xi_{2}=\xi_{1}-1$ and
$\mu_{1}\uc=(2,1)$.

\begin{figure}[h]
\caption{The quiver for $\rsd(1,2,1,2,1,2)$}
\label{fig:quiver-aff}\subfloat[]{

\begin{tikzpicture}  [node distance=48pt,on grid,>={Stealth[round]},bend angle=45,      pre/.style={<-,shorten <=1pt,>={Stealth[round]},semithick},    post/.style={->,shorten >=1pt,>={Stealth[round]},semithick},  unfrozen/.style= {circle,inner sep=1pt,minimum size=12pt,draw=black!100,fill=red!100},  frozen/.style={rectangle,inner sep=1pt,minimum size=12pt,draw=black!75,fill=cyan!100},   point/.style= {circle,inner sep=1pt,minimum size=5pt,draw=black!100,fill=black!100},   boundary/.style={-,draw=cyan},   internal/.style={-,draw=red},    every label/.style= {black}] \node[unfrozen] (v1) at (0,1) {1}; \node[unfrozen] (v2) at (-0.5,0) {2}; \node[unfrozen] (v3) at (-1,1) {3}; \node[unfrozen] (v4) at (-1.5,0) {4}; \node[frozen] (v5) at (-2,1) {5}; \node[frozen] (v6) at (-2.5,0) {6}; \draw[<-]  (v1) edge (v2); \draw[<-]  (v2) edge (v3); \draw[<-]  (v3) edge (v4); \draw[<-]  (v4) edge (v5); \draw[<-,dashed]  (v5) edge (v6); \draw[<-]  (v5) edge (v3); \draw[<-]  (v3) edge (v1); \draw[<-]  (v6) edge (v4); \draw[<-]  (v4) edge (v2); \end{tikzpicture}

}\subfloat[]{

\begin{tikzpicture}  [scale=1,node distance=48pt,on grid,>={Stealth[round]},bend angle=45,      pre/.style={<-,shorten <=1pt,>={Stealth[round]},semithick},    post/.style={->,shorten >=1pt,>={Stealth[round]},semithick},  unfrozen/.style= {circle,inner sep=1pt,minimum size=12pt,draw=black!100,fill=red!100},  frozen/.style={rectangle,inner sep=1pt,minimum size=12pt,draw=black!75,fill=cyan!100},   point/.style= {circle,inner sep=1pt,minimum size=5pt,draw=black!100,fill=black!100},   boundary/.style={-,draw=cyan},   internal/.style={-,draw=red},    every label/.style= {black}] \node (v1) at (0,1) {$\mybinom{1}{0}$}; \node (v2) at (-0.5,0) {$\mybinom{2}{0}$}; \node (v3) at (-1,1) {$\mybinom{1}{1}$}; \node (v4) at (-1.5,0) {$\mybinom{2}{1}$}; \node (v5) at (-2,1) {$\mybinom{1}{2}$}; \node (v6) at (-2.5,0) {$\mybinom{2}{2}$}; \draw[<-]  (v1) edge (v2); \draw[<-]  (v2) edge (v3); \draw[<-]  (v3) edge (v4); \draw[<-]  (v4) edge (v5); \draw[<-,dashed]  (v5) edge (v6); \draw[<-]  (v5) edge (v3); \draw[<-]  (v3) edge (v1); \draw[<-]  (v6) edge (v4); \draw[<-]  (v4) edge (v2); \end{tikzpicture}}\subfloat[]{ \begin{tikzpicture}  [scale=2,node distance=48pt,on grid,>={Stealth[round]},bend angle=45,      pre/.style={<-,shorten <=1pt,>={Stealth[round]},semithick},    post/.style={->,shorten >=1pt,>={Stealth[round]},semithick},  unfrozen/.style= {circle,inner sep=1pt,minimum size=12pt,draw=black!100,fill=red!100},  frozen/.style={rectangle,inner sep=1pt,minimum size=12pt,draw=black!75,fill=cyan!100},   point/.style= {circle,inner sep=1pt,minimum size=5pt,draw=black!100,fill=black!100},   boundary/.style={-,draw=cyan},   internal/.style={-,draw=red},    every label/.style= {black}] \node (v1) at (0,0.5) {$(1,\xi_1)$}; \node (v2) at (-0.5,0) {$(2,\xi_2)$}; \node (v3) at (-1,0.5) {$(1,\xi_1-2)$}; \node (v4) at (-1.5,0) {$(2,\xi_2-2)$}; \node (v5) at (-2,0.5) {$(1,\xi_1-4)$}; \node (v6) at (-2.5,0) {$(2,\xi_2-4)$}; \draw[<-]  (v1) edge (v2); \draw[<-]  (v2) edge (v3); \draw[<-]  (v3) edge (v4); \draw[<-]  (v4) edge (v5); \draw[<-,dashed]  (v5) edge (v6); \draw[<-]  (v5) edge (v3); \draw[<-]  (v3) edge (v1); \draw[<-]  (v6) edge (v4); \draw[<-]  (v4) edge (v2); \end{tikzpicture}}

\end{figure}

\end{Eg}

From now on, assume $C$ to be of type $ADE$. Let $\qAff$ denote
the associated quantum affine algebra, where $\varepsilon\in\C^{\times}$
is not a root of unity. It has distinguished finite dimensional simple
modules $L_{a,\varepsilon'}$, $\forall a\in J,\varepsilon'\in\C^{\times}$,
called fundamental modules. Denote $L(e_{\mybinom a{-d}}):=L_{a,\varepsilon^{\xi_{a}+2d}}$.
The composition factors of the tensor products of $L(e_{\mybinom a{-d}})$
could be parameterized by $L(w)$, for $w\in\oplus\N e_{\mybinom a{-d}}$.
We refer the reader to \cite{Nakajima01}\cite{Nakajima04}\cite{HernandezLeclerc09}
for details.

Following \cite{HernandezLeclerc09}, let $\cC_{\Z}(\xi)$ denote
the monoidal category whose objects have these composition factors
$L(w)$. Let $\cC_{\ugamma}(\xi)$ denote the monoidal subcategory
of $\cC_{\Z}(\xi)$ whose objects have the composition factors $L(w)$,
for $w\in\oplus_{k=1}^{l(\ugamma)}\N\beta_{k}$. When $j\leq k$ and
$\gamma_{j}=\gamma_{k}$, $L(\beta_{[j,k]})$ is called a Kirillov-Reshetikhin
module, where $\beta_{[j,k]}:=\beta_{j}+\beta_{j[1]}+\cdots+\beta_{k}$.
Its deformed Grothendieck ring, denoted by $K$, is constructed from
quiver varieties \cite{Nakajima04}\cite{VaragnoloVasserot03}\cite[Sections 7, 8.4]{qin2017triangular}. 

By \cite[Theorem 8.4.3]{qin2017triangular}, we have an isomorphism
$\kappa:\bClAlg(\rsd(\ugamma))\simeq K$, such that $\kappa(x_{i}(\rsd(\ugamma)))=[L(\beta_{[i^{\min},i]})]$,
where the quantization of $K$ and $\rsd(\ugamma)$ are chosen as
in \cite[Section 7.3]{qin2017triangular}.

\begin{Thm}[{\cite[Theorem 1.2.1(II)]{qin2017triangular}}]\label{thm:basis-for-HL}

$\bClAlg(\rsd(\ugamma))$ has the common triangular basis $\can$.
And $(\bClAlg(\rsd(\ugamma)),\can)$ is categorified by $\cC_{\ugamma}(\xi)$,
such that $\kappa\can_{m}=[L(w)]$, where $m=(m_{i})_{i\in[1,l]}$
satisfies $\sum_{i}m_{i}\beta_{[i^{\min},i]}=\sum_{i}w_{i}\beta_{i}$. 

\end{Thm}

We have a new proof for Theorem \ref{thm:basis-for-HL}: The first
statement is implied by Theorem \ref{thm:bases-compactified-dBS}.
Combining with \cite[Corollary 9.1.9]{qin2017triangular}, it implies
the second statement.

Denote $\uc^{l}:=(\uc,\uc,\ldots,\uc)$ for $l$-copies of $\uc$.
Particularly, Theorem \ref{thm:basis-for-HL} implies the conjecture
of Hernandez of Leclerc that $\cC_{\uc^{l}}(\xi)$ categorifies $\bClAlg(\rsd(\uc^{l}))$
(Definition \ref{def:categorification}).

\section{Applications: cluster algebras from double Bott-Samelson cells\label{sec:Applications-dBS}}

Take any word $\ueta$. Denote $l=l(\ueta)$, $\rsd=\rsd(\ueta)$.
We will explore rich structures of the quantum cluster algebra $\bUpClAlg(\rsd)$
arising from the double Bott-Samelson cell $\dBS_{\beta_{\ueta}}^{e}$.

\subsection{Mutation sequence $\Sigma$}

$\forall j\leq k\in[1,l]\simeq I(\ueta)=I$, $\eta_{j}=\eta_{k}$,
define $\seq_{[j,k]}:=\mu_{k}\cdots\mu_{j[1]}\mu_{j}$, $\Sigma_{k}:=\Sigma_{k}(\ueta):=\seq_{[k^{\min},k^{\min}[o_{+}(k)-1]]}=\seq_{[k^{\min},k^{\max}[-1-o_{-}(k)]]}$,
$\Sigma:=\Sigma(\ueta):=\Sigma_{l}\cdots\Sigma_{2}\Sigma_{1}$, where
the mutation sequences $\Sigma_{k^{\max}}$ are the identity.

Recall that $k\in[1,l]$ is identified with $\binom{\eta_{k}}{o_{-}(k)}^{\ueta}$,
abbreviated as $\binom{\eta_{k}}{o_{-}(k)}$. Denote $\binom{a}{-1-r}=-\infty$,
$\binom{a}{O(a)+r}=+\infty$ for $r\in\N$. The order $<$ on $[1,l]$
induces the order $<_{\ueta}$ on $I$, abbreviated as $<$.

Denote $\rsd:=\rsd(\ueta)$. When we apply $\Sigma$ to $\rsd$, the
seeds appearing can be described as $\rsd\{r\}_{s}$: For any $r\in[0,l]$,
denote $\rsd\{0\}:=\rsd$, $\rsd\{r\}:=\Sigma_{r}\cdots\Sigma_{2}\Sigma_{1}\rsd$.
Further denote $\rsd\{r\}_{0}:=\rsd\{r\}$, $r':=r+1$, and $\rsd\{r\}_{s}:=\seq_{[r',r'[s-1]]}\rsd\{r\}$
for $s\in[0,o_{+}(r')]$. Note that $\rsd\{r\}_{o_{+}(r')}=\rsd\{r'\}_{0}$.
We denote the entries $b_{jk}$ by $B_{jk}$ below to avoid confusion.

\begin{Lem}\label{lem:mutation_shuffle_seeds} (1) $\rsd\{r\}=\rsd(\ueta_{[r+1,l]},-\ueta_{[1,r]}\op)$.

(2) $\rsd\{r\}_{s}=\rsd(\ueta_{[r'+1,r'[s]]},-\eta_{r'},\ueta_{[r'[s]+1,l]},-\ueta_{[1,r]}\op)$. 

(3) Take any $\rsd\{r\}_{s}\neq\Sigma\rsd$. Denote $a:=\eta_{r'}$.
Then we have $B_{\binom{a}{s-1},\binom{a}{s}}(\rsd\{r\}_{s})=B_{\binom{a}{s+1},\binom{a}{s}}(\rsd\{r\}_{s})=-1$,
$B_{\binom{b}{d},\binom{a}{s}}(\rsd\{r\}_{s})=-C_{ba}$ if $b\neq a$
and $\binom{b}{d+O([1,r];b)}<_{\ueta}\binom{a}{s+O([1,r];a)+1}<_{\ueta}\binom{b}{d+O([1,r];b)+1}$,
and $B_{j,\binom{a}{s}}(\rsd\{r\}_{s})=0$ for other $j\in I$.

\end{Lem}

\begin{proof}

(1)(2) Consider the seeds appearing when we apply $\Sigma_{1}$ to
$\rsd=\rsd\{0\}$. In this case, $r'=1$. We rewrite $\rsd(\eta_{1},\eta_{[2,l]})=\rsd(-\eta_{1},\eta_{[2,l]})$
by the left reflection. Applying $\Sigma_{1}$ to $\rsd(-\eta_{1},\eta_{[2,l]})$
via flips (Section \ref{subsec:change-string-diagrams}), we get $\rsd_{s}=\rsd(\ueta')$
for $\ueta'=(\ueta_{[2,1[s]]},-\eta_{1},\ueta_{[1[s]+1,l]})$. Next,
use the reflection $\rsd(\eta_{2},\eta_{[3,l]},-\eta_{1})=\rsd(-\eta_{2},\eta_{[3,l]},-\eta_{1})$
and apply $\Sigma_{2}$ to $\rsd(-\eta_{2},\eta_{[3,l]},-\eta_{1})$
via flips. Repeating this process, we obtained the desired claims.

(3) We prove the claim by induction on $l$. The case $l=1$ is trivial.
Assume they have been proved for lengths $\leq l-1$.

Take $r=0$. Consider $\rsd_{s}=\rsd(\ueta')$ where $\ueta'=(\ueta_{[2,1[s]]},-\eta_{1},\ueta_{[1[s]+1,l]})$,
$1[s]=\binom{\eta_{1}}{s}$. So the $\binom{\eta_{1}}{s}$-th triangle
$T'_{\binom{\eta_{1}}{s}}$ for $\ueta'$ satisfies $\bi_{T'_{\binom{\eta_{1}}{s}}}=-\eta_{1}$,
and $\binom{\eta_{1}}{s}$ is the interval $\ell_{\eta_{1}}^{+}$
incident to its interior point. (\ref{eq:dBS_B_matrix}) implies $B_{\binom{\eta_{1}}{s-1},\binom{\eta_{1}}{s}}(\rsd_{s})=B_{\binom{\eta_{1}}{s+1},\binom{\eta_{1}}{s}}(\rsd_{s})=-1$,
 $B_{\binom{b}{d},\binom{\eta_{1}}{s}}(\rsd_{s})=-C_{b\eta_{1}}$
if $b\neq\eta_{1}$ and $\binom{b}{d}<\binom{\eta_{1}}{s}<\binom{\eta_{1}}{s+1}<\binom{b}{d+1}$
or $\binom{\eta_{1}}{s}<\binom{b}{d}<\binom{\eta_{1}}{s+1}<\binom{b}{d+1}$,
and $B_{j,\binom{\eta_{1}}{s}}=0$ for other $j\in I$.

Therefore the claim is true for $r=0$. Particularly, it holds for
$\rsd':=\rsd\{1\}=\rsd\{0\}_{o_{+}(1)}=\rsd(\ueta_{[2,l]},-\eta_{1})$.
Then $B_{1^{\max},k}(\rsd')=0$ for $k\in I_{\ufv}\backslash\{1^{\max}[-1]\}$.
Denote $\usd:=\rsd(\ueta_{[2,l]})$. By Lemma \ref{lem:calibration-word},
we have the cluster embedding $\iota$ from $\usd$ to $\rsd$ such
that $\iota\binom{b}{d}^{\usd}=\binom{b}{d+\delta_{b,\eta_{1}}}$
and the cluster embedding $\iota'$ from $\usd$ to $\rsd'$ such
that $\iota'\binom{b}{d}^{\usd}=\binom{b}{d}$. Recall that $\binom{b}{d}^{\usd}<_{\ueta_{[2,l]}}\binom{c}{h}^{\usd}$
if and only if $\iota\binom{b}{d}^{\usd}<\iota\binom{c}{h}^{\usd}$,
i.e., $\binom{b}{d+\delta_{b,1}}<\binom{c}{h+\delta_{c,1}}$.

Note that the mutation sequence $\Sigma':=\Sigma_{l}\cdots\Sigma_{2}$
does not act on $1^{\max}[-1]$. So $B_{1^{\max},k}(\rsd\{r+1\}_{s})=0$,
$\forall r\in\N$, $k\in I_{\ufv}\backslash\{1^{\max}[-1]\}$. Moreover,
$\Sigma'$ is identified with the mutation sequence $\Sigma(\ueta_{[2,l]})$
via $\iota'$. Let $\seq$ denote any proper sub sequence of $\Sigma(\ueta_{[2,l]})$,
such that $\Sigma(\ueta_{[2,l]})=\cdots\mu_{k_{h+1}}\mu_{k_{h}}\cdots\mu_{k_{1}}=\cdots\mu_{k_{h+1}}\seq$.
Then $\seq\usd$ takes the form $\usd\{r\}_{s}$, and $\iota'$ is
a cluster embedding from $\seq\usd$ to $(\iota'\seq)\rsd'=\rsd\{r+1\}_{s}$. 

By the induction hypothesis on $l-1$, $B_{\binom{a}{s-1}^{\usd},\binom{a}{s}^{\usd}}(\usd\{r\}_{s})=B_{\binom{a}{s+1}^{\usd},\binom{a}{s}^{\usd}}(\usd\{r\}_{s})=-1$,
$B_{\binom{b}{d}^{\usd},\binom{a}{s}^{\usd}}(\usd\{r\}_{s})=-C_{ba}$
if $b\neq a$ and $\binom{b}{d+O^{\ueta_{[2,l]}}([1,r];b)}^{\usd}<_{\ueta_{[2,l]}}\binom{a}{s+O^{\ueta_{[2,l]}}([1,r];a)+1}^{\usd}$
$<_{\ueta_{[2,l]}}\binom{b}{d+O^{\ueta_{[2,l]}}([1,r];b)+1}^{\usd}$,
and $B_{j,\binom{a}{s}^{\usd}}(\usd\{r\}_{s})=0$ for other $j\in I(\usd)$. 

Since $\iota':\binom{b}{d}^{\usd}\rightarrow\binom{b}{d}$ is a cluster
embedding from $\usd\{r\}_{s}$ to $\rsd\{r+1\}_{s}$, we obtain the
desired entries $B_{\iota'j,\binom{a}{s}}(\rsd\{r+1\}_{s})=B_{j,\binom{a}{s}^{\usd}}(\usd\{r\}_{s})$
for $j\in I(\usd)$. In addition, we already know $B_{1^{\max},\binom{a}{s}}(\rsd\{r+1\}_{s})=0$
for $\binom{a}{s}\neq1^{\max}[-1]$. It remains to translate the inequality
condition for for $\rsd\{r+1\}_{s}$. Applying $\iota:I(\usd)\rightarrow I(\rsd)$,
$\iota\binom{b}{d}=\binom{b}{d+\delta_{\eta_{1},b}}$, and using $O^{\ueta_{[2,l]}}([1,r];b)+\delta_{\eta_{1},b}=O([2,r+1];b)+\delta_{\eta_{1},b}=O([1,r+1];b)$,
we obtain $\binom{b}{d+O([1,r+1];b)}<_{}\binom{a}{s+O([1,r+1];a)+1}<_{}\binom{b}{d+O([1,r+1];b)+1}$
as desired.

\end{proof}

\begin{Eg}\label{eg:Kronecker-braid-quiver}

Take $C=\left(\begin{array}{cc}
2 & -2\\
-2 & 2
\end{array}\right)$, $\ueta=(1,2,1,1,2,2,1)$, $\rsd=\rsd(\ueta)$. We have $\Sigma=\Sigma_{5}\Sigma_{4}\Sigma_{3}\Sigma_{2}\Sigma_{1}=\mu_{2}\mu_{1}\mu_{[1,3]}\mu_{[2,5]}\mu_{[1,4]}=\mu_{2}\mu_{1}(\mu_{3}\mu_{1})(\mu_{5}\mu_{2})(\mu_{4}\mu_{3}\mu_{1})$
(read from right to left). We can check Lemma \ref{lem:mutation_shuffle_seeds}
directly. See Figure \ref{fig:dBS-sequence} for some $\ddB(\rsd\{r\})$.

\end{Eg}

\begin{figure}[h]
\caption{Quivers with weights associated with $\ddB$.}
\label{fig:dBS-sequence}\subfloat[$\ddB(\rsd)$ for $\rsd=\rsd(1,2,1,1,2,2,1)$]{\label{fig:dBS}

\begin{tikzpicture}  [node distance=48pt,on grid,>={Stealth[round]},bend angle=45,      pre/.style={<-,shorten <=1pt,>={Stealth[round]},semithick},    post/.style={->,shorten >=1pt,>={Stealth[round]},semithick},  unfrozen/.style= {circle,inner sep=1pt,minimum size=12pt,draw=black!100,fill=red!100},  frozen/.style={rectangle,inner sep=1pt,minimum size=12pt,draw=black!75,fill=cyan!100},   point/.style= {circle,inner sep=1pt,minimum size=5pt,draw=black!100,fill=black!100},   boundary/.style={-,draw=cyan},   internal/.style={-,draw=red},    every label/.style= {black}]      \node[unfrozen](v1) at (0,0){1};       \node[unfrozen](v2) at (20pt,40pt){2};       \node[unfrozen](v3) at (40pt,0){3};       \node[unfrozen](v4) at (60pt,0){4};       \node[unfrozen](v5) at (80pt,40pt){5};       \node[frozen](v6) at (100pt,40pt){6};       \node[frozen](v7) at (120pt,0){7};   \draw[->] (v1) -- (v3);  \draw[->] (v3) -- (v4);  \draw[->] (v4) -- (v7);  \draw[->] (v2) -- (v5);   \draw[->] (v5) -- (v6);  \draw[->] (v2) --node[left]{2}  (v1);  \draw[->] (v4) --node[left]{2}  (v2);  \draw[->] (v6) --node[left]{2}  (v4);  \draw[->] (v7) -- (v6);  \end{tikzpicture}}\subfloat[$\ddB(\rsd\{1\})$]{\label{fig:dBS-shift}

\begin{tikzpicture}  [node distance=48pt,on grid,>={Stealth[round]},bend angle=45,      pre/.style={<-,shorten <=1pt,>={Stealth[round]},semithick},    post/.style={->,shorten >=1pt,>={Stealth[round]},semithick},  unfrozen/.style= {circle,inner sep=1pt,minimum size=12pt,draw=black!100,fill=red!100},  frozen/.style={rectangle,inner sep=1pt,minimum size=12pt,draw=black!75,fill=cyan!100},   point/.style= {circle,inner sep=1pt,minimum size=5pt,draw=black!100,fill=black!100},   boundary/.style={-,draw=cyan},   internal/.style={-,draw=red},    every label/.style= {black}]      \node[unfrozen] (v1) at (40pt,0) {1};       \node[unfrozen](v2) at (20pt,40pt){2};       \node[unfrozen] (v3) at (60pt,0) {3};       \node[unfrozen] (v4) at (120pt,0) {4};       \node[unfrozen](v5) at (80pt,40pt){5};       \node[frozen](v6) at (100pt,40pt){6};       \node[frozen] (v7) at (140pt,0) {7};   \draw[->] (v1) -- (v3);  \draw[->] (v3) -- (v4);  \draw[<-] (v4) -- (v7);  \draw[->] (v2) -- (v5);   \draw[->] (v5) -- (v6);  \draw[->] (v3) --node[left]{2}  (v2);  \draw[->] (v6) --node[left]{2}  (v3);  \draw[->] (v4) --node[left]{2}  (v6);  \draw[->] (v6) --  (v7); 
\end{tikzpicture}}$\quad$\subfloat[$\ddB(\Sigma\rsd)$]{\label{fig:dBS=00005B1=00005D}\begin{tikzpicture}  [node distance=48pt,on grid,>={Stealth[round]},bend angle=45,      pre/.style={<-,shorten <=1pt,>={Stealth[round]},semithick},    post/.style={->,shorten >=1pt,>={Stealth[round]},semithick},  unfrozen/.style= {circle,inner sep=1pt,minimum size=12pt,draw=black!100,fill=red!100},  frozen/.style={rectangle,inner sep=1pt,minimum size=12pt,draw=black!75,fill=cyan!100},   point/.style= {circle,inner sep=1pt,minimum size=5pt,draw=black!100,fill=black!100},   boundary/.style={-,draw=cyan},   internal/.style={-,draw=red},    every label/.style= {black}]      \node[unfrozen] (v1) at (0pt,0) {1};       \node[unfrozen](v2) at (10pt,40pt){2};       \node[unfrozen] (v3) at (40pt,0) {3};       \node[unfrozen] (v4) at (60pt,0) {4};       \node[unfrozen](v5) at (30pt,40pt){5};       \node[frozen](v6) at (70pt,40pt){6};       \node[frozen] (v7) at (80pt,0) {7};   \draw[<-] (v1) -- (v3);  \draw[<-] (v3) -- (v4);  \draw[<-] (v4) -- (v7);  \draw[<-] (v2) -- (v5);   \draw[<-] (v5) -- (v6);  \draw[->] (v1) --node[left]{2}  (v5);  \draw[->] (v5) --node[left]{2}  (v4);  \draw[->] (v4) --node[left]{2}  (v6);  \draw[->] (v6) --node[right]{}  (v7); 
\end{tikzpicture}}
\end{figure}

\subsection{Interval variables and fundamental variables}

\begin{Def}\label{def:interval-variables}

Consider the seeds $\rsd\{r\}$, $r\in[0,l]$. $\forall i\in I,a\in J$,
denote $r_{a}=O([1,r];a)$ and $r_{i}:=r_{\eta_{i}}$. Then, $\forall i=i^{\min}$,
$d<O(\eta_{i})-r_{i}$, $W_{[i[r_{i}],i[r_{i}+d]]}(\rsd):=x_{i[d]}(\rsd\{r\})$
are called the interval variables of $\bClAlg(\rsd(\ueta))$. $W_{j}(\rsd\{r\}):=W_{[j,j]}(\rsd\{r\})$
are called the fundamental variables.\footnote{When $C$ is of type $ADE$ and $\ueta$ is adapted, $W_{[j,k]}$
and $W_{[j,j]}$ correspond to the Kirillov-Reshetikhin modules and
the fundamental modules of the quantum affine algebras respectively,
whence the names. The fundamental variables might also be called root
variables in analogous to the root vectors for PBW bases.}

\end{Def}

Equivalently, $W_{[\binom{a}{r_{a}},\binom{a}{r_{a}+d}]}(\rsd)=x_{\binom{a}{d}}(\rsd\{r\})$
for $d<O(a)-r_{a}$. We often omit $\rsd$ for simplicity. The degrees
$\beta_{[j,k]}:=\deg^{\rsd}W_{[j,k]}$, $\beta_{j}:=\beta_{[j,j]}$,
have the following surprisingly nice description.

\begin{Lem}\label{lem:interval_degree}

We have $\beta_{[j,k]}=f_{k}-f_{j[-1]}$, where we denote $f_{\pm\infty}=0$.

\end{Lem}

An example is given as Example \ref{eg:Kronecker-braid-deg}. Before
proving Lemma \ref{lem:interval_degree}, let us make some preparation.

Given any $1\leq j\leq k\leq l$. By Lemma \ref{lem:calibration-word},
$\usd:=\rsd(\ueta_{[j,k]})$ is a good sub seed of $\rsd':=\rsd(\ueta_{[j,l]},-\ueta_{[1,j-1]}\op)$
via the cluster embedding $\iota'$ such that $\iota'\binom{a}{d}^{\usd}=\binom{a}{d}$,
where $\binom{a}{d}^{\usd}\in I(\usd)$ and $\binom{a}{d}\in I(\rsd')=I(\rsd)$.
Then we obtain the following inclusion of algebras:
\begin{align}
\iota_{[j,k]} & :\bClAlg(\rsd(\ueta_{[j,k]}))\overset{\iota'}{\hookrightarrow}\bClAlg(\rsd(\ueta_{[j,l]},-\ueta_{[1,j-1]}\op)\overset{(\Sigma_{j-1}\cdots\Sigma_{1})^{*}}{\simeq}\bClAlg(\rsd(\ueta))\label{eq:inclusion-subword-cluster}
\end{align}
 such that $\iota_{[j,k]}x_{h}(\usd)=x_{\iota'(h)}(\rsd')$, $\forall h\in I(\usd)$,
where we omit the symbol $(\Sigma_{j-1}\cdots\Sigma_{1})^{*}$ as
usual. For any mutation sequence $\seq$ on $I_{\ufv}(\usd)$, let
$\iota'\seq$ denote the corresponding sequence on $I_{\ufv}(\rsd')$.
Then we obtain $\iota_{[j,k]}x_{h}(\seq\usd)=x_{\iota'(h)}((\iota'\seq)\rsd')$.
Particularly, by taking $\iota'\seq$ such that $\Sigma_{l}\cdots\Sigma_{j}=\cdots\mu_{k_{d+1}}\mu_{k_{d}}\cdots\mu_{k_{1}}=\cdots\mu_{k_{d+1}}(\iota'\seq)$,
we deduce the following result by definitions of the interval variables.

\begin{Lem}\label{lem:embed-interval-variables}

The inclusion $\iota_{[j,k]}:\bClAlg(\rsd(\ueta_{[j,k]}))\hookrightarrow\bClAlg(\rsd(\ueta_{[j,l]},-\ueta_{[1,j-1]}\op))=\bClAlg(\rsd(\ueta))$
sends the fundamental variables $W_{[\binom{a}{d}^{\ueta_{[j,k]}},\binom{a}{d'}^{\ueta_{[j,k]}}]}(\rsd(\ueta_{[j,k]}))$
to the fundamental variables $W_{[\binom{a}{d+O([1,j-1];a])},\binom{a}{d'+O([1,j-1];a])}]}(\rsd(\ueta))$.

\end{Lem}

\begin{proof}[Proof of Lemma \ref{lem:interval_degree}]

We prove the claim by induction on $l$. The case $l=1$ is trivial.
Assume it has been proved for lengths $\leq l-1$. 

Denote $\usd=\rsd(\ueta_{[2,l]})$. It is a good sub seed of $\rsd':=\rsd\{1\}=\rsd(\ueta_{[2,l]},-\eta_{1})$.
Then we have the cluster embedding from $\usd$ to $\rsd'$ sending
$\binom{a}{r}^{\usd}\in I(\usd)$ to $\binom{a}{r}\in I(\rsd')=I(\rsd)$.
By Lemma \ref{lem:embed-interval-variables}, the inclusion $\iota_{[2,l]}:\bClAlg(\usd)\overset{\iota'}{\hookrightarrow}\bClAlg(\rsd')\simeq\bClAlg(\rsd)$
sends $W_{[\binom{a}{r}^{\usd},\binom{a}{r+d}^{\usd}]}(\usd)$ to
$W_{[\binom{a}{r+\delta_{a,\eta_{1}}},\binom{a}{r+\delta_{a,\eta_{1}}+d}]}$.
The induction hypothesis implies $\deg^{\usd}(W{}_{[\binom{a}{r}^{\usd},\binom{a}{r+d}^{\usd}]}(\usd))=f_{\binom{a}{r+d}^{\usd}}(\usd)-f_{\binom{a}{r-1}^{\usd}}(\usd).$
Using the cluster embedding $\iota'$, we deduce
\begin{align*}
\deg^{\rsd'}(W_{[j,k]}) & =\begin{cases}
f_{k}(\rsd')-f_{j[-1]}(\rsd') & a\neq\eta_{1},j=\binom{a}{r},k=\binom{a}{r+d}\in I\\
f_{k[-1]}(\rsd')-f_{j[-2]}(\rsd') & r>0,j=\binom{\eta_{1}}{r},k=\binom{\eta_{1}}{r+d}\in I
\end{cases}.
\end{align*}

Recall that $\rsd_{s}=\seq_{[1,1[s-1]]}\rsd$ and $\rsd_{o_{+}(1)}=\rsd\{1\}=\rsd'$.
For the composition of tropical mutations $\phi_{\rsd_{0},\rsd_{o_{+}(1)}}=\phi_{\rsd_{0},\rsd_{1}}\circ\phi_{\rsd_{1},\rsd_{2}}\circ\cdots\circ\phi_{\rsd_{o_{+}(1)-1},\rsd_{o_{+}(1)}}$,
we have $\deg^{\rsd_{0}}W_{[j,k]}=\phi_{\rsd_{0},\rsd_{o_{+}(1)}}\deg^{\rsd_{o_{+}(1)}}W_{[j,k]}$.
If $\eta_{j}=\eta_{k}\neq\eta_{1}$, the vector $f_{k}-f_{j[-1]}$
remains unchanged and the desired claim for $W_{[j,k]}$ follows.

Next, compute $\deg^{\sd}(W_{[1[d],1[s]]})$, $\forall d\geq1$. Note
that $\rsd_{s+1}=\mu_{1[s]}\rsd_{s}$, $\forall s\in[0,o_{+}(1)-1]$.
Lemma \ref{lem:mutation_shuffle_seeds} implies $b_{1[s-1],1[s]}(\rsd_{s+1})=b_{1[s+1],1[s]}(\rsd_{s+1})=1$
and $b_{i,1[s]}(\rsd_{s+1})\leq0$ for $\eta_{i}\neq\eta_{1}$. So
we deduce
\begin{align*}
\phi_{\rsd_{s},\rsd_{s+1}}(f_{1[s]}(\rsd_{s+1})) & =f_{1[s+1]}(\rsd_{s})-f_{1[s]}(\rsd_{s})+f_{1[s-1]}(\rsd_{s-1}),\\
\phi_{\rsd_{s},\rsd_{s+1}}(\pm f_{1[d]}(\rsd_{s+1})) & =\pm f_{1[d]}(\rsd_{s}),\ \forall d\neq s.
\end{align*}
Then direct computation shows the following results for $0<d\leq s\leq o_{+}(1)$:
\begin{align*}
\deg^{\rsd_{s}}W_{[1[d],1[s]]} & =\phi_{\rsd_{s},\rsd_{s+1}}\cdots\phi_{\rsd_{o_{+}(1)-1},\rsd_{o_{+}(1)}}(\deg^{\rsd_{o_{+}(1)}}W_{[1[d],1[s]]})\\
 & =f_{1[s-1]}(\rsd_{s})-f_{1[d-2]}(\rsd_{s})\text{ (coordinates are unchanged)}\\
\deg^{\rsd_{d-1}}W_{[1[d],1[s]]} & =\phi_{\rsd_{d-1},\rsd_{d}}\cdots\phi_{\rsd_{s-1},\rsd_{s}}(\deg^{\rsd_{s}}W_{[1[d],1[s]]})\\
 & =f_{1[s]}(\rsd_{d-1})-f_{1[d-1]}(\rsd_{s-1})\\
\deg^{\rsd_{0}}W_{[1[d],1[s]]} & =\phi_{\rsd_{0},\rsd_{1}}\cdots\phi_{\rsd_{d-2},\rsd_{d-1}}(\deg^{\rsd_{d-1}}W_{[1[d],1[s]]})\\
 & =f_{1[s]}(\rsd_{0})-f_{1[d-1]}(\rsd_{0})\ \text{(coordinates are unchanged)}
\end{align*}

Finally, note that $W_{[1,1[s]]}$ are cluster variable of $\rsd$.
Thus $\deg^{\rsd}W_{[1,1[s]]}=f_{1[s]}=f_{1[s]}-f_{-\infty}$. Therefore,
we have verified the desired statement for all $W_{[j,k]}$.

\end{proof}

Recall that we use $[\ ]^{\rsd}$ to denote normalization in $\LP(\rsd)$,
abbreviated as $[\ ]$.

\begin{Prop}\label{prop:T-systems}

$\forall r\in[0,l]$, denote $a=\eta_{r+1}$ and $r_{b}=O([1,r];b)$
for $b\in J$. The interval variables satisfy the following equations,
called the $T$-systems:
\begin{align}
W_{[\binom{a}{r_{a}},\binom{a}{r_{a}+s}]}*W_{[\binom{a}{r_{a}+1},\binom{a}{r_{a}+s+1}]}= & q^{\alpha}[W_{[\binom{a}{r_{a}+1},\binom{a}{r_{a}+s}]}*W_{[\binom{a}{r_{a}},\binom{a}{r_{a}+s+1}]}]\label{eq:T-systems}\\
 & +q^{\alpha'}[\prod_{\binom{b}{r_{b}+d}}W_{[\binom{b}{r_{b}},\binom{b}{r_{b}+d}]}^{-C_{ba}}],\nonumber 
\end{align}
where $\binom{b}{r_{b}+d}$ appearing satisfy $b\neq a$ and $\binom{b}{r_{b}+d}<_{\ueta}\binom{a}{r_{a}+s+1}<_{\ueta}\binom{b}{r_{b}+d+1}$,
$\alpha=\Hf\lambda(\beta_{[\binom{a}{r_{a}},\binom{a}{r_{a}+s}]},\beta_{[\binom{a}{r_{a}+1},\binom{a}{r_{a}+s+1}]})$,
and $\alpha'=\Hf\lambda(\beta_{[\binom{a}{r_{a}},\binom{a}{r_{a}+s}]},-\sum_{\binom{b}{r_{b}+d}}C_{ba}\beta_{[\binom{b}{r_{b}},\binom{b}{r_{b}+d}]})$.
We have $\alpha>\alpha'$.

\end{Prop}

Equivalently, $\forall1\leq j\leq j[s]<l$, we can rewrite the $T$-systems
for $r=j-1$ as
\begin{align}
W_{[j,j[s]]}*W_{[j[1],j[s+1]]} & =q^{\alpha}[W_{[j[1],j[s]]}*W_{[j,j[s+1]]}]+q^{\alpha'}[\prod W_{[i,i[d]]}^{-C_{\eta_{i},\eta_{j}}}],\label{eq:T-system-other}
\end{align}
where $[i,i[d]]$ appearing satisfy $\eta_{i}\neq\eta_{j}$, $i=i^{\min}[O([1,j-1];\eta_{i})]$,
and $i[d]<j[s+1]<i[d+1]$.

\begin{proof}

Denote $\rsd'=\rsd\{r\}$ and $\rsd''=\rsd\{r+1\}$. We use $'$ and
$''$ denote their associated data respectively. Recall that $\mu_{\binom{a}{s}}\rsd\{r\}_{s}=\rsd\{r\}_{s+1}$
for $r'=r+1$, $s<o_{+}(r')$, $a=\eta_{r'}$, and $x_{\binom{b}{d}}'=W_{[\binom{b}{r_{b}},\binom{b}{r_{b}+d}]}$
when $b\neq a$, $\binom{b}{r_{b}+d}\in I$. Mutate $\rsd\{r\}_{s}$
at $\binom{a}{s}$. Then Lemma \ref{lem:mutation_shuffle_seeds} implies
\begin{align}
x_{\binom{a}{s}}'*x_{\binom{a}{s}}'' & =q^{\alpha}[x_{\binom{a}{s-1}}''*x_{\binom{a}{s+1}}']^{\rsd'}+q^{\alpha'}[\prod(x_{\binom{b}{d}}')^{-C_{ba}}]^{\rsd'},\label{eq:T-system-via-x}
\end{align}
such that $\binom{b}{d}$ appearing satisfy $\binom{b}{d+r_{b}}<_{\ueta}\binom{a}{s+r_{a}+1}<_{\ueta}\binom{b}{d+r_{b}+1}$,
and $\alpha,\alpha'$ satisfy $\alpha>\alpha'$, $\alpha=\Lambda'(\deg'x_{\binom{a}{s}}',\deg'x_{\binom{a}{s}}'')$,
$\alpha'=\Lambda'(\deg x_{\binom{a}{s}}',-\sum C_{ba}\deg'x_{\binom{b}{d}}')$.

Note that $\binom{b}{d+r_{b}},\binom{a}{s+r_{a}+1}\in I$. So the
cluster variables in (\ref{eq:T-system-via-x}) are the interval variables
in (\ref{eq:T-systems}). Note that normalization and the coefficients
$\alpha,\alpha'$ could be equally computed in the seed $\rsd$, since
the cluster variables involved belong to the same cluster (\cite{qin2017triangular}).
Then we obtain (\ref{eq:T-systems}) from (\ref{eq:T-system-via-x}).

\end{proof}

Let $\sigma$ denote the permutation on $I_{\ufv}$ such that $\sigma k^{\min}[r]=k^{\max}[-1-r]$.
Note that we have $x_{\sigma k}(\Sigma\rsd)=W_{k[1],k^{\max}}(\rsd)$
for $k\in I_{\ufv}$ by construction. Lemma \ref{lem:interval_degree}
implies the following result.

\begin{Prop}\label{prop:injective_degree}

We have $\deg^{\rsd}x_{\sigma k}(\Sigma\rsd)=f_{k^{\max}}-f_{k}$
for $k\in I_{\ufv}$.

\end{Prop}

Particularly, $\Sigma$ is a green to red sequence. We will denote
$\rsd[1]:=\Sigma\rsd=\rsd(-\ueta\op)$. Note that $x_{j^{\max}}(\rsd[1])=W_{j^{\min},j^{\max}}$
for any frozen $j^{\max}$. 

\begin{Lem}\label{lem:dominance_lex_order}

$\forall k\in I_{\ufv}$, we have
\begin{align*}
\deg y_{k}(\rsd) & =-\beta_{k}-\beta_{k[1]}-\sum_{j,j[d]}C_{\eta_{j}\eta_{k}}\beta_{[j,j[d]]}
\end{align*}
where $j,j[d]$ appearing satisfy $j[-1]<k<j\leq j[d]<k[1]<j[d+1]$.

\end{Lem}

\begin{proof}

We have $\deg y_{k}(\rsd)=f_{k[-1]}-f_{k[1]}+\sum_{j:\eta_{j}\neq\eta_{k}}b_{jk}f_{j}$,
and, by Lemma \ref{lem:interval_degree}, $f_{k[-1]}-f_{k[1]}=-\beta_{k}-\beta_{k[1]}$.
By (\ref{eq:dBS_B_matrix}), the other entries $b_{jk}$ are given
by $b_{j_{1}k}=C_{a,\eta_{k}}$for $j_{1}<k<j_{1}[1]<k[1]$ or $b_{j_{2}k}=-C_{a,\eta_{k}}$
for $k<j_{2}<k[1]<j_{2}[1]$, where $a=\eta_{j_{1}}=\eta_{j_{2}}\neq\eta_{k}$.
If $j_{1},j_{2}$ both exist, we have $j_{1}<k<j_{1}[1]<\cdots<j_{2}<k[1]<j_{2}[1]$.

Note that $k[1]<+\infty$. If $j_{2}$ does not exist, then $O([k,k[1]];a)=0$
and thus $j_{1}$ does not exist. Otherwise, define $j$ such that
$j[-1]<k<j\leq j_{2}$, where we denote $j[-1]=-\infty$ if $j=j^{\min}$.
Denote $j_{2}=j[d]$. Then $j[-1]<k<j\leq j[d]<k[1]<j[d+1]$. Moreover,
$j[-1]$ and $j_{2}$ contribute to $-C_{a,\eta_{k}}(f_{j[d]}-f_{j[-1]})$,
which equals $-C_{a,\eta_{k}}\beta_{[j,j[d]]}$ by Lemma \ref{lem:interval_degree}.
The desired claim follows.

\end{proof}

\begin{Eg}\label{eg:Kronecker-braid-deg}

Continue Example \ref{eg:Kronecker-braid-quiver}. Then $x_{i}(\rsd)=W_{[i^{\min},i]}$,
$\beta_{i}=\deg^{\rsd}W_{i,i}$. We can verify Lemma \ref{lem:dominance_lex_order}.
For example, $\deg y_{4}(\rsd)=f_{3}-f_{7}+2(f_{6}-f_{2})=-\beta_{[4,7]}+2\beta_{[5,6]}$.

Applying the mutation sequence $\Sigma=\seq_{(1,3,4,2,5,1,3,1,2)}$
to $\rsd$, we obtain the seed $\rsd[1]$. We obtain the following
new cluster variables along $\Sigma$ with distinct degrees:
\begin{center}
{\small$\begin{array}{|c|c|c|c|c|c|}
\hline  & W_{[3,3]} & W_{[3,4]} & W_{[3,7]} & W_{[5,5]} & W_{[5,6]}\\
\deg^{\rsd} & -f_{1}+f_{3} & -f_{1}+f_{4} & -f_{1}+f_{7} & -f_{2}+f_{5} & -f_{2}+f_{6}\\
\hline  & W_{[4,4]} & W_{[4,7]} & W_{[7,7]} & W_{[6,6]} & \\
\deg^{\rsd} & -f_{3}+f_{4} & -f_{3}+f_{7} & -f_{4}+f_{7} & -f_{5}+f_{6} & 
\\\hline \end{array}$}{\small\par}
\par\end{center}

\noindent We have $x_{\sigma k}(\rsd[1])=W_{[k[1],k^{\max}]}$, $\deg^{\rsd}x_{\sigma k}(\rsd[1])=-f_{k}+f_{k^{\max}}$,
$\forall k\in I_{\ufv}$, where $\sigma=(1,4)(2,5)$.

\end{Eg}

\subsection{Standard bases}

For any $w=\sum w_{i}\beta_{i}\in\oplus_{i\in I}\N\beta_{i}$, define
the $w$-pointed standard monomial $\stdMod(w):=[W_{1}^{w_{1}}*\cdots*W_{l}^{w_{l}}]^{\rsd}$.
Note that $\{\beta_{i}|i\in I\}$ is a $\Z$-basis of $\cone(\rsd)=\oplus_{i=1}^{l}\Z\beta_{i}$.
Let $<_{\lex}$ denote lexicographical order and $<_{\rev}$ the reverse
lexicographical order on $\Z^{[1,l]}\simeq\oplus_{i=1}^{l}\Z\beta_{i}$.
Working at the quantum level $\kk=\Z[q^{\pm\Hf}]$, we will prove
the follow result.

\begin{Thm}\label{thm:dBS_PBW}

(1) The set $\stdMod:=\{\stdMod(w)|w\in\N^{[1,l]}\}$ is a $\kk$-basis
of $\bUpClAlg(\rsd(\ueta))$. 

(2) It satisfies the analog of the Levendorskii-Soibelman straightening
law:
\begin{align}
W_{k}W_{j}-q^{\lambda(\deg W_{k},\deg W_{j})}W_{j}W_{k} & \in\sum_{w\in\N^{[j+1,k-1]}}\kk\stdMod(w),\ \forall j\leq k.\label{eq:LS-law}
\end{align}

(3) We have $\bClAlg(\rsd(\ueta))=\bUpClAlg(\rsd(\ueta))$.

(4) The common triangular basis element $\can_{\sum w_{i}\beta_{i}}$
for $\bUpClAlg(\rsd(\ueta))$ equals the Kazhdan-Lusztig basis element
$\can(w)$ associated with $\stdMod$ sorted by $<_{\rev}$, i.e.,
$\can(w)$ is bar-invariant and satisfies 
\begin{align*}
\can(w) & =\stdMod(w)+\sum_{w'<_{\rev}w}b_{w'}\stdMod(w'),\ b_{w'}\in q^{-\Hf}\Z[q^{\Hf}].
\end{align*}
This statement still holds if when we replace $<_{\rev}$ by $<_{\lex}$.

\end{Thm}

The $\kk$-basis $\stdMod$ will be called the standard basis. Since
$\rsd(\ueta)$ can be optimized by Lemma \ref{lem:optimized_dBS_last},
the set of dominant degrees $\domCone(\rsd(\ueta))$ is given by Lemma
\ref{lem:optimized-domCone}. 

\begin{Lem}\label{lem:dominance_cone}

The set of dominant degrees $\domCone(\rsd(\ueta))$ (Definition \ref{def:domCone})
coincides with $\oplus_{i\in[1,l]}\N\beta_{i}$. 

\end{Lem}

\begin{proof}

We give an indirect proof based on freezing operators. Note that $\beta_{i}$
are degrees of $W_{i,i}\in\bUpClAlg$. It remains to check the claim
$\domCone(\rsd(\ueta))\subset\oplus_{i\in[1,l]}\N\beta_{i}$.

(i) The claim is true when $\ueta=(a,a,\ldots,a)$. In this case,
$\bUpClAlg(\rsd)$ is isomorphic to the deformed Grothendieck ring
of the level $l-1$ monoidal category $\cC_{l-1}$ consisting of the
$U_{q}(\widehat{\mathfrak{sl}}_{2})$-modules, whose properties are
well-known, see \cite{HernandezLeclerc09}\cite{qin2017triangular}.

(ii) Take any $g\in\domCone(\rsd)$. Let $\can_{g}(\rsd)\in\bUpClAlg(\rsd)$
denote the common triangular basis element pointed at the tropical
point $[g]$.\footnote{Alternatively, we can work with the theta function instead of a triangular
basis element. By \cite{qin2023freezing}, the freezing operator sends
it to a theta function.} Denote $g=\sum c_{i}\beta_{i}$ for $i\in\Z$. We need to show that,
for any $k=k^{\min}$, we have $c_{k'}\geq0$ whenever $\eta_{k'}=\eta_{k}$. 

Denote $I':=\{k,k[1],\ldots,k^{\max}\}$. Then $I\backslash I'=\{i|\eta_{i}\neq\eta_{k}\}$.
Choose $F:=(I\backslash I')\cap I_{\ufv}$. We have $g'=\pr_{I'}g=\sum_{I\backslash I'}c_{i}\beta_{i}$
and $p=\pr_{\oplus_{i\in I\backslash I'}\Z f_{i}}g=\sum_{i\in I\backslash I'}c_{i}\beta_{i}$.
Moreover, we have $\frz_{F}\can_{g}(\rsd)=\can_{g'}(\frz_{F}\rsd)\cdot x^{p}$,
see Theorems \ref{thm:sub_cluster_triangular_basis} and \ref{thm:freezing_common_triangular}.
Take $\kk=\Z$ from now on for simplifying arguments.

We obtain the seed $\rsd':=\rsd((\eta_{k},\eta_{k[1]},\ldots,\eta_{k^{\max}}))$
by deleting the frozen vertices $I\cap I'$ from $\frz_{F}\rsd$.
The $g'$-pointed triangular basis element $\can_{g'}(\rsd')$ of
$\upClAlg(\rsd')$ is obtained from $\can_{g'}(\frz_{F}\rsd)$ by
evaluating $x_{i}$ to $1$ for $\eta_{i}\neq\eta_{k}$. Since $\can_{g}\in\bUpClAlg(\rsd)$,
$\can_{g'}(\frz_{F}\rsd)$ and $\can_{g'}(\rsd')$ are still regular
at $x_{k^{\max}}=0$. Then (i) implies that $g'\in\oplus_{r\in[0,o_{+}(1)]}\N\beta_{k[r]}$.

\end{proof}

By Lemma \ref{lem:dominance_lex_order}, $\prec_{\rsd}$ implies $<_{\rev}$
and $<_{\lex}$. Then $\prec_{\rsd}$ is bounded on $\domCone(\rsd(\ueta))$
by Lemma \ref{lem:dominance_cone}.

\begin{Lem}\label{lem:standard-basis}

$\stdMod$ is a $\kk$-basis of $\bUpClAlg(\rsd)$.

\end{Lem}

\begin{proof}

By Lemma \ref{lem:optimized_dBS_last} and Proposition \ref{prop:compactified_basis},
$\{\can_{m}|m\in\domCone(\rsd)\}$ is a $\domCone(\rsd)$-pointed
basis for $\bUpClAlg(\rsd)$. Then the desired statement follows from
the boundedness of $\prec_{\rsd}$, see Lemmas \ref{lem:dominance_cone}
and \ref{lem:bounded_basis}.

\end{proof}

For any $i\leq j\in[1,l]$, we define $\bClAlg_{[i,j]}$ to be the
$\kk$-subalgebra of $\bClAlg(t(\ueta))$ generated by $W_{i},W_{i+1},\ldots,W_{j}$.
We set $\bClAlg_{[i,j]}=\kk$ if $i>j$.

\begin{Lem}\label{lem:fundamental_commutator}

We have $W_{i}*W_{j}-q^{\lambda(\deg W_{j},\deg W_{i})}W_{j}*W_{i}\in\bClAlg_{[i+1,j-1]}$,
$\forall i\leq j$.

\end{Lem}

\begin{proof}

We give a proof by induction on $l$. The case $l=1$ is trivial.
Let us assume the desired statement has been verified for lengths
$\leq l-1$.

Use the natural inclusion $\bClAlg(\rsd(\ueta_{[1,l-1]}))\subset\bClAlg(\rsd(\ueta))$.
Then, $\forall i\in[1,l-1]$, the fundamental variables $W_{i}$ of
$\bClAlg(\rsd(\ueta_{[1,l-1]}))$ are identified with the fundamental
variables $W_{i}$ of $\bClAlg(\rsd(\ueta))$. Therefore, the induction
hypothesis implies the desired statement for $i,j\in[1,l-1]$.

Next, take the seed $\rsd\{1\}:=\Sigma_{1}\rsd=\rsd(\eta_{2},\ldots,\eta_{l},-\eta_{1})$.
Use the natural inclusion $\bClAlg(\rsd(\ueta_{[2,l]}))\subset\bClAlg(\rsd\{1\})$.
It follows from the construction of fundamental variables that, $\forall i\in[1,l-1]$,
the fundamental variables $W_{i}'$ of $\bClAlg(\rsd(\ueta_{[2,l]}))$
are identified with $W_{i+1}$ of $\bClAlg(\rsd)$. Therefore, the
induction hypothesis implies the desired statement for $i,j\in[2,l]$.

It remains to show $z:=W_{1}*W_{l}-q^{\lambda(\deg W_{l},\deg W_{1})}W_{l}*W_{1}$
belongs to $\bClAlg_{[2,l-1]}$. By Lemma \ref{lem:standard-basis},
we have a finite decomposition $z=\sum_{w\in\N^{[1,l]}}c_{w}M(w)$,
$c_{w}\in\kk$. Since $M(w)$ are pointed at distinct degrees, we
have $c_{w}\neq0$ only if $w\prec_{\rsd}\beta_{1}+\beta_{l}$ (see
\cite[Lemma 3.1.10(iii)]{qin2017triangular}). By Lemma \ref{lem:dominance_lex_order},
$w\prec_{\rsd}\beta_{1}+\beta_{l}$ implies $w<_{\rev}\beta_{1}+\beta_{l}$
and $w<_{\lex}\beta_{1}+\beta_{l}$. Moreover, any nonzero element
of $\oplus_{k\in I_{\ufv}}\N\deg^{\rsd}y_{k}$ takes the form $-b_{k_{1}}\beta_{k_{1}}+\sum_{k_{1}<j<k_{2}}b_{j}\beta_{j}-b_{k_{2}}\beta_{k_{2}}$
for $b_{k_{1}},b_{k_{2}}\in\N_{>0}$, $b_{j}\in\Z$.

Denote $w=\sum_{i}w_{i}\beta_{i}$, $w_{i}\in\N$. Then $w<_{\rev}\beta_{1}+\beta_{l}$
implies $w=\beta_{l}$ or $w_{l}=0$. The former case is impossible
since $w-(\beta_{1}+\beta_{l})=-\beta_{1}$ is not contained in $\oplus_{k}\N\deg^{\rsd}y_{k}$,
i.e. $w\nprec_{\rsd}\beta_{1}+\beta_{l}$. Similarly, $w<_{\lex}\beta_{1}+\beta_{l}$
implies $w=\beta_{1}$ or $w_{1}=0$. The former case is impossible
since $w-(\beta_{1}+\beta_{l})=-\beta_{l}$ is not contained in $\oplus_{k}\N\deg^{\rsd}y_{k}(\rsd)$.
We conclude $w\in\N^{[2,l-1]}$ as desired.

\end{proof}

\begin{proof}[Proof of Theorem \ref{thm:dBS_PBW}]

(1) and (2) were proved in Lemmas \ref{lem:standard-basis} and \ref{lem:fundamental_commutator}.
Since $\stdMod\subset\bClAlg(\rsd)$, Lemma \ref{lem:standard-basis}
implies (3). Finally, we claim that the Kazhdan-Lusztig basis associated
to the standard basis $\stdMod$, sorted by $<_{\rev}$, is the triangular
basis with respect to $\rsd$. This will imply (4). 

Let us give a proof of this claim based on induction of $l$. The
case $l=1$ is trivial. Assume the claim has been proved for $l-1$.

Use the natural inclusion $\bClAlg(\rsd')\subset\bClAlg(\rsd)$ where
$\rsd'=\rsd(\ueta_{[1,l-1]})$. Let $\can$ and $\can'$ denote the
(common) triangular basis for $\bClAlg(\rsd)$ and $\bClAlg(\rsd')$
respectively. Note that $\rsd'$ can be obtained from $\rsd$ by freezing
$l^{\max}[-1]$ and deleting $l^{\max}$. Take the freezing operator
$\frz:=\frz_{\{l^{\max}[-1]\}}$. By Theorem \ref{thm:sub_cluster_triangular_basis},
for any $w\in\oplus_{i}\N\beta_{i}$, $\frz\can(w)$ is the $w$-pointed
triangular basis element of $\upClAlg(\frz\rsd)$. By the similarity
between $\rsd'$ and $\frz\rsd$, we have $\frz\can(w)=\can'(u)\cdot x_{l}^{w_{l}}$,
where $\can'(u)$ is the $u$-pointed triangular basis for $\upClAlg(\rsd')$,
$u\in\oplus_{i=1}^{l-1}\Z\beta_{i}$, and $u+w_{l}\beta_{[l^{\min},l]}=w$.

We aim to show that $\forall w=\sum w_{i}\beta_{i}$ with $w_{i}\in\N$,
$w_{l}=0$, $\stdMod(w)$ is $(\prec_{\rsd},\mm)$-unitriangular to
$\can$. If so, any standard monomial $[\stdMod(w)*W_{l}^{w_{l}}]$
must be $(\prec_{\rsd},\mm)$-unitriangular to $\can$ as well, since
$W_{l}^{w_{l}}=x_{l^{\min}}(\rsd[1])$, see \cite[Lemma 6.2.4]{qin2017triangular}.

Denote the decomposition $\stdMod(w)=\can(w)+\sum_{w'\prec_{\rsd}w}c_{w'}\can(w')$
in $\bUpClAlg$. Then $w'$ appearing satisfy $w'<_{\rev}w$. It follows
that $w'\in\oplus_{i\in[1,l-1]}\N\beta_{i}$. Applying the freezing
operator $\frz:=\frz_{w}:=\frz_{\{l^{\max}[-1]\},w}$, we obtain $\frz\stdMod(w)=\frz\can(w)+\sum_{w'<_{\rev}w}c_{w'}\frz_{w}\can(w').$

Note that $w'$ appearing satisfy $w'-w=\sum_{k\in I_{\ufv}}n_{k}\deg y_{k}$,
$n_{k}\in\N$. In addition, $\deg y_{k}$ are given by Lemma \ref{lem:dominance_lex_order}.
Since $w_{l}=0$ and $w'_{l}\geq0$, we must have $n_{l^{\max}[-1]}=0$
and $w'_{l}=0$. Therefore, $\frz\can(w)=\can'(w)$ and $\frz_{w}\can(w')=\can'(w')$
by the above discussion of freezing triangular basis elements. In
addition, for $i\in[1,l-1]$, $\frz_{\beta_{i}}W_{i}$ is the $\beta_{i}$-pointed
localized cluster monomial of $\bClAlg(\rsd')$, which must be the
$\beta_{i}$-pointed fundamental variables $W'_{i}$ for $\bClAlg(\rsd')$
by Lemma \ref{lem:interval_degree}. So $\frz\stdMod(w)$ is a standard
monomial of $\bClAlg(\rsd')$. Then the induction hypothesis implies
that $c_{w'}\in\mm$ as desired.

\end{proof}

\begin{Rem}\label{rem:previous_dBS_is_CGL}

Work at the classical level $\kk=\C$. Then Theorem \ref{thm:dBS_PBW}(1)
was implied by \cite[Theorem 2.30]{shen2021cluster} \cite[Proposition B.6]{gao2020augmentations}.
Moreover, when $C$ is of finite type, Theorem \ref{thm:dBS_PBW}(2)
was obtained in \cite{elek2021bott}, and the same cluster algebra
on the Bott-Samelson cells was obtained via Poisson geometry \cite{elek2021bott}\cite{goodearl2018cluster},
which can be quantzied by \cite{mi2018quantization}\cite{GY13}.

\end{Rem}

\subsection{Categories associated to positive braids\label{subsec:Categories-associated-to-braid}}

Let $\ueta^{(1)}$, $\ueta^{(2)}$ denote any two words and $\ubi$
denote any shuffle of $\ueta^{(1)},-\ueta^{(2)}$. Recall that $\bUpClAlg(\rsd(\ubi))$
equals $\bUpClAlg(\rsd(\ueta))$, where $\ueta=((\ueta^{(2)})\op,\ueta^{(1)})$,
and it has the common triangular basis $\can$. Assume that the Cartan
matrix $C$ is of finite type. We will prove that $(\bUpClAlg(\rsd(\ubi)),\can)$
admits categorification in this case (Conjecture \ref{conj:dBS-categorification}).
We will also construct categorification for $\bUpClAlg(\dsd(\ubi))$.

Consider the following partial orders in analogous to the weak Bruhat
orders.

\begin{Def}

$\forall$ $\beta,\beta',\beta''\in\Br^{+}$, denote $\beta'\leq_{R}\beta$,
$\beta''\leq_{L}\beta$ if $\beta=\beta'\beta''$.

\end{Def}

Let $\ow_{0}$ denote any reduced word for $w_{0}$. The corresponding
positive braid $\Delta:=\beta_{\ow_{0}}\in\Br^{+}$ is known to be
independent of the choice of $\ow_{0}$. There exists an involution
$\nu$ on $J$ such that $\sigma_{a}\Delta=\Delta\sigma_{\nu(a)}$,
$\forall a$. We also have $\sigma_{a}\leq_{R}\Delta$.

\begin{Lem}\label{lem:initial_subword}

$\forall\beta\in\Br^{+}$, there exists $d\in\N$ such that $\beta\leq_{R}\Delta^{d}$,
$\beta\leq_{L}\Delta^{d}$.

\end{Lem}

\begin{proof}

We only prove $\beta\leq_{R}\Delta^{d}$ by induction on $l(\beta)$,
and $\beta\leq_{L}\Delta^{d}$ can be proved similarly. The case $l(\beta)=1$
is trivial. Decompose $\beta=\sigma_{a}\beta'$ for some $\beta'\in\Br^{+}$.
By the induction hypothesis, we have $\Delta^{d'}=\beta'\zeta'$ for
some $d'\in\N$. It follows that $\beta\leq_{R}\sigma_{a}\Delta^{d'}=\Delta^{d'}\sigma_{\nu^{d'}(a)}\leq_{R}\Delta^{d'+1}$.

\end{proof}

Choose $\beta:=\beta_{\ueta}$ associated with $\ueta$. Denote $l=l(\beta)$.
By Lemma \ref{lem:initial_subword}, we can choose a word $\uzeta$
such that $\Delta^{d}=\beta_{\uzeta}$, $\ueta=\uzeta_{[1,l]}$. Choose
any Coxeter word $\uc$ and any height function $\xi$ as in Section
\ref{subsec:Cluster-structures-quantum-affine}. Then we can find
an adapted word $\ugamma$ such that $\Delta^{d}=\beta_{\ugamma}$,
see the knitting algorithm \cite{Keller08Note}. Note that we have
$\rsd(\uzeta)=\seq^{\sigma}\rsd(\ugamma)$ where $\seq^{\sigma}$
is given in Section \ref{subsec:change-string-diagrams}. 

Recall that $\bUpClAlg(\rsd(\ugamma))$ is categorified by a monoidal
category $\cC_{\ugamma}(\xi)$. We have $\kappa:\bUpClAlg(\rsd(\ugamma))\simeq K$,
where $K$ is the deformed Grothendieck ring. See Section \ref{subsec:Cluster-structures-quantum-affine}. 

Then we have $\bClAlg(\rsd(\uzeta))\overset{}{=}\bClAlg(\rsd(\ugamma))\overset{\kappa}{\simeq}K$.
Let $S_{i}$ denote the simple module such that $[S_{i}]=\kappa(x_{i}(\rsd(\uzeta)))$,
and $L_{[j,k]}$ denote the simple module such that $[L_{[j,k]}]=\kappa(W_{[j,k]}(\rsd(\uzeta)))$,
called the interval modules. Denote $L_{k}:=L_{[k,k]}=\kappa(W_{k}(\rsd(\uzeta)))$.
$\forall w\in\N^{[1,dl(\Delta)]}$, define the generalized standard
modules $M(w):=L_{1}^{w_{1}}\otimes\cdots\otimes L_{dl(\Delta)}^{w_{dl(\Delta)}}$.

\begin{Def}

Let $\simpObj_{\beta}$ denote the set of the composition factors
of $M(w)$, $\forall w\in\N^{[1,l]}$. Define $\cC_{\beta}(\xi)$
as the full subcategory of $\cC_{\ugamma}(\xi)$ whose objects have
the composition factors in $\simpObj_{\beta}$.

\end{Def}

Denote $\rsd=\rsd(\uzeta)$, $\bClAlg=\bUpClAlg=\bUpClAlg(\rsd)$,
$\rsd'=\rsd(\ueta)$, $\bClAlg'=\bUpClAlg'=\bUpClAlg(\rsd')$. Let
$\stdMod$ and $\can$ (resp. $\stdMod'$ and $\can'$) denote the
standard basis and the common triangular basis for $\bUpClAlg$ (resp.
$\bUpClAlg'$). Since $\ueta=\uzeta_{[1,l]}$, $\rsd'$ is a good
sub seed of $\rsd$ via the cluster embedding $\iota:=\iota_{\uzeta,\ueta}$,
which sends $j\in[1,l]$ to $j$ (Lemma \ref{lem:calibration-word}).
So we have the inclusion $\bClAlg'\subset\bClAlg$, such that $W_{[j,k]}(\rsd')$
is identified with $W_{[j,k]}(\rsd)$ by Lemma \ref{lem:embed-interval-variables}.
It follows that $\stdMod'=\{\stdMod(w)|w\in\oplus_{j\in[1,l]}\N\beta_{j}\}$.
Then $\can'$ equals $\{\can(w)|w\in\oplus_{j\in[1,l]}\N\beta_{j}\}$
since it is the Kazhdan-Lusztig basis associated to $\stdMod'$ (Theorem
\ref{thm:dBS_PBW}).

\begin{Lem}\label{lem:property-cat-beta}

(1) We have $\simpObj_{\beta}=\{S\in\simpObj|[S]\in\kappa\can'\}$. 

(2) $\cC_{\beta}(\xi)$ is a monoidal category.

\end{Lem}

\begin{proof}

(1) Any $S\in\simpObj_{\beta}$ is a composition factor of some $M(w)$,
$w\in\oplus_{j=1}^{l}\N\beta_{j}$. Then $\kappa^{-1}[S]$ is some
common triangular basis element $\can(w')$ that appears in the decomposition
of $\stdMod(w)$ into $\can$. Since $w\in\oplus_{i=1}^{l}\N\beta_{j}$,
$\stdMod(w)$ lies in $\Span_{\kk}\can'$. Therefore, $\can(w')\in\can'$. 

Conversely, any $\can(w)\in\can'$ appears in the decomposition of
$M(w)$ into $\can$. So, when we denote $\kappa\can(w)=[S]$, $S$
is a composition factor of $M(w)\in\stdMod'$. Thus $S\in\simpObj_{\beta}$.

(2) Note that $\Span_{\kk}\can'$ is closed under multiplication since
it equals $\bUpClAlg'$. We deduce that $\cC_{\beta}$ is closed under
the tensor product.

\end{proof}

Lemma \ref{lem:property-cat-beta} has the following consequence.

\begin{Thm}\label{thm:ADE-braid-categorification}

When $C$ is of type $ADE$, $(\bUpClAlg(\rsd(\ueta)),\can')$ is
categorified by $\cC_{\beta}$. 

\end{Thm}

Let us give an alternative description of $\simpObj_{\beta}$. $\forall$
modules $V_{1},V_{2}$, we say $V_{2}|V_{1}$ if there exists a module
$V_{3}$ such that $V_{2}\otimes V_{3}\simeq V_{1}$. Associated with
$\uzeta$, we define $F:=\{j\in[1,l(\uzeta)]|j[1]>l\}$, $F'=[l+1,l(\uzeta)]$.
Then $\rsd(\ueta)$ can be obtained from $\frz_{F}\rsd(\uzeta)$ by
removing the non-essential frozen vertices $j\in F'$.

\begin{Prop}

We have the following description 
\[
\simpObj_{\beta}=\{S\in\simpObj|S\otimes S_{j}\text{ is simple \ensuremath{\forall j\in F}, and \ensuremath{S_{j}\nmid S} \ensuremath{\forall j\in F'}}\}.
\]

\end{Prop}

\begin{proof}

Denote the right hand side by $\simpObj'$. Then $\kappa\can'\subset\{[S]|S\in\simpObj'\}$. 

Conversely, if $S(w)\in\simpObj'$, then $\can(w):=\kappa^{-1}[S(w)]$
$q$-commutes with $x_{j}$, $\forall j\in F$, and $\can(w)\neq x_{j}\cdot\can(w')$,
$\forall j\in F'$. The first condition implies that $\can(w)$ is
a common triangular basis element for $\upClAlg(\frz_{F}\rsd)$, see
Theorem \ref{thm:sub_category_upClAlg}. By the similarity between
$\frz_{F}\rsd$ and $\rsd'$, we have $\can(w)=x^{m}\cdot\can(w')$
for some $\can(w')\in\can'$ such that $m\in\oplus_{j\in F'}\Z f_{j}$.
Moreover, since the vertices of $F'$ are non essential in $\frz_{F}\rsd$,
the order of vanishing $\nu_{j}(\can(w'))$ equals $0$, $\forall j\in F'$
(Section \ref{subsec:Optimized-seeds}). Then $m_{j}=\nu_{j}(\can(w))\geq0$,
$\forall j\in F'$, implies $m\geq0$. By the second condition, we
further have $m=0$. So $\can(w)=\can(w')\in\can'$ and thus $S(w)\in\simpObj_{\beta}$.

\end{proof}

\begin{Prop}

$\cC_{\beta}(\xi)$ does not depend on $\uzeta$ where $\beta_{\uzeta}=\Delta^{d}$,
$\beta_{\uzeta_{[1,l]}}=\beta$.

\end{Prop}

The subsequent paper \cite{qin2024infinite} will give an alternative
proof using braid groups.

\begin{proof}

Let $\ualpha$ be another choice for $\uzeta$. Since $\beta_{\ualpha{}_{[1,l]}}=\beta_{\uzeta_{[1,l]}}=\beta$
and $\beta_{\ualpha}=\beta_{\uzeta}=\Delta^{d}$, we have $\beta_{\ualpha{}_{[l+1,l(\uzeta)]}}=\beta_{\uzeta{}_{[l+1,l(\uzeta)]}}$.
Denote $\ualpha':=\ualpha_{[1,l]}$, $\ualpha'':=\ualpha_{[l+1,l(\uzeta)]}$,
$\uzeta'=\uzeta_{[1,l]}=\ueta$, $\uzeta''=\uzeta_{[l+1,l(\uzeta)]}$.
Using the permutation mutation sequences associated to braid moves
and the cluster embeddings in Section \ref{subsec:change-string-diagrams},
we obtain commutative diagrams:
\begin{align*}
\begin{array}{ccccc}
\bUpClAlg(\rsd(\ualpha')) & \subset & \bUpClAlg(\rsd(\ualpha',\uzeta'')) & \overset{}{\overset{(\seq_{\ualpha'',\uzeta''}^{\sigma})^{*}}{\xleftarrow{\sim}}} & \bUpClAlg(\rsd(\ualpha',\ualpha''))\\
\simeqd(\seq_{\ualpha',\uzeta'}^{\sigma})^{*} &  & \simeqd(\seq_{\ualpha',\uzeta'}^{\sigma})^{*} &  & \simeqd(\seq_{\ualpha',\uzeta'}^{\sigma})^{*}(\seq_{\ualpha'',\uzeta''}^{\sigma})^{*}\\
\bUpClAlg(\rsd(\uzeta')) & \subset & \bUpClAlg(\rsd(\uzeta',\uzeta'')) & = & \bUpClAlg(\rsd(\uzeta',\uzeta''))
\end{array}
\end{align*}
Then the common triangular basis $\can'$ for $\bUpClAlg(\rsd(\ualpha'))$
and $\bUpClAlg(\rsd(\uzeta'))$ are identified in $\bUpClAlg(\rsd(\uzeta',\uzeta''))$.
The claim follows from Lemma \ref{lem:property-cat-beta}(1).

\end{proof}

Finally, consider any shuffle $\ubi$ of $\ueta$ and $-\uzeta$.
Choose any Coxeter word $\uc'$. Define $\widetilde{\ubi}=(\uc',\ubi)$,
$\ugamma=(\uzeta\op,\uc',\ueta)$, and $\widetilde{\beta}=\beta_{\ugamma}$.
Then $\rsd(\widetilde{\ubi})=\seq\rsd(\ugamma)$ for some mutation
sequence $\seq$. Note that $\dsd(\ubi)=\frz_{F}\rsd(\widetilde{\ubi})$,
where $F=\{\binom{a}{0}^{\widetilde{\ubi}}|a\in J\}$. By Theorem
\ref{thm:ADE-braid-categorification}, $\bUpClAlg(\rsd(\widetilde{\ubi}))=\bUpClAlg(\rsd(\ugamma))$
is categorified by $\cC_{\widetilde{\beta}}$. Choose simples $S_{k}$
of $\cC_{\widetilde{\beta}}$ such that $[S_{k}]\sim x_{k}(\rsd(\widetilde{\ubi}))$
for $k\in F$. Let $\cC_{\widetilde{\beta}}'$ denote the monoidal
subcategory of $\cC_{\widetilde{\beta}}$ whose objects have the composition
factors in $\{S\in\cC_{\widetilde{\beta}}|S\otimes S_{k}\text{ is simple},\forall k\in F\}$,
see Section \ref{subsec:Monoidal-subcategories}. By Theorem \ref{thm:bases-compactified-dBS},
$\bUpClAlg(\dsd(\ubi))$ has the common triangular basis, denoted
$\can$. Combining Theorem \ref{thm:ADE-braid-categorification} ,
Proposition \ref{prop:optimize_double_sd} and Corollary \ref{cor:subcategorify-combine},
we deduce the following result.

\begin{Thm}\label{thm:ADE-braid-categorification-ddBS}

$(\bUpClAlg(\dsd(\ubi)),\can)$ is categorified by $\cC_{\widetilde{\beta}}'$
when $C$ is of type $ADE$.

\end{Thm}

\begin{proof}[Proof of Theorem \ref{thm:intro-qO-categorification}]

We could obtain the categorification of quantum function algebras
$\qO[G]$ from those of decorated double Bott-Samelson cells $\bUpClAlg(\dsd(\ubi))$.

More precisely, for any reduced word $\ueta$ of $w_{0}$, we have
the anti-isomorphism $\bUpClAlg(\dsd(\ueta,-\ueta))\otimes\Q(q^{\Hf})\simeq\qO[G]\otimes\Q(q^{\Hf})$
by Claim \ref{claim:G_case}. Extend the coefficients ring of the
deformed Grothendieck rings from $\Z[q^{\pm\Hf}]$ to $Q(q^{\Hf})$.
Then, $\bUpClAlg(\dsd(\ueta,-\ueta))\otimes\Q(q^{\Hf})$ is categorified
by the monoidal category $\cC_{\widetilde{\beta}}'$, where $\widetilde{\beta}=\beta_{(\ueta\op,c',\ueta)}$
and $c'$ is any Coxeter word, see Theorem \ref{thm:ADE-braid-categorification-ddBS}.
It follows that $\qO[G]\otimes\Q(q^{\Hf})$ is categorified by the
opposite monoidal category $(\cC_{\widetilde{\beta}}')\op$ (Section
\ref{subsec:Monoidal-categories}).

\end{proof}

\appendix

\section{Principal coefficients\label{subsec:From-principal-coefficients}}

Given any seed $\sd$. Let $I_{\fv}^{\prin}:=\{k'|k\in I_{\ufv}\}$
denote a copy of $I_{\ufv}$. We recall the seed of principal coefficients
$\sd^{\prin}$ \cite{FominZelevinsky07}. Its set of vertices is $I^{\prin}:=I_{\ufv}\sqcup I_{\fv}^{\prin}$
with $I_{\fv}^{\prin}$ being frozen. The symmetrizers are $d_{k}^{\prin}=d_{k'}^{\prin}=d_{k}$
for $k\in I_{\ufv}$. Introduce the matrix entries $b_{ij}^{\prin}:=\begin{cases}
b_{ij} & i,j\in I_{\ufv}\\
\delta_{kj} & j,k\in I_{\ufv},i=k'\\
-\delta_{ik} & i,k\in I_{\ufv},j=k'\\
0 & i,j\in I_{\fv}^{\prin}
\end{cases}.$ Then its $B$-matrix is $\tB^{\prin}:=(b_{ij}^{\prin})_{i\in I^{\prin},j\in I_{\ufv}}$.
Let $\var$ denote the linear map from $\Z^{I^{\prin}}=\oplus_{i\in I^{\prin}}\Z f_{i}$
to $\Z^{I}=\oplus_{i\in I}\Z f_{i}$ such that $\var(f_{i})=\begin{cases}
f_{i} & i\in I_{\ufv}\\
\sum_{j\in I_{\fv}}b_{jk}f_{j} & i=k',k\in I_{\ufv}
\end{cases}$. Then we have $\var(\sum_{j\in I^{\prin}}b_{jk}^{\prin}f_{j})=\sum_{i\in I}b_{ik}f_{i}$,
$\forall k\in I_{\ufv}$.

When $\sd$ is a quantum seed endowed with a compatible skew-symmetric
bilinear form $\lambda$, we define the skew-symmetric bilinear form
$\lambda^{\prin}$ on $\Z^{I^{\prin}}$ by $\lambda^{\prin}(m,m')=\lambda(\var m,\var m')$.
We deduce the following result.

\begin{Lem}

The bilinear form $\lambda(t^{\prin})$ is compatible with $\tB(t^{\prin})$.
Moreover, the quantum seed $t^{\prin}$ is similar to $t$ with the
scalar $\rho=1$.

\end{Lem}

The linear map $\var$ induces a monomial $\kk$-algebra homomorphism
$\var:\LP(\sd^{\prin})\rightarrow\LP(\sd)$, sending $x^{m}(\sd^{\prin})$
to $x^{\var(m)}(\sd)$. It sends an $m$-pointed element in $\LP(\sd^{\prin})$
to a $\var(m)$-pointed similar element in $\LP(\sd)$. It is a (quantum)
variation map in the sense of \cite{kimura2022twist}.

%\bibliographystyle{../amsalphaURL} 
%\bibliography{../referenceEprint} 
%\end{document}

\section{Freezing operators in additive categorifications\label{sec:Freezing-operators-in-additive}}

We refer the reader to \cite[Section 2.3]{Qin10} for necessary notions
of additive categorification of cluster algebras following \cite{Plamondon10b}.

Assume $\kk=\Z$ and $B$ is skew-symmetric. We associate with the
initial seed $\sd_{0}$ the generalized cluster category $\cC$. For
simplicity, we assume $\cC$ is Hom-finite to avoid discussing subcategories
of $\cC$ as in \cite{Plamondon10b}. Let $T=\oplus_{i\in I}T_{i}$
denote the cluster tilting object associated to $\sd_{0}$. Let $[1]$
denote the shift functor of $\cC$. Note that $\Hom(T,T[1])=0$. Denote
$T^{m}=\oplus T_{i}^{m_{i}}$ for $m=(m_{i})_{i\in I}\in\N^{I}$.
The corresponding completed Jacobian algebra $J$ equals $\End_{\cC}(T)\op$.

For any object $V$ of $\cC$, let $\uV$ denote the (left) $J$-module
$\Hom_{\cC}(T,V[1])$. Further assume $(\udim\uV)_{j}=0$ for $j\in I_{\fv}$.
There exists a triangle $V\rightarrow T^{m_{-}}[1]\rightarrow T^{m_{+}}[1]\rightarrow V[1]$.
Define the index $\ind_{T}V=m_{+}-m_{-}$ (it equals the negative
coindex $-\mathrm{coind}V$ associated with $T[-1]$ in \cite{Palu08a}).
Then the cluster character of $V$ is given by
\begin{align*}
\CC(V)= & x^{\ind_{T}V}\cdot(\sum_{n\in\N^{I_{\ufv}}}\chi(\Gr_{n}\uV)\cdot y^{n}),
\end{align*}
where $\chi$ denote the Euler characteristics, and $\Gr_{\un}\uV$
the quiver Grassmannian consisting of the $n$-dimensional submodules
of $\uV$. 

Choose any $F\subset I_{\ufv}$ and denote $T_{F}=\oplus_{k\in F}T_{k}$.
Choose any $d\geq\dim\Hom(V,T_{F}[1])$ and let $V'[1]$ be the cone
of a generic map $\ev$ in $\Hom(V,T_{F}^{\oplus d}[1])$. We then
have a triangle $T_{F}^{\oplus d}\rightarrow V'\rightarrow V\xrightarrow{\ev}T_{F}^{d}[1]$
and view $V'$ as the universal extension of $V$ and $T_{F}^{\oplus d}$.
Denote $\uV'=\Hom(T,V'[1])$.

\begin{Prop}\label{prop:max_submod}

$\uV'$ is the maximal $J$-submodule of $\uV$ such that $(\udim\uV')_{k}=0$
for $k\in F$.

\end{Prop} 

\begin{proof}

(1) We first show that $\uV'$ is a submodule of $\uV$ satisfying
$(\udim\uV')_{k}=0$ for $k\in F$.

Applying $\Hom(T,(\ )[1])$ to the triangle $T_{F}^{\oplus d}\rightarrow V'\rightarrow V\xrightarrow{\ev}T_{F}^{\oplus d}[1]$,
we obtain an exact sequence $0\rightarrow\Hom(T,V'[1])\rightarrow\Hom(T,V[1])$.
It follows that $\uV'$ is a submodule of $\uV$.

Applying $\Hom(\ ,T_{F}[1])$ to the triangle $T_{F}^{\oplus d}\rightarrow V'\rightarrow V\xrightarrow{\ev}T_{F}^{\oplus d}[1]$,
we obtain an exact sequence $\Hom(T_{F}^{\oplus d}[1],T_{F}[1])\xrightarrow{\ev^{*}}\Hom(V,T_{F}[1])\rightarrow\Hom(V',T_{F}[1])\rightarrow0$.
Since $\ev$ is generic and $d\geq\dim\Hom(V,T_{F}[1])$, $\ev^{*}$
is surjective. It follows that $\Hom(V',T_{F}[1])=0$. Since $\cC$
is $2$-Calabi-Yau \cite{Amiot09}, we obtain $\Hom(T_{F},V'[1])=0$.

(2) Let $\uX$ denote any submodule of $\uV$ satisfying $(\udim\uX)_{k}=0$
for $k\in F$. Then there exists a triangle $Y\xrightarrow{-r[-1]}X\rightarrow V\xrightarrow{u}Y[1]$
such that $\Hom(T,X[1])=\uX$ and $-r[-1]$ factors through $\add T$
(\cite[Lemma 3.1 Proposition 1.1]{Palu08a}). Applying $\Hom(\ ,T_{F}^{\oplus d}[1])$
to it, we obtain $\Hom(Y[1],T_{F}^{\oplus d}[1])\xrightarrow{u^{*}}\Hom(V,T_{F}^{\oplus d}[1])\rightarrow\Hom(X,T_{F}^{\oplus d}[1])=0$.
Then $\ev\in\Hom(V,T_{F}^{\oplus d}[1])$ equals the composition $v\circ u$
for some $v\in\Hom(Y[1],T_{F}^{\oplus d}[1])$.

Applying the octahedral axiom to the triangles $V\xrightarrow{\ev}T_{F}^{\oplus d}[1]\rightarrow V'[1]\rightarrow V[1]$,
$V\xrightarrow{u}Y[1]\xrightarrow{r}X[1]\rightarrow V[1]$, and $Y[1]\xrightarrow{v}T_{F}^{\oplus d}[1]\xrightarrow{}Z\xrightarrow{s}Y[2]$,
we obtain a triangle $X[1]\rightarrow V'[1]\rightarrow Z\xrightarrow{h}X[2]$
such that $h=r[1]\circ s$, see Figure \ref{fig:octahedral}(A). Since
$r[-1]$ factors through $\add T$, $r[1]$ factors through $\add T[2]$.
Thus $\Hom(T[1],h)$ vanishes. Applying $\Hom(T[1],\ )$ to $Z\xrightarrow{h}X[2]\rightarrow V'[2]\rightarrow Z[1]$,
we obtain that $\uX$ is a submodule of $\uV'$.

\end{proof}

\begin{figure}
\caption{Octahedral axiom}
\label{fig:octahedral}

\subfloat[]{

\begin{tikzpicture}
\node (v1) at (-1,2.5) {$V$}; \node (v2) at (2,2.5) {$T_F^{\oplus d}[1]$}; \node (v3) at (0.5,1.5) {$Y[1]$}; \node (v4) at (0.5,4.5) {$V'[1]$}; \node (v5) at (-1,3.5) {$X[1]$}; \node (v6) at (2,3.5) {$Z$}; \draw[->]  (v1) edge node[above]{$\mathrm{ev}$} (v2);
\draw[->]  (v1) edge node[below]{$u$}(v3);
\draw[->]  (v3) edge node[below]{$v$}(v2);
\draw[->]  (v2) edge (v4);
\draw[->,dashed]  (v4) edge (v1); \draw[->,dashed]  (v5) edge (v1); \draw[->]  (v3) edge node[right, near start]{$r$}(v5);
\draw[->]  (v2) edge (v6); \draw[->,dashed]  (v6) edge node[left, near end]{s}(v3); \draw[->]  (v5) edge (v4); \draw[->]  (v4) edge (v6); \draw[->,dashed]  (v6) edge node[below]{$h$} (v5); \end{tikzpicture}}$\quad$\subfloat[]{

\begin{tikzpicture}[scale=2.5]
\node (v1) at (-0.5,3) {$T_F^{\oplus d}$}; \node (v2) at (1.5,2.5) {$T^{m_+'}[1]$}; \node (v3) at (1.5,3) {$T^{m_-'}[1]\oplus T_F^d[1]$}; \node (v4) at (1.5,4) {$V$}; \node (v5) at (0,3.25) {$V'$}; \node (v6) at (0.5,3) {$T^{m_-'}[1]$};
\node (v8) at (2.2004,3) {}; \node (v7) at (2,4.25) {}; \node (v9) at (2,2.25) {}; \node (v10) at (1.5,2.1503) {};
\draw[->]  (v1) edge node[above]{$u$} (v5); \draw[->]  (v5) edge (v4); \draw[->]  (v1) edge node[below]{$0$} (v6); \draw[->]  (v6) edge (v3); \draw[->]  (v5) edge node[above, near end]{$v$} (v6); \draw[->]  (v6) edge (v2); \draw[->]  (v4) edge (v3); \draw[->]  (v3) edge (v2); \draw[->]  (v4) edge (v7); \draw[->]  (v3) edge (v8); \draw[->]  (v2) edge (v9); \draw[->]  (v2) edge (v10); \end{tikzpicture}}
\end{figure}

\begin{Lem}\label{lem:ind-submod}

We have $\ind V=\ind V'-d\sum_{k\in F}f_{k}$.

\end{Lem}

\begin{proof}

Recall that we have a triangle $T_{F}^{d}\xrightarrow{u}V'\rightarrow V\rightarrow T_{F}^{d}[1]$.
Choose triangles $V'\xrightarrow{v}T^{m'_{-}}[1]\rightarrow T^{m'_{+}}[1]\rightarrow V'[1]$
and $T_{F}^{d}\xrightarrow{0}T^{m'_{-}}[1]\rightarrow T^{m'_{-}}[1]\oplus T_{F}^{d}[1]\rightarrow T_{F}^{d}[1]$.
Note that $vu\in\Hom(T,T[1])=0$. Using the octahedral axiom, we obtain
a triangle $V\rightarrow T^{m'_{-}}[1]\oplus T_{F}^{d}[1]\rightarrow T^{m'_{+}}[1]\rightarrow V[1]$,
see Figure \ref{fig:octahedral}(B). The desired claim follows.

\end{proof}

Proposition \ref{prop:max_submod} and Lemma \ref{lem:ind-submod}
imply the following.

\begin{Thm}\label{thm:additive-freezing}

The freezing operator $\frz_{F}$ acts on $\CC(V)$ by
\begin{align}
\frz_{F}\CC(V) & =x^{-d\sum_{k\in F}f_{k}}\CC(V')\label{eq:additive-freezing}
\end{align}
for $d\in\N$ large enough. In this case, $\uV'$ is the maximal $J$-submodule
of $\uV$ such that $(\udim\uV')_{k}=0$ for $k\in F$.

\end{Thm}

Let $\cC'$ denote the subcategory of $\cC$ such that it consists
of objects $X$ satisfying $\Hom(T_{F},X[1])=0$. Then $V'\in\cC'$
by Proposition \ref{prop:max_submod}. This category provides additive
categorification for $\upClAlg(\frz_{F}\sd)$.

%%%%%%%%%%%%%%%%%%%%%%%%%%%%                      Bibliography%%%%%%%%%%%%%%%%%%%%%%%%%%

%\bibliographystyle{amsalphaURL} 
%\bibliography{referenceEprint} 
\input{analogs_dual_canonical.bbl}

%\bibliographystyle{../amsalphaURL}
%\bibliography{../referenceEprint}

\end{document}

%% file: LyxMacro.tex
\usepackage{amsmath}
\usepackage{amssymb}
\usepackage{stmaryrd}%for \longmapsfrom
\usepackage{amsthm}
\usepackage[mathscr]{eucal}
\usepackage{amscd}
\usepackage[all]{xy}
\usepackage{url}
\usepackage{enumitem}

\usepackage{mathtools}
\newtheorem{Thm}{Theorem}[section]
\usepackage{aliascnt}

\usepackage{comment}%\end{comment} must be used in the starting of the
\usepackage{hyperref}
\hypersetup{%
pdftitle={Preprint},
pdfauthor={},
pdfkeywords={},
bookmarksnumbered,
pdfstartview={FitH},
breaklinks=true,
colorlinks=true,%to avoid color option clash
urlcolor=blue,
citecolor=blue,
}%
\usepackage[all]{hypcap} %to fix fraction numbers showing problems in package {hyperref}

\usepackage{breakurl}%If there are problems in the references (which might occur if use dvi->ps->pdf), use the following package
\usepackage{color}

\theoremstyle{plain}
\newtheorem{Lem}[Thm]{Lemma}
\theoremstyle{plain}
\newtheorem{Prop}[Thm]{Proposition}
\theoremstyle{plain}
\newtheorem{Cor}[Thm]{Corollary}
\theoremstyle{plain}
\newtheorem{Conj}[Thm]{Conjecture}
\theoremstyle{plain}
\newtheorem*{Thm*}{Theorem}
\theoremstyle{plain}
\newtheorem*{Conj*}{Conjecture}
\theoremstyle{plain}

\theoremstyle{plain}
\newtheorem*{Prob*}{Problem}
\theoremstyle{plain}

\newtheorem{Claim}[Thm]{Claim}
\theoremstyle{plain}

\theoremstyle{definition}
\newtheorem{Def}[Thm]{Definition}
\theoremstyle{definition}
\newtheorem*{Def*}{Definition}
\theoremstyle{definition}
\newtheorem{Eg}[Thm]{Example}
\theoremstyle{definition}

\theoremstyle{definition}
\newtheorem{Assumption}{Assumption}
\theoremstyle{definition}

\theoremstyle{definition}

\theoremstyle{definition}

\theoremstyle{remark}
\newtheorem{Rem}[Thm]{Remark}
\newtheoremstyle{custom}
  {}% Space above
  {}% Space below
  {\itshape}% Body font
  {}% Indent amount
  {\bfseries}% Theorem head font
  {}% Punctuation after theorem head
  { }% Space after theorem head
  {Condition \thmnote{#3}}% Theorem head spec

\theoremstyle{custom}
\newtheorem*{Cond*}{}

\newcommand{\kk}{\Bbbk}
\newcommand{\Z}{\mathbb{Z}}
\newcommand{\N}{\mathbb{N}}
\newcommand{\Q}{\mathbb{Q}}
\newcommand{\C}{\mathbb{C}}
\newcommand{\R}{\mathbb{R}}
\newcommand{\F}{\mathbb{F}}

\renewcommand{\hat}[1]{\widehat{#1}}
\renewcommand{\tilde}[1]{\widetilde{#1}}

\newcommand{\opname}[1]{\operatorname{\mathsf{#1}}}
\newcommand{\Spec}{\operatorname{\mathsf{Spec}}}

\newcommand{\pr}{\opname{pr}}

\newcommand{\ind}{\opname{ind}}

\newcommand{\inj}{\opname{inj}}
\newcommand{\add}{\opname{add}}

\newcommand{\op}{{^{\opname{op}}}}

\newcommand{\ra}{\rightarrow}

\newcommand{\Gr}{\opname{Gr}}

\newcommand{\sign}{\opname{sign}}

\newcommand{\End}{\opname{End}}

\newcommand{\Hom}{\opname{Hom}}

\newcommand{\ext}{\opname{ext}}

\newcommand{\supp}{\opname{supp}}
\renewcommand{\deg}{\opname{deg}}

\newcommand{\Hf}{{\frac{1}{2}}}

\newcommand{\Rm}[1]{{\longmapsto}}
\newcommand{\Lm}[1]{{\longmapsfrom}}

\newcommand{\cA}{{\mathcal A}}

\newcommand{\cC}{{\mathcal C}}

\newcommand{\cF}{{\mathcal F}}

\newcommand{\cS}{{\mathcal S}}
\newcommand{\cT}{{\mathcal T}}

\newcommand{\cZ}{{\mathcal Z}}

\newcommand{\bL}{{\mathbf L}}

\newcommand{\bi}{{\mathbf i}}
\newcommand{\bj}{{\mathbf j}}

\newcommand{\bq}{{\mathbf q}}

\newcommand{\uV}{{\underline{V}}}

\newcommand{\uX}{{\underline{X}}}

\newcommand{\uc}{{\underline{c}}}

\newcommand{\ui}{{\underline{i}}}

\newcommand{\uk}{{\underline{k}}}

\newcommand{\um}{{\underline{m}}}
\newcommand{\un}{{\underline{n}}}

\newcommand{\uz}{{\underline{z}}}

\newcommand{\tB}{{\widetilde{B}}}
\newcommand{\tC}{{\widetilde{C}}}

\newcommand{\tE}{{\widetilde{E}}}

\newcommand{\tJ}{{\widetilde{J}}}

\newcommand{\tN}{{\widetilde{N}}}

\newcommand{\tP}{{\widetilde{P}}}

\newcommand{\tW}{{\widetilde{W}}}

\newcommand{\hV}{{\widehat{V}}}

\newcommand{\hk}{{\widehat{k}}}

\newcommand{\stdMod}{\mbf{M}}%standard module, dual pbw basis
\newcommand{\can}{L}%simple module, dual canonical basis
\newcommand{\clAlg}{{\cA}}%cluster algebra
\newcommand{\qClAlg}{\cA_q}

\newcommand{\diag}{{\delta}}
\newcommand{\rev}{\opname{rev}}%reverse

\newcommand{\res}{{\mathrm{Res}}}

\newcommand{\tw}{{\tilde{w}}}

\newcommand{\mbf}[1]{{\mathbf{#1}}}

\usepackage{tikz}
\usetikzlibrary{positioning,shapes,shadows,arrows,snakes}
\usetikzlibrary {arrows.meta,positioning}

\pgfdeclarelayer{edgelayer}
\pgfdeclarelayer{nodelayer}
\pgfsetlayers{edgelayer,nodelayer,main}

\tikzstyle{none}=[inner sep=0pt]
\tikzstyle{black box}=[draw=black, fill=black!25]
\tikzstyle{white box}=[draw=black, fill=white]
\tikzstyle{black circle}=[circle,draw=black!50, fill=black!25]
\tikzstyle{red circle}=[circle,draw=red!50, fill=red!25]
\tikzstyle{blue circle}=[circle,draw=blue!50, fill=blue!25]
\tikzstyle{green circle}=[circle,draw=green!50, fill=green!25]
\tikzstyle{yellow circle}=[circle,draw=yellow!50, fill=yellow!25]

\newcommand{\thistheoremname}{}
\newtheorem*{genericthm*}{\thistheoremname}
\newenvironment{namedthm*}[1]
  {\renewcommand{\thistheoremname}{#1}%
   \begin{genericthm*}}
  {\end{genericthm*}}

\renewcommand{\diag}{{d}}

\renewcommand{\inj}{{\bI}}
\renewcommand{\can}{{\bL}}

\newcommand{\var}{\opname{var}}

\newcommand{\fv}{\opname{f}}
\newcommand{\ufv}{\opname{uf}}
\newcommand{\codeg}{\opname{codeg}}
\newcommand{\suppDim}{\opname{suppDim}}

%% file: analogs_dual_canonical.bbl
\newcommand{\etalchar}[1]{$^{#1}$}
\def\cprime{$'$}
\providecommand{\bysame}{\leavevmode\hbox to3em{\hrulefill}\thinspace}
\providecommand{\MR}{\relax\ifhmode\unskip\space\fi MR }
\providecommand{\MRhref}[2]{%
  \href{http://www.ams.org/mathscinet-getitem?mr=#1}{#2}
}
\providecommand{\href}[2]{#2}